\documentclass[10pt,a4paper]{article}
\usepackage[utf8]{inputenc}
\usepackage[english]{babel}
\usepackage[left=2.5cm, top=3cm, right=2.5cm, bottom=3cm]{geometry}                
\usepackage{graphicx}
\usepackage{amssymb, mathtools}
\usepackage{multicol}
\usepackage{epstopdf}
\usepackage{times}
\usepackage{xcolor}
\usepackage{graphicx, float, subfig, overpic}
\usepackage{times}
\usepackage[pdftex]{hyperref}
\usepackage{amsthm,amsfonts,euscript,
mathrsfs,graphics,color,amsmath,latexsym,marginnote}
\usepackage{dsfont} 
\usepackage[nottoc,numbib]{tocbibind}

\numberwithin{equation}{section}

\newtheorem{theorem}{Theorem}[section]
\newtheorem{corollary}[theorem]{Corollary}

\newtheorem{lem}[theorem]{Lemma}
\newtheorem{proposition}[theorem]{Proposition}

\theoremstyle{definition}
\newtheorem{definition}[theorem]{Definition}
\newtheorem{remark}[theorem]{Remark}

\newcommand{\tz}{\tilde{z}}
\newcommand{\tg}{\widetilde{\gamma}}

\newcommand{\e}{\varepsilon}
\newcommand{\dom}{B_{\delta}(\ell^{\infty})\times \mathbb{T}^d}
\newcommand{\T}{\mathbb{T}}
\newcommand{\Z}{\mathbb{Z}}
\newcommand{\PP}{\mathcal{P}}
\newcommand{\R}{\mathbb{R}}
\newcommand{\ZZ}{\mathbb{Z}}
\newcommand{\NN}{\mathbb{N}}
\newcommand{\TT}{\mathbb{T}}

\newcommand{\gr}{\mbox{graph}}
\newcommand{\HH}{\mathcal{H}}
\newcommand{\lip}{\mathrm{Lip}}
\newcommand{\infd}{\ell^{\infty}\times\mathbb{R}^d}
\newcommand{\infn}{\ell^{\infty}\times\mathbb{T}^d}
\newcommand{\tM}{\widetilde{M}}
\newcommand{\eps}{\varepsilon}

\newcommand{\cO}{\mathcal{O}}
\newcommand{\cA}{\mathcal{A}}
\newcommand{\cB}{\mathcal{B}}
\newcommand{\cF}{\mathcal{F}}
\newcommand{\cM}{\mathcal{M}}
\newcommand{\cN}{\mathcal{N}}
\newcommand{\cV}{\mathcal{V}}
\newcommand{\cL}{\mathcal{L}}
\newcommand{\tC}{\mathtt{C}}
\newcommand{\wkappa}{\widetilde{\kappa}}
\newcommand{\wrho}{\widetilde{\rho}}
\newcommand{\tv}{\mathtt{v}}
\newcommand{\tu}{\mathtt{u}}
\newcommand{\pa}{\partial}
\newcommand{\rr}{\rho}
\newcommand{\La}{\Lambda}
\newcommand{\de}{\delta}
\newcommand{\wt}{\widetilde}
\newcommand{\Ker}{\mathrm{Ker}}
\newcommand{\tB}{\mathtt{B}}
\newcommand{\id}{\mathrm{Id}}
\newcommand{\K}{K}
\newcommand{\C}{\mathcal{C}}

\newcommand{\XX}{\mathcal{X}}
\newcommand{\cX}{\mathcal{X}}
\newcommand{\cI}{\mathcal{I}}
\newcommand{\cE}{\mathcal{E}}
\newcommand{\cH}{\mathcal{H}}

\title{Arnold diffusion in Hamiltonian systems on infinite lattices}

\author{Filippo Giuliani$^{1}$, Marcel Guardia$^{2, 3, 4}$ \thanks{The authors are supported by the European Research Council (ERC) 
under the European Union's Horizon 2020
research and innovation programme under grant agreement 
No 757802.}  \\
\small${}^{1}$ Dipartimento di Matematica, Politecnico di Milano, Milano, Italy.\\
\small${}^{2}$ Departament de Matem\`atiques, Universitat Polit\`ecnica de Catalunya (UPC), Barcelona, Spain.\\
\small${}^{3}$ IMTECH, Universitat Polit\`ecnica de Catalunya (UPC), Barcelona, Spain.\\
\small${}^{4}$ Centre de Recerca Matem\`atica, Barcelona, Spain.
}

\begin{document}

\maketitle

\begin{abstract}
We consider a system of infinitely many penduli on an $m$-dimensional lattice with a  weak coupling. For any prescribed path in the lattice, for suitable couplings, we construct orbits for this Hamiltonian system of infinite degrees of freedom which transfer energy between nearby penduli along the path. We allow the weak coupling to be next-to-nearest neighbor or long range as long as it is strongly decaying.

The transfer of energy is given by an Arnold diffusion mechanism  which relies on the original V. I Arnold approach: to construct a sequence of hyperbolic invariant quasiperiodic tori with transverse heteroclinic orbits. We implement this approach in an infinite dimensional setting, both in the space of bounded $\mathbb{Z}^m$-sequences and in spaces of decaying $\mathbb{Z}^m$-sequences. Key steps in the proof  are  an  invariant manifold theory for hyperbolic tori  and a Lambda Lemma for infinite dimensional coupled map  lattices with decaying interaction.
\end{abstract}

\tableofcontents

\section{Introduction}\label{sec:intro}

Transport and transfer of energy are one of the fundamental behaviors that arise in Hamiltonian dynamics both of finite and infinite dimensions. In finite dimensional nearly integrable Hamiltonian systems one of the main mechanisms to achieve such behavior is Arnold diffusion \cite{Arnold64} which leads to large drift in actions in phase space. Arnold diffusion is known to be one of the main sources of unstable motions in many physical models such as 
the Solar system and outstanding progress has been achieved in the last decades. 
In Hamiltonian PDEs (which can be seen as infinite dimensional Hamiltonian systems) the phenomenon of transfer of energy was considered by Bourgain one of the fundamental problems to study in Hamiltonian PDEs in the XXIst century (see Bourgain \cite{Bourgain00b}) and has drawn a lot of attention in the last decades. Even if the dynamics underlying such behavior presents substantial differences from the classical finite dimensional Arnold diffusion, some of the works  also rely on analyzing invariant objects and heteroclinic connections (see \cite{CKSTT, GuardiaK12, GuardiaHP16}).

The purpose of this paper is to construct transfer of energy solutions in a quite different context which has strong connections with both settings presented above: Hamiltonian systems with \emph{infinitely many} degrees of freedom defined on lattices, that is infinite dimensional Hamiltonian systems  with \emph{spatial structure}. 

The study of transfer of energy phenomenon in Hamiltonian systems on lattices  goes back to the seminal numerical study by Fermi, Pasta and Ulam \cite{FPU}, on the nowadays called Fermi-Pasta-Ulam model, and the discovery of the so-called FPU paradox. Since then, there has been a lot of effort on understanding  both the  phenomenon of energy localization and energy transfer both in periodic lattices (that is, finite dimensional phase space) or on infinite lattices. 

On energy localization, there are several papers that apply KAM Theory techniques to prove the existence of invariant tori  \cite{FrohlichSW86,Cherchia95,Yuan02,Geng07, Geng14, Wu21} which have strong decay in space and therefore have localized energy. There are also several results providing time estimates for energy localization (see for instance \cite{Wayne86a,Wayne86b,Bambusi93,Bambusi06, GaPa,Roeck15}).


Arnold diffusion results on Hamiltonian systems with spatial structure (either of finite or infinite dimensions) are rather scarce. In particular there are no results for the classical Fermi-Pasta-Ulam model (however, see \cite{KappelerH08,KappelerH09} for the analysis of hyperbolic objects in its normal form). 

In \cite{KaloshinLS14, Huang17} the authors consider a periodic  lattice model which consists on  penduli with weak coupling and prove the existence of transfer energy orbit by means of variational methods.
Inspired by these works, the goal of the present paper is to construct Arnold diffusion orbits for models in infinite lattices.
The mechanism considered in the seminal work by Arnold (and many of the most recent ones) relies on the analysis of invariant objects (typically invariant tori) and their heteroclinic connections. We also rely on this very same approach but in an \emph{infinite dimensional setting both in $\ell^\infty$ and in spaces with decay}.
To this end we consider geometric techniques which are currently widely used in finite dimensional Hamiltonian systems (invariant manifold theory for hyperbolic tori, Lambda lemma) and we develop them in a rather wide generality for Hamiltonian systems on lattices with spatial structure.

Then we apply them  to formal Hamiltonians of the form
\begin{equation}\label{def:Hamiltonian}
H(p, q)=\sum_{j\in\mathbb{Z}^m} E_j+\varepsilon H_1(p,q), 
\end{equation}
where 
\begin{equation}\label{def:energies}
E_j:=\frac{p_j^2}{2}+V(q_j), \qquad V(q_j):=\cos q_j-1,
\end{equation}
 together with the formal symplectic structure $\Omega=\sum_{j\in\Z^m} d p_j\wedge d q_j$.

The perturbation $H_1$ is assumed to have certain spatial structure that will be specified later. Roughly speaking, we either assume that only interaction with nearest and next-to-nearest neighbors is allowed or long range interaction is admitted provided it has strong decay. 
Under such assumptions, even if the Hamiltonian is just formal (the sum in \eqref{def:Hamiltonian} is not convergent),  the equations of motion
\begin{equation}\label{eq:motions}
\begin{cases}
\dot{q}_j=p_j+\varepsilon \partial_{p_j} H_1(p,q)\\[1.5mm]
\dot{p}_j=\sin q_j -\varepsilon \partial_{q_j} H_1(p,q),
\end{cases}  \qquad \qquad \qquad  j\in\mathbb{Z}^m
\end{equation}
define a well-behaved system of differential equations.

Even if the developed techniques are applied to Hamiltonian systems of the form   \eqref{def:Hamiltonian}, they are valid for a much wider class of Hamiltonian systems and, thus, we expect that they can be used in future results on Arnold diffusion in more general lattice models. Before stating the main results, let us review the literature in Arnold diffusion to put our result in context.

%
%

Since the seminal work by Arnold \cite{Arnold64} and specially since the 90s there has been a huge progress in understanding the phenomenon of Arnold diffusion in finite dimensional nearly integrable Hamiltonian systems. Such models are usually classified as a \emph{priori stable} (when the first order satisfies the Liouville Arnold Theorem) or \emph{a priori unstable} (when the first order is integrable but presents hyperbolicity). The model \eqref{def:Hamiltonian} belongs to the second setting.

The first results in the a priori unstable setting date back to the early 2000s \cite{DelshamsLS06a, ChengY04,Treschev04,Bernard08} for 2 and half degrees of freedom. The results in arbitrary dimension are more scarce \cite{DelshamsLS16,Treschev12}. 

A priori stable settings are much harder to analyze since the hyperbolicity which should lead to unstable motions must arise thanks to the perturbation. The results in this setting are much more recent \cite{Bernard16,KaloshinZ20, Marco16, GideaM17}. Many fundamental lattice models fit the a priori stable setting (for instance a weakly coupled sequence of rotators, such as the Fermi Pasta Ulam model in the low energy regime). Constructing Arnold diffusion orbits in such models is an outstanding open problem.

%


\subsection{Main results}
We devote this section to present transfer of energy results for the Hamiltonian system \eqref{eq:motions} and suitable perturbation $H_1$. A more complete statement would  require to set up first a functional setting to define the class of perturbations $H_1$ for which transfer of energy is possible. This more precise result  is deferred to Section \ref{sec:DescriptionProof} after establishing a functional setting considered first in \cite{JiangL00} and developed by Fontich, Mart\'in and de la Llave in \cite{FontMartin1} (see Section \ref{sec:functional} below). For now, we just present a simplified version.

When $\eps=0$, the Hamiltonian \eqref{def:Hamiltonian} is just a countable number of decoupled penduli. Therefore, the dynamics is integrable and transfer of energy among sites is not possible in the sense that the energies $E_j$ (see \eqref{def:energies}) are constants of motion. The goal of this paper is to construct, for suitable perturbations $H_1$, solutions $(q(t),p(t))$ such that its energy is transfered among modes as time evolves. Note that when we talk about the energy we refer to the values of $\{E_i(q_i, p_i)\}_{i\in\mathbb{Z}^m}$ without assuming that its sum is finite.

The transfer of energy solutions that we construct are such that its energy is supported essentially in one or two modes and it is transfered, as time evolves,  to neighboring sites. Therefore, to describe it we consider paths in $\Z^m$ formed by neighboring sites, that is sequences
\[
  \{\sigma_i\}_{i\geq 0}\subset \mathbb{Z}^m,\qquad |\sigma_{i+1}-\sigma_i|=1,
 \]
(see Figure \ref{fig:path}) where $|\cdot|$ denotes the usual $1$-norm for  $k\in \Z^m$, $|k|=\sum_{i=1}^m|k_i|$.

\begin{figure}[h]
\begin{center}
\includegraphics[height=5cm]{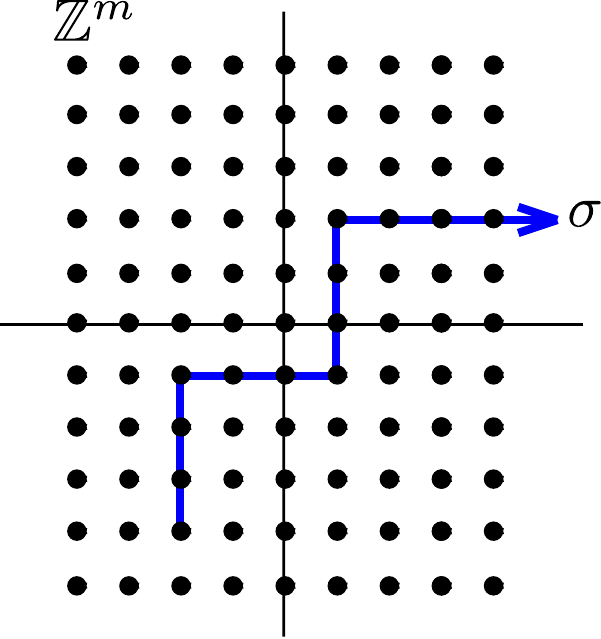}
\end{center}
\caption{Example of path in $\Z^m$.}\label{fig:path}
\end{figure}

To state the  results on transfer of energy orbits for Hamiltonians of the form \eqref{eq:motions}, we consider two different phase spaces.
For the first theorem we consider  the space of bounded sequences $\ell^{\infty}$. We consider the $\ell^{\infty}$--topology, rather than working with the topology based on pointwise convergence of the coordinates. This has the advantage that we can use Banach space techniques rather than relying just on metric spaces (which do not allow the standard tools of differential calculus).

We consider then as phase space 
\[
\ell^{\infty}(\Z^m; M)=\left\{ z=(z_i)_{i\in\mathbb{Z}^m}\in M^{\Z^m} : \sup_{i\in\Z^m} |z_i|<\infty \right\},
\]
where
\[
 M:=\T\times\R, \qquad M^{\Z^m}:=\prod_{j\in\Z^m} M,
\]
which is a Banach manifold modeled on $\ell^{\infty}(\Z^m; \R^2)$. 


Note that in this phase space the total energy may  not be finite. That is, one has to consider $H$ in \eqref{def:Hamiltonian} as a formal Hamiltonian. 
The second main result below, constructs  transfer of energy solutions which belong to a ``smaller'' phase space of strongly decaying sequences which makes $H$ well defined. 

\begin{theorem}\label{thm:MainFormal}
Fix $m\in\NN$, $h>0$ and consider the formal Hamiltonian $H$  in \eqref{def:Hamiltonian}. Then, there exist formal Hamiltonians $H_1$ of the form 
\begin{equation}\label{def:hamthm1}
 H_1(p,q)=\sum_{j_1,j_2,j_3\in\ZZ^m,|j_1-j_2|=1,|j_2-j_3|=1}\HH_1(p_{j_1},q_{j_1},p_{j_2},q_{j_2},p_{j_3},q_{j_3}),
\end{equation}
where $\mathcal{H}_1$ is a function of class $C^4$ and $\eps_0>0$,
such that for any $\eps\in (0,\eps_0)$ small enough, any $\eta>0$ small enough
  and any sequence 
 \[
  \{\sigma_i\}_{i\geq 0}\subset \mathbb{Z}^m,\qquad |\sigma_{i+1}-\sigma_i|=1,
 \]
there exist trajectories $(q(t),p(t))\in \ell^{\infty}(\Z^m; M)$ of \eqref{def:Hamiltonian} and an increasing sequence of times $\{t_i\}_{i\geq 0}$ such that
 \[
|  E_{\sigma_i}(q(t_i),p(t_i))-h|\leq\eta\qquad\text{and}\qquad |  E_{k}(q(t_i),p(t_i)|\leq\eta \quad\text{for}\quad k\neq \sigma_i.
 \]
\end{theorem}

Theorem \ref{thm:MainFormal} obtains Arnold diffusion orbits in $\ell^\infty$. In particular, the solutions do not have any particular decay and therefore the Hamiltonian $H$ may be unbounded. 

The next theorem, which is proven independently from Theorem \ref {thm:MainFormal} deals with sequences with a prescribed decay which makes the Hamiltonian $H$ well defined. To state it we define a different functional setting. 

Following \cite{FontMartin1}, we define a \emph{decay function} $\Gamma\colon\mathbb{Z}^m\to [0, \infty)$ such that
\begin{equation}\label{def:GammaDecay}
\sum_{i\in \Z^m} \Gamma(i)\le 1,\qquad \text{and}\qquad \sum_{j\in \Z^m} \Gamma(i-j)\,\Gamma(j-k)\le \Gamma(i-k), \quad i, k\in \Z^m.
\end{equation}
For instance given $\alpha>m$ and $\beta\geq 0$, there exists $a>0$ such that
 \begin{equation}\label{def:GammaExample}
 \Gamma(i)=\begin{cases}
 a |i|^{-\alpha}\,e^{-\beta |i|}, \quad i\neq0,\\
 a, \qquad \quad \qquad \,\,\,\, i=0
 \end{cases}
 \end{equation}
 is a decay function. 
Then, for a given decay function $\Gamma$ and $j\in \Z^m$, we define the space of sequences
\[
\Sigma_{j, \Gamma}:=\left\{ v\in \ell^{\infty}(\Z^m; M) : \sup_{k\in\Z^m} |v_k| \,\Gamma(k-j)^{-1}<\infty \right\}.
\]
Next theorem proves transfer of energy orbits in this phase space. 

\begin{theorem}\label{thm:MainEnergy}
Fix $m\in\NN$, $j\in \Z^m$, $h>0$, a decay function $\Gamma$ satisfying \eqref{def:GammaDecay} and the  Hamiltonian $H$ in \eqref{def:Hamiltonian}.  Then, there exist a $C^4$ Hamiltonians $H_1$ of the form \eqref{def:hamthm1} and $\eps_0>0$ such that for any $\eps\in (0,\eps_0)$, any $\eta>0$ small enough   and any sequence 
 \[
  \{\sigma_i\}_{i\geq 0}\subset \mathbb{Z}^m,\qquad |\sigma_{i+1}-\sigma_i|=1,
 \]
there exist trajectories $(q(t),p(t))$ of \eqref{def:Hamiltonian} such that, for $t\geq0$ satisfy $(q(t),p(t))\in \Sigma_{j, \Gamma}$, which  in particular implies that
 $ H(q(t),p(t))$ is finite, 
and an increasing sequence of times $\{t_i\}_{i\geq 0}$ such that
 \[
|  E_{\sigma_i}(q(t_i),p(t_i))-h|\leq\eta\qquad\text{and}\qquad |  E_{k}(q(t_i),p(t_i))|\leq\eta \quad\text{for}\quad k\neq \sigma_i.
 \]
\end{theorem}

Figure \ref{fig:sequence} shows schematically the evolution of transfer of energy orbits obtained in Theorems \ref{thm:MainFormal} and \ref{thm:MainEnergy}. The statements of these theorems only ensure the existence of ``one'' perturbation $H_1$ for which they apply. Certainly, they apply to families of perturbations. As mentioned above, in Section \ref{sec:DescriptionProof}, once the functional setting we work with is established, we give explicit conditions for $H_1$ which lead to transfer of energy. These conditions are essentially of two types. Some of them impose the invariance of certain finite dimensional subspaces. The others are of Melnikov-type and allow to ensure that certain invariant manifolds intersect transversally.
These conditions are not only satisfied by perturbations $H_1$ of next-to-nearest neighbor interaction type but they are also satisfied by  $H_1$ which have strongly decaying long-range interactions and they are \emph{explicit} and thus checkable in concrete examples (see \ref{sec:DescriptionProof}).


\begin{figure}[h]
\begin{center}
\includegraphics[height=3.4cm]{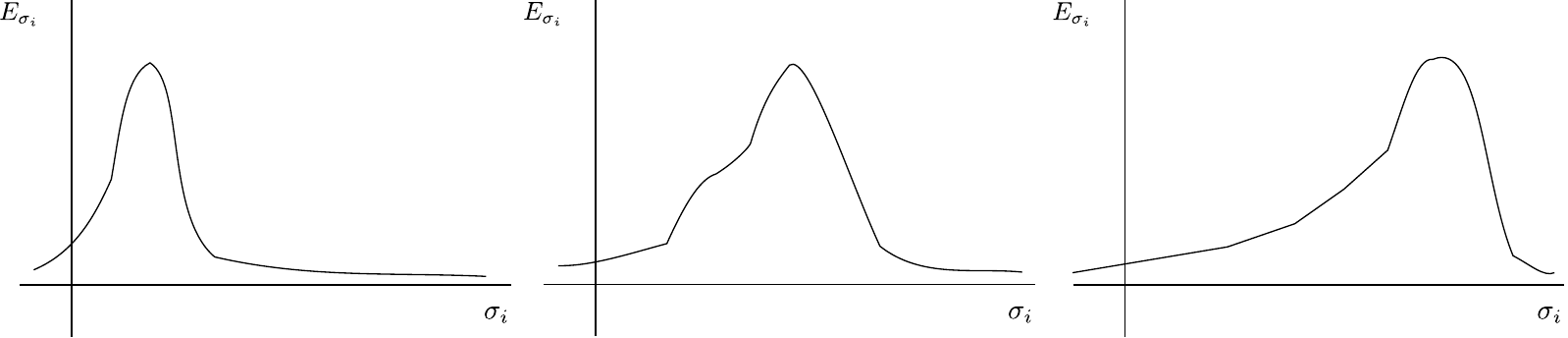}
\end{center}
\caption{Example of the evolution of transfer of energy. In the first picture the energy is essentially localized at one site $\sigma_i$, in the second picture is transferred to the next site in the path, $\sigma_{i+1}$, and in the final one is localized in  $\sigma_{i+1}$. Note that the tails in energy can be either decaying (as in Theorem \ref{thm:MainEnergy}) or just small and bounded (as in Theorem \ref{thm:MainFormal}). In this second case the total energy may be unbounded.}\label{fig:sequence}
\end{figure}

\subsection{Comments on Theorems \ref{thm:MainFormal} and \ref{thm:MainEnergy}}

\begin{enumerate}
\item Even if Theorem  \ref{thm:MainFormal} can be seen as a consequence of Theorem \ref{thm:MainEnergy}, their proofs are independent (although they follow the same scheme). That is, all our techniques are independent of the fact whether the Hamiltonian $H$ is convergent or not and the techniques we use are flexible enough so that can be applied in different functional settings.
 
 \item Note that the choice of the function $\Gamma$ is rather flexible. If one considers $\Gamma$ as in  \eqref{def:GammaExample} one can impose either polynomial decay or exponential decay. Moreover the exponentail decay can be as strong as desired ($\beta$ as large as desired) although certainly the smallness of $\eps$ depends on the choice of $\alpha$ and $\beta$.
 
 

\item The proof of Theorems \ref{thm:MainFormal} and \ref{thm:MainEnergy} is achieved through geometric methods. This implies that the convexity in actions of the Hamiltonian \eqref{def:Hamiltonian} does not play any role. Indeed, one can obtain the same result for a Hamiltonian of the form 
\[H(p, q)=\sum_{j\in\mathbb{Z}} \rho_j E_j+\varepsilon H_1(p,q), \]
where $\rho_j$ is either $\rho_j=1$ or $\rho_j=-1$. Then, however, one has to take $h\in (0,2)$. Indeed, even if the energy may be unbounded, it is not in the invariant objects and their associated invariant manifolds which are used to construct the diffusing orbit. Since the energy of the pendulum is bounded by below by $-2$ one has to impose that $h\in (0,2)$ if some of the $\rho_j$'s are negative.

\item The results in \cite{KaloshinLS14, Huang17}  and also in the present paper rely on models whose first order presents hyperbolicity (penduli) and whose perturbations  are carefully chosen so that preserve certain invariant subspaces. However Arnold diffusion should appear for generic perturbations with spatial structure (for instance generic nearest neighbor interaction) and in particular also for physical models such as discrete Klein-Gordon equations. It would also be interesting to involve more general invariant objects in the construction of the Arnold diffusion orbits (see for instance \cite{FontichLS15}, where the authors construct infinite dimensional hyperbolic tori).

\item Note that there are other mechanisms which lead to transfer of energy in Hamiltonian systems on lattices such as traveling waves (see \cite{Friesecke99}). They are of rather different nature compared to Arnold diffusion. 

\end{enumerate}

We devote the next sections to put our result in context. First in Section \ref{sec:Vadim} we compare our result with those of \cite{KaloshinLS14} and \cite{Huang17} where the same pendulum lattice model is considered but in a periodic setting. 
Section \ref{sec:PDEs} is devoted to make a connection between our main results with  the transfer of energy phenomenon in Hamiltonian PDEs, which is usually measured by the growth of Sobolev norms. Finally Sections \ref{sec:geometricmethods} and \ref{sec:heuristics} are devoted to explain the fundamental geometric tools that we develop to construct the transfer orbits and to explain the heuristics behind the transfer mechanism respectively.

\subsection{Comparison with \cite{KaloshinLS14} and \cite{Huang17}}\label{sec:Vadim}

The model \eqref{def:Hamiltonian} was considered in \cite{KaloshinLS14} and \cite{Huang17} in a periodic one dimensional lattice and therefore in a finite dimensional phase space. 
In \cite{KaloshinLS14}, Kaloshin, Levi and Saprykina  prove the existence of transfer of energy orbits for suitable next-to-nearest-neighbor perturbations by means of variational methods (in the spirit of \cite{Bessi96, BeBo,BeBo02,BertiBB03}, see also \cite{KaloshinL08a,KaloshinL08b}) and provide time estimates. The perturbations they consider are $C^\infty$ and  localized. Thanks to this localization one could expect that their techniques could be implemented in an infinite dimensional setting. 
The paper \cite{Huang17} considers the very same model but with analytic perturbations. 
Both paper follow the same diffusion mechanism  developed in the original paper \cite{KaloshinLS14}. 
We also rely on the same \emph{heuristic mechanism}, which is explained in Section \ref{sec:heuristics} below.

%
%

In the present paper, we consider an infinite dimensional setting in a rather wide generality, both in $\ell^\infty$ or in spaces with decay. In particular, we do not impose finite energy. Moreover, we consider a completely different approach. Instead of considering variational methods as in \cite{KaloshinLS14,Huang17}, we consider geometric methods following the original Arnold approach. 

The choice of perturbations $H_1$ present similar features in all three works. Indeed, a very important property is that they leave invariant certain finite dimensional subspaces. In \cite{KaloshinLS14, Huang17} the perturbations are constructed so  that certain barrier function is non-degenerate. These conditions are rather similar (in fact slightly weaker) compared to the Melnikov-type conditions that we impose to ensure that the invariant manifolds of certain invariant tori intersect transversally (see Section \ref{sec:DescriptionProof} below). The conditions that we impose are explicit and can be checkable in concrete examples.
 
The advantage of the choice of both the perturbation and the variational methods  in  \cite{KaloshinLS14,Huang17} allows the author to obtain time estimates on how fast is the transfer of energy. The tool used to perform shadowing in the present work, a Lambda lemma, is quite flexible but unfortunately does not lead to time estimates. To obtain time estimates one would need to develop a more  quantitative Lambda lemma or implement the 
 variational methods of \cite{KaloshinLS14,Huang17} in the infinite dimensional setting.
%
%
%
%


\subsection{Transfers of energy and growth of Sobolev norms: PDEs vs Lattices}\label{sec:PDEs}

For $s\geq 0$ let us define the Sobolev spaces
\[
H^s:=\left\{  u \colon \Z^m\to \mathbb{C} : \| u \|_{H^s}:=\left( \sum_{k\in \Z^m} |u_k|^2\,\langle k \rangle^{2s}\right)^{1/2}<\infty \right\},
\]
where $\langle k \rangle:=\max\{ 1, |k| \}$.
Observe also that the space of sequences $\Sigma_{j, \Gamma}$ with $j=0, \beta=0$ coincides with the H\"older space 
\[
W^{\alpha, \infty}:=\left\{  u \colon \Z^m\to \mathbb{C} : \| u \|_{W^{\alpha, \infty}}:= \sup_{k\in \Z^m} |u_k|\,\langle k \rangle^{\alpha}<\infty \right\}.
\]
As it is well known $W^{\alpha+s_0, \infty}(\T^m)\subset H^{\alpha}(\T^m)$ with $s_0>m/2$.
Then it is easy to see that Theorem \ref{thm:MainEnergy} provides the existence of solutions whose Sobolev norms explode as time goes to infinity.  In particular it provides a result of ``strong'' Lyapunov instability (see \cite{GuardiaHHMP19}) for some finite dimensional invariant tori in the topology of Sobolev spaces. Indeed, the perturbations  $H_1$ considered in Theorems \ref{thm:MainFormal} and 
\ref{thm:MainEnergy} are such that the tori 
\[
\T_{\sigma_1, \sigma_2, h_1,h_2}=\{ E_{\sigma_1}=h_1, E_{\sigma_2}=h_2, \,E_{k}=0\,\,\,\mbox{for}\,\,k\neq \sigma_{1}, \sigma_2 \}
\]
are invariant. The next corollary, direct consequence of Theorem \ref{thm:MainEnergy}, implies that these tori possess a strong form of Lyapunov instability.

\begin{corollary}
Fix $m\in \NN$, $s>m$, $\sigma_1, \sigma_2\in \Z^m$ with $|\sigma_1-\sigma_2|=1$, $h_1, h_2>0$.
Assume the assumptions of Theorem
\ref{thm:MainEnergy} and consider
the invariant torus $\T_{\sigma_1, \sigma_2, h_1,h_2}$.
Then, there exist $\e_0>0$ small enough such that for all $0\le \e\le \e_0$ and any $\eta>0$ small enough,  there is a trajectory $u(t)=(q(t), p(t))$ of \eqref{def:Hamiltonian} such that
\[
\mathrm{dist}_{H^s}(u(0), \T_{\sigma_1, \sigma_2, h_1,h_2}):=\sup_{z\in \T_{\sigma_1, \sigma_2, h_1,h_2}} \| u(0)-z \|_{H^s}<\eta
\]
and
\[
 \lim_{t\to +\infty} \| u(t) \|_{H^s}=+\infty.
\]
\end{corollary}

Actually the same result holds in any space of sequences with decay $\Gamma$ (for instance H\"older and analytic spaces).  Indeed the energy is transferred to arbitrarily far (with respect to the initial conditions) regions of the lattice and the decay norms give a weight to the sites.

\medskip

One of the most important issue in the modern analysis of Hamiltonian PDEs concerns the transfers of energy between modes of solutions of nonlinear equations on compact manifolds. In particular, when this transfer occurs between modes of characteristically different scale, this phenomenon is named \emph{energy cascade} (in the weak wave turbulence theory) and it can be measured by analyzing growth of Sobolev norms of the solutions as time evolves. In the last decade several papers have been dedicated to prove existence of solutions undergoing an arbitrarily large growth in their high order Sobolev norms (\cite{CKSTT}, \cite{GuardiaK12}, \cite{HPquintic},\cite{GuardiaHP16}, \cite{GuardiaHHMP19}). Such results are very interesting from the point of view of the study of dynamics of PDEs, since they can be read as Lyapunov instability phenomena in the topology of Sobolev spaces of some invariant objects (fixed points, periodic orbits, quasi-periodic tori \dots).

It remains an interesting open problem \cite{Bourgain00b} whether there are solutions of the cubic NLS on $\T^d$, $d\geq 2$ exhibiting an \emph{unbounded} growth, i.e.
\[
\limsup_{t\to \infty} \| u(t) \|_{H^s(\T^2)}=+\infty.
\] 
We mention that solutions displaying an unbounded growth have been found by Hani \cite{Hani12} and Hani-Pausader-Tzvetkov-Visciglia \cite{HaniPTV15} respectively in the case of NLS with cubic nonlinearities which are "almost polynomial" and for the NLS on the cross product $\T^2\times \R$.

\smallskip

As it is well known, partial differential equations $\dot{u}(t, x)=X(u(t, x))$ under periodic boundary conditions, $x\in \T^m$, can be seen as infinite dimensional systems of ODEs for the Fourier coefficients
\[
u_k:=(2\pi)^{-m}\int_{\T^m} u(x)\,e^{-\mathrm{i} k x}\,dx.
\]
 In many important models this takes the following form 
\begin{equation}\label{odeFourier}
\dot{u}_k=\mathrm{i} \omega(k) u_k+f_k (u),  \qquad k\in \Z^m,
\end{equation}
where $\omega(k)$ are complex numbers.
The modes are uncoupled at the linear level, while the nonlinearity couples all of them. If the nonlinear terms have a zero of order at least two at the origin, in a sufficiently small neighborhood of the origin these systems can be seen as nearly-integrable, where the linear part plays the role of the unperturbed equation. Then, one can make a comparison with lattice models of the form $H_0+\e H_1$ (see \eqref{def:Hamiltonian}).
We notice some fundamental differences between \eqref{odeFourier} and our lattice model which play a significant role in the study of unstable orbits:
\begin{itemize}

\item In many dispersive PDEs the linear frequencies $\omega(k)$ are real numbers, except for finitely many $k$'s (for instance Klein-Gordon with negative mass). Then the linear dynamics is stable, more precisely all the linear motions are oscillations.
Hyperbolicity should arise from nonlinear terms.
 In our model the unperturbed system $H_0$ presents strong hyperbolicity properties, in the sense that there exists an equilibrium which is hyperbolic in all the \emph{infinitely many} directions. In other words, to deal with dispersive PDEs we should be able to treat a priori stable infinite dimensional problems. 

\item The nonlinear coupling in PDEs is not just long range, but the interaction between very distant modes is as strong as between nearest neighbor modes. This is a fundamental difference with our model where the interaction is nearest-neighbor or long range but with strong decay. 

\item In our model the subspaces obtained by keeping at rest any set of modes are invariant. In PDEs in general this is true at the linear level, when the modes are all uncoupled, but not considering the nonlinear effects. 

\end{itemize}

Filling these gaps would provide a significant step forward to the extension of Arnold diffusion to PDEs.

Besides the interest per se, it would be interesting to understand whether this kind of phenomena may provide results of existence of solutions displaying unbounded growth in Sobolev norms.

%
%
%
%
%
%

\subsection{Main tools of the proofs of Theorems \ref{thm:MainFormal} and \ref{thm:MainEnergy}}\label{sec:geometricmethods}

The proofs of both Theorems \ref{thm:MainFormal} and \ref{thm:MainEnergy} rely on the same techniques which are usually referred to as \emph{geometric methods for Arnold diffusion}, which go back to the seminal paper by Arnold \cite{Arnold64}. These techniques have been shown to be extremely powerful in the analysis of unstable motions in nearly integrable systems \cite{DelshamsLS06a,DelshamsLS16,Marco16,GideaM17}.
In the last decades they have also been shown to be extremely powerful in combination with  Variational Methods (Mather Theory, Weak KAM).

The geometric methods  tools that are involved in the proof of Theorems \ref{thm:MainFormal} and \ref{thm:MainEnergy} are the following:
\begin{enumerate}
 \item Construct a sequence of invariant tori $\{\TT_k\}_{k\geq 1}$, which are (partially) hyperbolic. Usually KAM theory is needed in this step (see \cite{FontichLS15}). However, in the present paper we choose  the perturbation $H_1$  such that many of the tori of $H_0$ are preserved (see Section \ref{sec:heuristics}). Note that these tori do not need to have the same dimension. 
 \item Prove that these invariant tori have stable and unstable invariant manifolds (often called whiskers) and that they are regular with respect to parameters. There are several papers dealing with invariant manifolds of whiskered tori in lattice models\footnote{Note that \cite{Blazevski} deals with the invariant manifolds of both finite and infinite dimensional quasiperiodic invariant tori. } \cite{Blazevski,Berenguel,Berenguel19} (see also \cite{FontichLM11} for invariant manifolds of hyperbolic sets).
In Section \ref{sec:manifold} we develop an invariant manifold theory which can be applied to the Hamiltonian \eqref{def:Hamiltonian}. The results we develop are applicable both to maps and flows, require low regularity assumptions and do not require that the maps/flows preserve a symplectic structure.
 \item Prove that the unstable manifold of $\TT_k$ and the stable manifold of $\TT_{k+1}$ intersect transversally. This is usually done by means of (a suitable version of) Melnikov Theory. In the so--called Arnold regime (see Section \ref{sec:heuristics} below), one can also use the scattering map to understand the homoclinic connections to certain normally hyperbolic cylinders (see \cite{DelshamsLS08}). This analysis is done in Section \ref{sec:MelnikovTheory}.
 \item A sequence of invariant tori $\{\TT_k\}_{k\geq 1}$ whose consecutive tori are connected by transverse heteroclinics is usually called a \emph{transition chain}. The last step is to prove that there is an orbit which ``shadows'' (follows closely) this transition chain. To this end, one needs an (infinite dimensional) Lambda Lemma. As far as the authors know the Lambda lemma proved in the present paper is the first one in an infinite dimensional setting. It is proven in Section \ref{sec:lambdalemma}.
\end{enumerate}

The implementation of these steps in the pendulum lattice \eqref{def:Hamiltonian} is explained in Section \eqref{sec:heuristics} at an ``informal'' level and in full detail in  Section \ref{sec:DescriptionProof}. However, we believe that the techniques that we develop in this paper for the Steps 2, 3 and 4 above have wide applicability beyond pendulum lattices. For this reason, they are stated in a general form in Sections \ref{sec:manifold}--\ref{sec:lambdalemma}.

\subsection{Heuristics on the instability mechanism}\label{sec:heuristics}

The instability mechanism that leads to the transfer of energy trajectories of Theorems \ref{thm:MainFormal} and \ref{thm:MainEnergy} relies on the ideas of Arnold \cite{Arnold64} of building a sequence of invariant tori connected by transverse heteroclinic orbits, that is a transition chain of whiskered tori, as mentioned in Section \ref{sec:geometricmethods}. Let us give a rough idea of how this transition chain is constructed. When $\eps=0$, $H$ is just an infinite number of uncoupled penduli. Therefore the phase space possesses plenty of  invariant tori which may be of ``maximal'' (infinite) dimension or can be of (finite or infinite) ``lower dimension'' partially elliptic and hyperbolic.

We consider perturbations $H_1$ such that certain finite dimensional hyperbolic tori of $H_0$ are persistent. 
Fix an instability path
\[
  \{\sigma_i\}_{i\geq 0}\subset \mathbb{Z}^m,\qquad |\sigma_{i+1}-\sigma_i|=1.
\]
 Then we assume the following
\[
 \left.\partial_{q_{k}} H_1(q,p)\right|_{(q_k,p_k)=(0,0)}=\left.\partial_{p_k} H_1(q,p)\right|_{(q_k,p_k)=(0,0)}=0, \qquad \forall k\notin \{\sigma_i\}_{i\geq 0}.
\]
This condition implies that, if we set $S:=\{ \sigma_i \}_{i\geq 0}\subset\ZZ^m$, the subspace
\begin{equation}\label{def:subsp}
\mathcal{V}_S=\left\{q_k=p_k=0 \quad \text{for all}\quad k\not\in S\right\}
\end{equation}
is left invariant by the vector field of $H_1$. Let us define
\[
S_i:=\left\{\sigma_i,\sigma_{i+1}\right\}\subset\ZZ^m.
\]
We assume the following extra hypothesis so that the dynamics on $\mathcal{V}_{S_i}$ is integrable and given by two uncoupled penduli,
\[
\left.
\partial_{q_k} H_1(q,p)
 \right|_{\mathcal{V}_{S_i}}
 =\left.\partial_{p_k} H_1(q,p)\right|_{\mathcal{V}_{S_i}}=0\qquad \text{for}\qquad k=\sigma_i,\sigma_{i+1}.
\]
Indeed, it  implies that 
\[
 H|_{\mathcal{V}_{S_i}}= E_{\sigma_{i}}+E_{\sigma_{i+1}},
\]
which is an integral Hamiltonian given by two uncoupled penduli. 

The transfer of energy mechanism behind Theorems \ref{thm:MainFormal} and \ref{thm:MainEnergy} rely on a transition chain of whiskered invariant tori which are supported on the sequence of invariant subspaces $\cV_{S_i}$.

Fix $h>0$. Then, we consider a transition chain of invariant tori which belong to the energy level $H=h$. Note that Theorem \ref{thm:MainEnergy} constructs a shadowing orbit with finite energy 
whereas Theorem \ref{thm:MainFormal} is obtained through a shadowing argument performed in $\ell^{\infty}$ and thus the energy is just a formal object. However, to construct these orbit in both cases we rely on invariant objects which belong to the $h$ energy level.

The tori in the transition chain  are of two types, first, for $h_1,h_2>0$ such that $h_1+h_2=h$, we consider the two dimensional tori  defined by 
\[
\TT_{\sigma_i,\sigma_{i+1},h_1,h_2}=\{  E_{\sigma_{i}} =h_1, E_{\sigma_{i+1}} =h_2\,\text{ and }\, E_k=0\, \text{ for }\, k\neq \sigma_{i},\sigma_{i+1}\}.
\]
These are  hyperbolic invariant tori for $H$ in the full phase space with an infinite number of hyperbolic stable and unstable directions. In the two first rows of Figure \ref{fig:chain} we show two examples of these tori in a subspace of three penduli: two penduli are at a periodic orbit and the third one at the saddle. 

Second, we consider the one-dimensional invariant tori defined as 
\[
\PP_{\sigma_i} =\{   E_{\sigma_{i}}=h\,\text{ and }\, E_k=0\, \text{ for }\, k\neq \sigma_{i}\}.
\]
Note that $\PP_{\sigma_i}\subset \mathcal{V}_{S_{i-1}}$ and $\PP_{\sigma_i}\subset \mathcal{V}_{S_{i}}$. This can be seen at the last row of Figure \ref{fig:chain}: two penduli at the saddle and one at a periodic orbit.

\begin{figure}[h]
\begin{center}
\includegraphics[height=14cm, width=11cm]{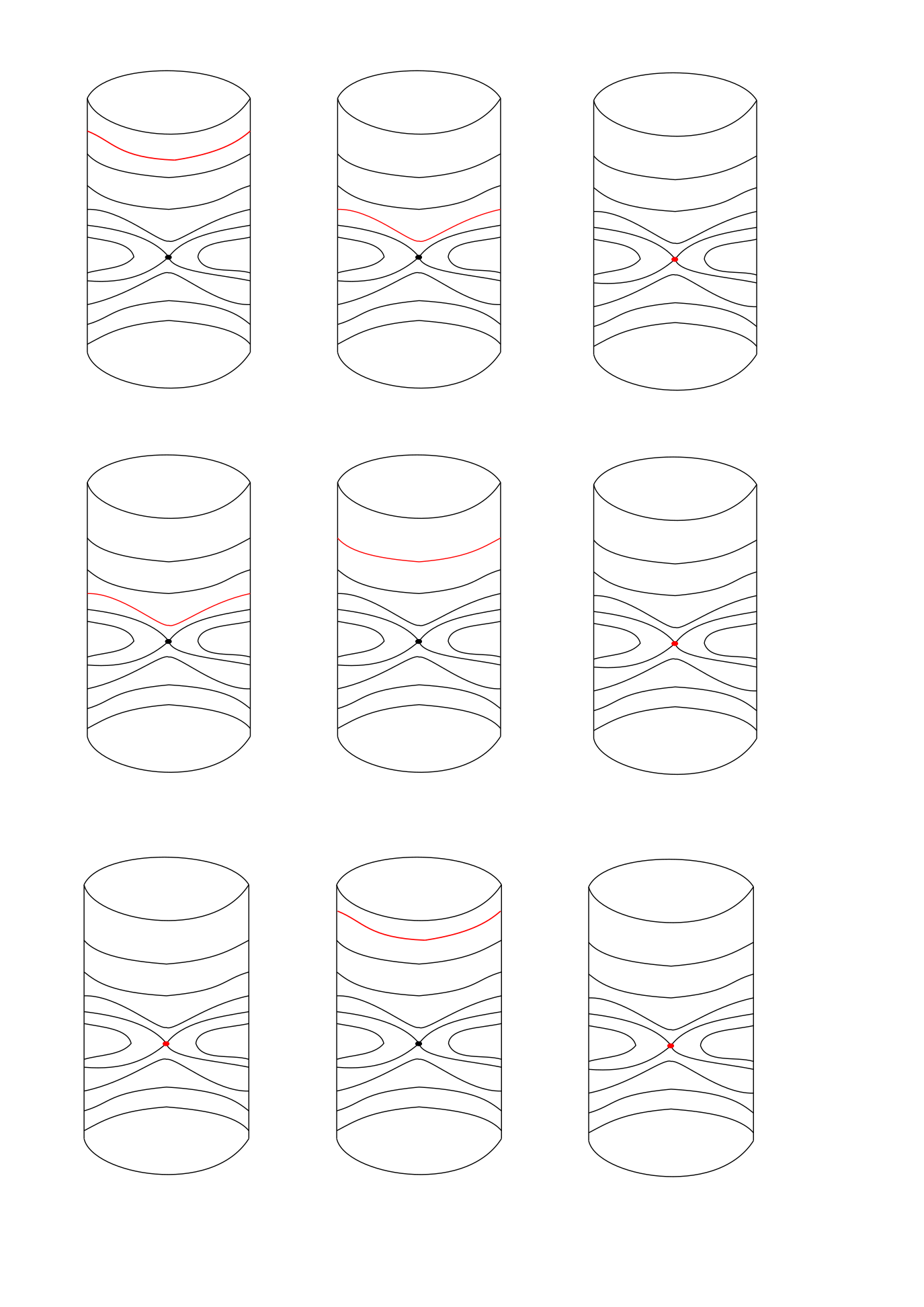}
\end{center}
\caption{The transition chain for the restricted system to the $6$-dimensional invariant subspace of three coupled penduli. First, in the Arnold regime, we connect the $2$-dimensional tori represented in the first two pictures (using the standard machinery of the scattering map). Then, in the Jumping regime, we connect the second torus with the periodic orbit, represented in the last picture. }\label{fig:chain}
\end{figure}

Then, to prove Theorems \ref{thm:MainFormal} and \ref{thm:MainEnergy}, we construct a transition chain of the form 
\[
\PP_{\sigma_0}\cup \bigcup_{k=1}^N \left\{\TT_{\sigma_0,\sigma_{1},h_k,h-h_k}\right\}\cup\PP_{\sigma_1}\cup \bigcup_{k'=1}^{N'} \left\{\TT_{\sigma_1,\sigma_{2},h_{k'},h-h_{k'}}\right\}\cup\PP_{\sigma_2}\cup \ldots
\]
for some sequences of energies $h_k$, $h_{k'}$. Note that the transition chain has tori of both  dimension one and two.  

To prove that such sequence of tori connected by heteroclinic connections exists, one can distinguish two regimes:

\begin{itemize}
 \item Arnold regime: Fix $\delta>0$. Then 
 \[ \La_{\de,\sigma_i,\sigma_{i+1}}=\bigcup_{\wt h\in[\de,h-\de]}\TT_{\sigma_i,\sigma_{i+1},\wt h,h-\wt h}\]
is a normally hyperbolic cylinder foliated by invariant tori as in the classical Arnold example \cite{Arnold64}. Note however that this cylinder has infinite dimensional invariant manifolds. To construct a transition chain we need to impose certain non-degeneracy conditions to an associated Melnikov potential (which only depends on a finite number of sites). This allows to define scattering maps \cite{DelshamsLS08} and by it a transition chain of two dimensional tori ``from top to bottom'' of the cylinder, i.e. with increasing $E_{\sigma_{i+1}}$ and decreasing $E_{\sigma_{i}}$ (the sum must be constant since we construct the transition chain in the energy level $H=h$), see first two rows of Figure \ref{fig:chain}. Thus, through this transition chain energy is only being transferred between the site $\sigma_i$ to the site $\sigma_{i+1}$. 
%
 \item Jumping regime: In the second regime we want to connect tori from $  \La_{\de, \sigma_i,\sigma_{i+1}}$ to the periodic orbit $\PP_{\sigma_{i+1}}$ (where the energy is all supported in the site $\sigma_{i+1}$) and then to the ``new'' cylinder $\La_{\de, \sigma_{i+1},\sigma_{i+2}}$, see the last two rows of Figure \ref{fig:chain}. In this second regime we must construct heteroclinic orbits between invariant tori of different dimension. To construct them we also rely on the non-degeneracy of certain Melnikov Potential. Note that in this regime one cannot rely on normally hyperbolic cylinders since when $\de\to 0$, the hyperbolicity inside the cylinder is as strong as the normal one.
 \end{itemize}

The last step is to construct an orbit shadowing this transition chain. This is a consequence of the Lambda lemma which implies that the unstable manifold of a given torus belongs to the closure of the unstable invariant manifold of the previous torus in the sequence.

\subsection{Structure of the paper}
We end Section \ref{sec:intro} by explaining the structure of the rest of the paper. First in Section \ref{sec:functional} we explain the functional setting that we consider in this paper, which was developed by de la Llave, Fontich and Mart\'in in \cite{FontMartin1}. In Section \ref{sec:DescriptionProof} we state a more detailed theorem which implies Theorems \ref{thm:MainFormal} and \ref{thm:MainEnergy}. Then, we explain the main steps to prove this theorem. Those main steps are an invariant manifold theory for hyperbolic tori, analysis of the transverse intersections of the invariant manifolds and a Lambda Lemma. 

The rest of sections, that is  Sections \ref{sec:manifold}, \ref{sec:MelnikovTheory} and \ref{sec:lambdalemma} are devoted to prove these three main steps. However, in this sections we do not just develop such theories for the model \eqref{def:Hamiltonian} but in a rather general setting.

First in Section \ref{sec:manifold} we develop an (infinite dimension) invariant manifold theory for finite dimensional hyperbolic tori for both coupled maps lattices and vector fields on lattices. We also prove regularity of the invariant manifolds with respect to parameters. In Section  \ref{sec:MelnikovTheory} we analyze the transversality of the invariant manifolds of the invariant tori by a Melnikov-type theory. Finally, in Section \ref{sec:lambdalemma} we prove a Lambda lemma for the invariant manifolds of hyperbolic tori both for flows and maps.

We want to emphasize that specially Sections \ref{sec:manifold} and \ref{sec:lambdalemma} apply for a rather wide class of infinite dimensional dynamical systems (both discrete and continuous) with spatial structure. We believe that the results obtained in these sections have a much wider applicability in analyzing unstable motions in infinite dimensions.

\subsection*{Acknowledgements}
We warmfully thanks Amadeu Delshams, Ernest Fontich and Pau Mart\'in for useful discussions. The authors are supported by the European Research Council (ERC) 
under the European Union's Horizon 2020
research and innovation programme (grant agreement 
No. 757802). M. Guardia is also supported by the Catalan Institution for Research and Advanced
Studies via an ICREA Academia Prize 2019. This work is also supported by the Spanish State Research Agency, through the Severo Ochoa and Mar\'ia de Maeztu Program for Centers and Units of Excellence in R\&D (CEX2020-001084-M).

\section{Functional setting}\label{sec:functional}
We devote this section to introduce the functional setting needed to prove Theorems \ref{thm:MainFormal} and \ref{thm:MainEnergy}. We use the functional setting developed in  \cite{FontMartin1}. Most of the results stated in this section are proven in that paper.

 Let $(\mathcal{X}_i, |\cdot|_{\mathcal{X}_i})$, $i\in \Z^m$, $m\geq 1$, be a sequence of Banach spaces and let us denote by $\ell^{\infty}(\mathcal{X}_i)$ the Banach space 
\[
\ell^{\infty}(\mathcal{X}_i):=\left\{z\in \prod_{i\in\mathbb{Z}^m} \mathcal{X}_i : \,\,\| z \|_{\ell^{\infty}(\mathcal{X}_i)}:=\sup_{i\in\mathbb{Z}^m} |z_i|_{\mathcal{X}_i}<\infty \right\}.
\] 
To lighten the notation we use $\ell^{\infty}$ instead of $\ell^{\infty}(\mathcal{X}_i)$ when this does not create confusion.

We want to introduce a class of maps that preserve these spaces of bounded sequences. The key point is to consider maps with some decay property.  Following \cite{FontMartin1}, we consider a \emph{decay function} $\Gamma\colon\mathbb{Z}^m\to [0, \infty)$ such that
 \begin{enumerate}
 \item $\sum_{i\in \Z^m} \Gamma(i)\le 1,$
 \item $\sum_{j\in \Z^m} \Gamma(i-j)\,\Gamma(j-k)\le \Gamma(i-k), \quad i, k\in \Z^m.$
 \end{enumerate}
For instance given $\alpha>m$ and $\beta\geq 0$, there exists $a>0$ such that
 \[
 \Gamma(i)=\begin{cases}
 a |i|^{-\alpha}\,e^{-\beta |i|}, \quad i\neq0,\\
 a, \qquad \quad \qquad \,\, i=0
 \end{cases}
 \]
 is a decay function. From now on when we refer to $\mathcal{X}_i$ as a sequence of Banach spaces we mean $(\mathcal{X}_i)_{i\in \Z^m}$.

 First in Section \ref{sec:lineardecay} we define linear operators from $\ell^{\infty}(\mathcal{X}_i)$ to itself with decay. This is the only class of linear operators that we consider in this paper. Later in Section \ref{sec:multilineardecay} we define accordingly the multilinear operators and nonlinear maps. 
 In Section \ref{sec:firstint} we give the definition of formal first integrals for both flows and maps.

 For Theorem \ref{thm:MainEnergy} we consider subspaces of  $\ell^{\infty}(\mathcal{X}_i)$ of sequences with decay. In Section \ref{sec:decayspaces} we analyze these spaces and state properties of the operators and maps introduced in Sections \ref{sec:lineardecay} and \ref{sec:multilineardecay} when restricted to these subspaces. 
 Finally, in Section \ref{sec:functsettorus} we consider certain coordinate transformations in  $\ell^{\infty}(\mathcal{X}_i)$ that are well adapted to analyze the dynamics close to certain invariant tori. Then, we see how the analysis performed in the previous sections is adapted to this new set of coordinates.

 \subsection{Linear operators with decay}\label{sec:lineardecay}
 
  Given two sequences of Banach spaces $\mathcal{X}_i, \mathcal{Y}_i$ we define the Banach space of linear maps with decay $\Gamma$ by
\[
\cL_{\Gamma}:=\cL_{\Gamma}(\ell^{\infty}(\mathcal{X}_i); \ell^{\infty}(\mathcal{Y}_i)):=\{ A\in \cL(\ell^{\infty}(\mathcal{X}_i); \ell^{\infty}(\mathcal{Y}_i)) : \| A \|_{\mathcal{L}_{\Gamma}}<\infty \}
\]
where
\begin{equation}\label{def:decaynorm}
\| A \|_{\mathcal{L}_{\Gamma}}:=\max\{ \| A \|_{\mathcal{L}}, \gamma(A) \} \qquad \text{with}\qquad
\gamma(A):=\sup_{i, j\in\mathbb{Z}^m} \,\,\, \sup_{\substack{\|u\|_{\ell^{\infty}}\le 1,\\ \pi_{\ell} u=0, \,\ell\neq j}} |(A u)_i|_{\mathcal{Y}_i}\,\Gamma(i-j)^{-1}.
\end{equation}

Now we state several properties of $\mathcal{L}_\Gamma$ and the operators with decay. 
The first key property is that the  space $\mathcal{L}_{\Gamma}$ is a Banach algebra with respect to the composition. 

\begin{lem}[Proposition $2.8$ in \cite{FontMartin1}]
Let $\mathcal{X}_i, \mathcal{Y}_i, \mathcal{Z}_i$ be sequences of Banach spaces. If $A\in \mathcal{L}_{\Gamma}(\ell^{\infty}(\mathcal{X}_i); \ell^{\infty}(\mathcal{Y}_i))$ and $B\in \mathcal{L}_{\Gamma}(\ell^{\infty}(\mathcal{Y}_i); \ell^{\infty}(\mathcal{Z}_i))$ then
\begin{itemize}
\item $BA\in\mathcal{L}_{\Gamma}(\ell^{\infty}(\mathcal{X}_i); \ell^{\infty}(\mathcal{Z}_i) )$;
\item $\gamma(BA)\le \gamma(B) \gamma(A)$;
\item $\| B A\|_{\mathcal{L}_{\Gamma}}\le \| B \|_{\mathcal{L}_{\Gamma}} \| A \|_{\mathcal{L}_{\Gamma}}$.
\end{itemize}
\end{lem}

The second property of operators with decay is that on certain subspaces they have a ``matrix representation''. Indeed, consider a sequence of Banach spaces $(\mathcal{X}_i)_{i\in \Z^m}$ and fix $j\in \Z^m$. We define the immersion map
\[\mathtt{I}_j\colon \mathcal{X}_j \to \ell^{\infty}(\mathcal{X}_i), \qquad \big(\mathtt{I}_j(v)\big)_j=v, \quad \big(\mathtt{I}_j(v)\big)_i=0\,\,\,\,i\neq j\]
and the projection
\[
\pi_j \colon \ell^{\infty}(\mathcal{X}_i)\to \mathcal{X}_j, \qquad \pi_j(z)=z_j.
\]
Given $A\in\mathcal{L} (\ell^{\infty}(\mathcal{X}_i); \ell^{\infty}(\mathcal{Y}_i))$ we define
\[A^i_j v:=\pi_i (A \mathtt{I}_j(v)) \qquad \forall v\in \mathcal{X}_j, \quad \forall i, j\in \Z^m.\]
In finite dimension this would coincide with the representation of $A$ as a matrix with entry $(i, j)$ given by $A^i_j$. We remark that linear operators acting on $\ell^{\infty}$ spaces cannot be always represented as matrices. However the following lemma shows that if they act on decaying sequences this is the case.

\begin{lem}[Lemma $2.6$ in \cite{FontMartin1}]
Let $A\in \mathcal{L} (\ell^{\infty}(\mathcal{X}_i); \ell^{\infty}(\mathcal{Y}_i))$ and $v\in \ell^{\infty}(\mathcal{X}_i)$ be such that $\lim_{|j| \to \infty} |v_j|=0$. Then
\[
(A v)_i=\sum_{j\in \Z^m} A_j^i v_j.
\]
\end{lem}
The set of linear invertible operators with decay is not a subalgebra of $\mathrm{Gl}(\ell^{\infty}(\mathcal{X}_i))$. However for small perturbations of invertible operator with decay and whose inverse also has decay we have the following classical result.

\begin{lem}[Neumann series]\label{lem:neumann}
Let $A\in\mathcal{L}_{\Gamma}(\ell^{\infty}(\mathcal{X}_i))$ be invertible and such that
\[
A^{-1}\in\mathcal{L}_{\Gamma}(\ell^{\infty}(\mathcal{X}_i)).
\]
Let $B\in\mathcal{L}_{\Gamma}(\ell^{\infty}(\mathcal{X}_i))$ such that $\| A^{-1} \|_{\mathcal{L}_{\Gamma}} \| B \|_{\mathcal{L}_{\Gamma}}<1$. Then $M:=A+B\in \mathcal{L}_{\Gamma}(\ell^{\infty}(\mathcal{X}_i))$ is invertible, $M^{-1}\in\mathcal{L}_{\Gamma}(\ell^{\infty}(\mathcal{X}_i))$ and
\[
|  \| M^{-1} \|_{\mathcal{L}_{\Gamma}}-\| A^{-1} \|_{\mathcal{L}_{\Gamma}}  |\le \| M^{-1}-A^{-1} \|_{\mathcal{L}_{\Gamma}}=\mathcal{O}(\| B \|_{\mathcal{L}_{\Gamma}}).
\]
\end{lem}
\begin{proof}
It is a direct consequence of classical Neumann series for bounded operators and the algebra property of $\mathcal{L}_{\Gamma}$.
\end{proof}

\subsection{Multilinear maps and $C^r$ functions}\label{sec:multilineardecay}
Let $k\geq 1$ and $\mathcal{X}_i^{(j)}$, $j=1, \dots, k$, be $k$-sequences of Banach spaces. We introduce the space of $k$-linear maps with decay $\Gamma$

\begin{align}
&\mathcal{L}_{\Gamma}^k (\ell^{\infty}(\mathcal{X}^{(1)}_i)\times \dots \times \ell^{\infty}(\mathcal{X}^{(k)}_i); \ell^{\infty}(\mathcal{Y}_i)):=\Big\{  A\in \mathcal{L}^k (\ell^{\infty}(\mathcal{X}^{(1)}_i)\times \dots \times \ell^{\infty}(\mathcal{X}^{(k)}_i); \ell^{\infty}(\mathcal{Y}_i)) :  \label{def:klinearmaps}  \\ \nonumber
& \imath_m(A)\in \mathcal{L}_{\Gamma} (\ell^{\infty}(\mathcal{X}^{(m)}_i); \ell^{\infty}(\mathcal{L}^{k-1} (\ell^{\infty}(\mathcal{X}^{(1)}_i)\times \dots \times \widehat{\ell^{\infty}(\mathcal{X}_i^{(m)})} \times  \dots \times \ell^{\infty}(\mathcal{X}^{(k)}_i); \mathcal{Y}_i))),\,\,m=1, \dots, k \Big\}
\end{align}
where the symbol $\widehat{(\cdot)}$ denotes that the term $(\cdot)$ is missing in the product and $\imath_m$ is defined by
\[
\imath_m (A) (v) (u_1, \dots, u_{m-1}, u_{m+1}, \dots, u_k) :=A(u_1, \dots, u_{m-1}, v , u_{m+1}, \dots, u_k).
\]
The space $\mathcal{L}_{\Gamma}^k (\ell^{\infty}(\mathcal{X}_i); \ell^{\infty}(\mathcal{Y}_i))$ is a Banach space with the norm
\[
\| A \|_{\mathcal{L}^k_{\Gamma}}=\max\{ \| A \|, \gamma(A) \}
\]
where
\[
\gamma(A)=\max_{1\le m\le k} \{ \gamma(\imath_m(A)) \}.
\]
This definition allows us to introduce also the set of (nonlinear)  $C^r$ maps with decay between $\ell^{\infty}$ spaces.

Given an open subset $\mathcal{U}$ of $\ell^{\infty}(\mathcal{X}_i)$ let
\begin{align*}
C^1_{\Gamma}:=C^1_{\Gamma}(\mathcal{U}; \ell^{\infty}(\mathcal{Y}_i)):=\{ &F\in C^1(\mathcal{U}; \ell^{\infty}(\mathcal{Y}_i)) : DF(x)\in \cL_{\Gamma}, \,\,\forall x\in\mathcal{U},\,\,\| F \|_{C^1_{\Gamma}}<\infty \}
\end{align*}
with 
\[
\| F \|_{C^1_{\Gamma}}:=\max \left\{  \| F \|_{C^0}, \sup_{x\in\mathcal{U}} \| DF (x) \|_{\mathcal{L}_{\Gamma}} \right\}.
\]
We point out that, by definition, the derivatives of a map $F\in C^1_{\Gamma}$ are uniformly bounded on $\mathcal{U}$.

For $r>1$ we define
\[
C^r_{\Gamma}:=C^r_{\Gamma}(\mathcal{U}; \ell^{\infty}(\mathcal{Y}_i)):=\left\{F\in C^r(\mathcal{U}; \ell^{\infty}(\mathcal{Y}_i)) : D^j F\in C^1_{\Gamma}, \,\,0\le j \le r-1 \right\}
\]
with the norm
\begin{equation}\label{def:Crnorm}
\| F \|_{C^r_{\Gamma}}:=\max\left\{  \| F \|_{C^0}, \max_{0\le j\le r-1}\,\, \sup_{x\in\mathcal{U}} \| D\, D^j\,F (x) \|_{\mathcal{L}_{\Gamma}}\right\}=\max_{0\le j\le r-1} \| D^j F\|_{C^1_{\Gamma}}.
\end{equation}
The next two lemmas analyze the behavior of  $C^r_{\Gamma}$ maps under composition and the limit of certain sequences in $C^r_{\Gamma}$.

\begin{lem}[Proposition $2.17$ in \cite{FontMartin1}]
Let $\mathcal{X}_i, \mathcal{Y}_i, \mathcal{Z}_i$ be sequences of Banach spaces.
Let $\mathcal{U}\subset \ell^{\infty}(\mathcal{X}_i)$ and $\mathcal{V}\subset \ell^{\infty}(\mathcal{V}_i)$ be open sets. Then, if $F\in C^r_{\Gamma}(\mathcal{U}; \ell^{\infty}(\mathcal{Y}_i))$, $G\in C^r_{\Gamma}(\mathcal{V}; \ell^{\infty}(\mathcal{Z}_i))$ and $F(\mathcal{U})\subset \mathcal{V}$ then $G\circ F\in C^r_{\Gamma}(\mathcal{U}; \ell^{\infty}(\mathcal{Z}_i))$. Moreover
\[
\| G\circ F \|_{C^r_{\Gamma}}\le K_r (1+\| F \|^r_{C^r_{\Gamma}})\,\| G \|_{C^r_{\Gamma}}
\]
for some constant $K_r>0$ independent of $F$ and $G$.
\end{lem}



\begin{lem}[Lemma $2.14$ in \cite{FontMartin1}]\label{lem:limitdecay}
Let $\mathcal{U}$ be an open subset of $\ell^{\infty} (\mathcal{X}_i)$ and let $B_{\rho}$ be the closed ball of radius $\rho$ in $C^r_{\Gamma} (\mathcal{U}; \ell^{\infty}(\mathcal{Y}_i))$. Assume that $(F_n)_{n\geq 0}$ is a sequence in $B_{\rho}$ and for all $0\le k\le r$, $x\in \mathcal{U}$, $D^k F_n(x)$ converges in the sense of $k$-linear maps to $D^k F(x)$, where $F$ is a $C^r$ function defined on $\mathcal{U}$. Then $F\in B_{\rho}$.  
\end{lem}



%

\subsection{First integrals}\label{sec:firstint}
In this paper we consider maps and vector fields acting on Banach spaces $\ell^{\infty}(\mathcal{X}_i)$. Some of these maps and vector fields will have first integrals. However, these first integrals may only be formal in the sense that they are unbounded. This happens for the Hamiltonian introduced in  \eqref{def:Hamiltonian}: even if it defines a Hamiltonian vector field, it is unbounded in  $\ell^{\infty}(\mathcal{X}_i)$.

We devote this section to properly define the notion of formal first integral both for maps and flows.

\begin{definition}\label{def:firstintegralmaps}
Consider a $C_\Gamma^1$ map $F\colon \mathcal{U}\subset\ell^{\infty}(\mathcal{X}_i)\to \ell^{\infty}(\mathcal{X}_i)$, where $\mathcal{U}$ is an open subset. Then, 
\[
 G:\mathcal{U}\to \R\cup\{\infty\}
\]
is a formal first integral $G$ for the map $F$ if it has the following properties.
\begin{itemize}
\item For all $z\in \mathcal{U}$, $i\in\mathbb{Z}^m$ and  $h\in\mathcal{X}_i$ with $\|h\|_{\mathcal{X}_i}$ small enough, the function 
\[
M(z,h)=G(z+\mathtt{I}_i(h))-G(z)
\]
is well defined, $|M(z,h)|<\infty$ and it is $C^1$ in all its variables.

\item As a consequence, for all $z\in \mathcal{U}$, $dG(z)$ is well defined, is continuous and  satisfies $dG(z)\in T_z^*\ell^{\infty}(\mathcal{X}_i)$ and
\item For all $z\in \mathcal{U}$, it satisfies
\begin{equation}\label{cond:FI:maps}
 DG(F(z))DF(z)=DG(z).
\end{equation}
\end{itemize}
\end{definition}
One can state an analogous definition of formal first integral for vector fields.

\begin{definition}\label{def:firstintegralflows}
Consider a $C_\Gamma^1$ vector field $X\colon \mathcal{U}\subset\ell^{\infty}(\mathcal{X}_i)\to \ell^{\infty}(\mathcal{X}_i)$. Then, 
\[
 G:\mathcal{U}\to \R\cup\{\infty\}
\]
is a formal first integral $G$ for the vector field $X$ if it has the following properties.
\begin{itemize}
\item For all $z\in \mathcal{U}$, $i\in\mathbb{Z}^m$ and  $h\in\mathcal{X}_i$ with $\|h\|_{\mathcal{X}_i}$ small enough, the function 
\[
M(z,h)=G(z+\mathtt{I}_i(h))-G(z)
\]
is well defined, $|M(z,h)|<\infty$ and it is $C^1$ in all its variables.

\item As a consequence, for all $z\in \mathcal{U}$, $dG(z)$ is well defined and  satisfies $dG(z)\in T_z^*\ell^{\infty}(\mathcal{X}_i)$ and
\item For all $z\in \mathcal{U}$, it satisfies
\[ DG(z)X(z)=0.\]
\end{itemize}
\end{definition}

\begin{remark}\label{rmk:Frobenius}
Note that the even if these definitions admit that the first integrals are just formal, they still define a codimension 1 foliation on the open sets $\mathcal U\subset\ell^{\infty}(\mathcal{X}_i)$ where $dG\neq 0$. This is a consequence of the classical Frobenius Theorem which also applies to Banach spaces (see for instance Chapter VI of \cite{Lang95}). Indeed, it is easy to check that the distribution $\Ker(dG(z))\subset T_z \mathcal U$ is integrable.
\end{remark}

\subsection{Sequences with decay: the subspace $\Sigma_{j, \Gamma}$}\label{sec:decayspaces}

Fix $j\in\mathbb{Z}^m$. We introduce the subspace of $\ell^{\infty}(\mathcal{X}_i)$ of vectors centered around the $j$-th component
\begin{equation}
\Sigma_{j, \Gamma}:=\{ v\in \ell^{\infty}(\mathcal{X}_i) : \| v \|_{j, \Gamma}<\infty \}
\end{equation}
where
\[
\| v \|_{j, \Gamma}:=\sup_{k\in\Z^m} |v_k| \,\Gamma(k-j)^{-1}.
\]
Note that for any $i,j\in\Z^m$, $\Sigma_{i, \Gamma}=\Sigma_{j, \Gamma}$ and their norms are equivalent as 
\[
 \| v \|_{i, \Gamma}\leq \| v \|_{j, \Gamma}\Gamma(i-j)^{-1},
\]
although the equivalence ``blows up'' as $|i-j|\to\infty$.

%

\begin{lem}[Proposition $2.7$ in \cite{FontMartin1}]\label{lem:fm1}
Let $A\in\mathcal{L} (\ell^{\infty}(\mathcal{X}_i); \ell^{\infty}(\mathcal{Y}_i))$.
\begin{enumerate}
\item If $A\in\mathcal{L}_{\Gamma} (\ell^{\infty}(\mathcal{X}_i); \ell^{\infty}(\mathcal{Y}_i))$, then for any $j\in \Z^m$ and for any $v\in \Sigma_{j, \Gamma}$, $A v\in \Sigma_{j, \Gamma}$ and $\| A v \|_{j, \Gamma}\le \gamma(A) \| v \|_{j, \Gamma}$.
\item If  there exists $C>0$ such that for any $j\in \Z^m$ and for any $v\in \Sigma_{j, \Gamma}$, $A v\in \Sigma_{j, \Gamma}$ and $\| A v \|_{j, \Gamma}\le C \| v \|_{j, \Gamma}$, then $A\in \mathcal{L}_{\Gamma}(\ell^{\infty}(\mathcal{X}_i); \ell^{\infty}(\mathcal{Y}_i))$ and $\gamma(A)\le C \Gamma(0)^{-1}$.
\end{enumerate}
Moreover, if $F\in C^1_{\Gamma}$ and $F(0)=0$ then $F(v)\in \Sigma_{j, \Gamma}$ for all $v\in \Sigma_{j, \Gamma}$. 
\end{lem}

\subsection{Partial action-angle variables and the adapted functional setting}\label{sec:functsettorus}

In this section we develop a functional setting adapted to a set of coordinates that we shall use to study the dynamics close to the invariant tori of the transition chain.

\smallskip

Let $S$ be a subset of $\Z^m$ with cardinality $d$.
We use the following notation
\begin{equation}\label{def:linftysc}
\begin{split}
&\ell^{\infty}_{S^c}:=\ell^{\infty}(\Z^m\setminus S; \R), \qquad \Sigma_{j, \Gamma, S^c}:=\Sigma_{j, \Gamma} (\Z^m\setminus S; \R),\\
& \R_S^d:=\ell^{\infty}(S; \R), \qquad \qquad \qquad \, \T^d_S:=\ell^{\infty}(S; \T).
\end{split}
\end{equation}

Let us consider the complete metric space
\[
\mathcal{M}:=\ell^{\infty}_{S^c}\times \ell^{\infty}_{S^c}\times \T_S^d\times \R_S^d
\]
which is a Banach manifold modeled on 
\[
\ell^{\infty}_{S^c}\times \ell^{\infty}_{S^c}\times \R_S^d\times \R_S^d.
\]
We consider the following coordinates $(x, y, \theta, r)$ on $\mathcal{M}$
\[
 x=(x_j)_{j\in \Z^m\setminus S}, \quad y=(y_j)_{j\in \Z^m\setminus S}, \quad \theta=(\theta_j)_{j\in S}\in \T_S^d, \quad r=(r_j)_{j\in S}\in \R_S^d.
\]
Such coordinates are useful to study the dynamics close to the invariant tori contained in the subspace $\mathcal{V}_S$ in \eqref{def:subsp}.

 We denote by $\pi_{x}, \pi_y$ the projections 
\[
\pi_{x}(x, y, \theta, r)=x, \quad \pi_y (x, y, \theta, r)=y.
\]
Consider linear operators 
\begin{itemize}
\item[(i)]
$
A \colon \ell^{\infty}_{S^c}\to \ell^{\infty}_{S^c};
$
\item[(ii)]
$
A \colon \ell^{\infty}_{S^c}\to \R^d_S;
$
\item[(iii)]
$
A \colon \R^d_S\to \R^d_S;
$
\end{itemize}
then we define respectively
\begin{itemize}
\item[(i)]
$
\gamma(A):=\sup_{i, j\in\mathbb{Z}^m\setminus S} \,\,\, \sup_{\|u\|_{\ell^{\infty}}\le 1} |(A \mathtt{I}_j (u))_i|\,\Gamma(i-j)^{-1};
$
\item[(ii)]
$
\gamma(A):=\sup_{i\in S, j\in\mathbb{Z}^m\setminus S} \,\,\, \sup_{\|u\|_{\ell^{\infty}}\le 1} |(A \mathtt{I}_j (u))_i|\,\Gamma(i-j)^{-1};
$
\item[(iii)]
$
\gamma(A):=\sup_{i, j\in S} \,\,\, \sup_{\|u\|_{\ell^{\infty}}\le 1} |(A \mathtt{I}_j (u))_i|\,\Gamma(i-j)^{-1}.
$
\end{itemize}
Hence we have a definition of operators with decay for linear maps of the form (i), (ii), (iii) and, similarly to (ii), for maps from $\R^d_S$ to $\ell^{\infty}_{S^c}$, by considering the norm \eqref{def:decaynorm} with the semi-norm $\gamma(A)$ introduced above.

Now consider a linear operator 
\[
A\colon  \left(\ell^{\infty}_{S^c} \right)^{\mathtt{a}} \times ( \R_S^d )^{\mathtt{b}}\to \ell^{\infty}(\mathcal{Y}_i):=\ell^{\infty}(\Z^m; \mathcal{Y}_i),\qquad \mathtt{a}, \mathtt{b}\in \{ 1, 2\}, 
\]
 where $\mathcal{Y}_i$ is a sequence of Banach spaces.  
We define
\[
T_{S, \mathtt{a}, \mathtt{b}}\colon \left(\ell^{\infty}_{S^c} \right)^{\mathtt{a}} \times ( \R_S^d )^{\mathtt{b}}\to \ell^{\infty}(\Z^m; \mathcal{X}_i)=:\ell^{\infty}(\mathcal{X}_i)
\]
where
\[
\mathcal{X}_i:=\begin{cases}
\R^{\mathtt{a}} \qquad \mathrm{if}\,\,i\in \Z^m\setminus S\\
\R^{\mathtt{b}}  \qquad\,\,\, \mathrm{if}\,\,i\in S
\end{cases}
\]
and
\[
\Big(T_{S, \mathtt{a}, \mathtt{b}}(x, y, r)\Big)_j:=\begin{cases}
(x_j, y_j) \qquad \mathrm{if}\,\,j\in \Z^m\setminus S,\\
r_j \qquad \,\,  \qquad\mathrm{if}\,\,j\in S.
\end{cases}
\]
Thus we say that $A\in \mathcal{L}_{\Gamma}(\left(\ell^{\infty}_{S^c} \right)^{\mathtt{a}} \times ( \R_S^d )^{\mathtt{b}},\, \ell^{\infty}(\mathcal{Y}_i))$ if 
\[
\|A\circ T_{S, \mathtt{a}, \mathtt{b}}\|_{ \mathcal{L}_{\Gamma}(\ell^{\infty}(\mathcal{X}_i), \ell^{\infty}(\mathcal{Y}_i))}<\infty
\]
and we set 
\[
\| A \|_{\mathcal{L}_{\Gamma}\left(\left(\ell^{\infty}_{S^c} \right)^{\mathtt{a}} \times ( \R_S^d )^{\mathtt{b}}, \ell^{\infty}(\mathcal{Y}_i) \right)}=\|A\circ T_{S, \mathtt{a}, \mathtt{b}}\|_{ \mathcal{L}_{\Gamma}(\ell^{\infty}( \mathcal{X}_i), \ell^{\infty}(\mathcal{Y}_i))}.
\]

One can proceed analogously for operators defined on the tangent space of the submanifold
\[
\mathcal{M}_{j, \Gamma}:=\Sigma_{j, \Gamma, S^c}\times \Sigma_{j, \Gamma, S^c}\times \T^d_S \times \R_S^d,
\]
that is on $\Sigma_{j,\Gamma}$ instead of $\ell^\infty$ spaces. Moreover, in the spaces with decay one has the following properties for nonlinear maps.

\begin{lem}\label{lemma:SigmajGammazero}
Let $\mathcal{U}\subset\mathcal{M}$ be an open subset containing the torus $x=y=r=0$  and consider a map $F\in C^1_{\Gamma}(\mathcal{U})$ such that $F(0, 0, \theta, 0)=0$ for all $\theta\in \mathbb{T}_S^d$. Then, $F$ maps $U\cap \mathcal{M}_{j, \Gamma}$ into $\mathcal{M}_{j, \Gamma}$.
\end{lem}

\begin{proof}
Let $(x, y, \theta, r)\in \mathcal{M}_{j, \Gamma}$. By the mean value theorem we have
\begin{align*}
\| F(x, y, \theta, r) \|_{j, \Gamma}&=\| F(x, y, \theta, r)-F(0, 0, \theta, 0) \|_{j, \Gamma}\le \int_0^1 \|\partial_{(x, y, r)} F (t x, t y, \theta, t r) [x, y, r]\|_{T \mathcal{M}_{j, \Gamma}}\,dt\\
& \le \| D F \|_{\mathcal{L}_{\Gamma}(T \mathcal{M})} \| (x, y, \theta, r) \|_{j, \Gamma}.
\end{align*}

\end{proof}

\section{A  detailed statement of the main results}
\label{sec:DescriptionProof}

In this section we state a theorem which  implies the results in Theorems \ref{thm:MainFormal} and \ref{thm:MainEnergy}. In fact it is stronger since it contains the  concrete hypotheses that  the perturbation $H_1$ must satisfy so that it leads to transfer of energy orbits. 

\smallskip

Given $\rho>0$, we define the balls
 \[
 \begin{split}
 B_{\rho}(\ell^{\infty})&=\left\{z\in\ell^\infty(\ZZ^m,M) : \|z\|_{\ell^\infty}\leq \rho\right\}\subset \ell^\infty(\ZZ^m;M), \quad M:=\T\times \R\\
 B_{\rho}(\Sigma_{j, \Gamma})&=\left\{z\in \Sigma_{j, \Gamma} : \|z\|_{j,\Gamma}\leq \rr\right\}\subset  \Sigma_{j, \Gamma}.
 \end{split}
 \]
Consider a Hamiltonian of the form \eqref{def:Hamiltonian}, fix a sequence
 \[
  \{\sigma_i\}_{i\geq 0}\subset \mathbb{Z}^m,\qquad |\sigma_{i+1}-\sigma_i|=1
 \]
 and define 
\[
S_i=\left\{\sigma_i,\sigma_{i+1}\right\}\subset\ZZ^m.
\]

Let us define the subspaces
 \begin{equation}\label{def:subspacesXj}
 \cV_i :=\left\{q_k=p_k=0\qquad\text{for }\quad k\neq \sigma_i,\sigma_{i+1},\sigma_{i+2}\right\},
\end{equation}
which will be assumed to be invariant,
and the Hamiltonian
\begin{equation}\label{def:Hamj}
 H_{1,i}\left(q_{\sigma_{i}},p_{\sigma_{i}},q_{\sigma_{i+1}},p_{\sigma_{i+1}},q_{\sigma_{i+2}},p_{\sigma_{i+2}}\right)=H_1(q,p)|_{\cV_i}.
\end{equation}
We consider the following hypotheses.

\begin{itemize}
 \item[$\mathbf{H1}$] For any  $\rho>0$,   $X_{H_1}\in C_\Gamma^3(B_{\rho}(\ell^{\infty}); \ell^{\infty})$\footnote{Note that being $C_\Gamma^r$ implies that the associated norm is bounded (see \eqref{def:Crnorm}). For this reason, the hypothesis must be stated restricted to the ball where $DH_1$ has uniform estimates.}. 

\item[$\mathbf{H2}$] For any $i\geq 0$, the Hamiltonian $H_1$ satisfies
\[
\left.\partial_{q_k} H_1(q,p)\right|_{(q_k,p_k)=(0,0)}=\left.\partial_{p_k} H_1(q,p)\right|_{(q_k,p_k)=(0,0)}=0 \qquad \forall k\neq \sigma_i.
\]

Moreover, for any $i\geq 0$, 
\[
 \left.
\partial_{q_k} H_1(q,p)
 \right|_{S_i}
 =\left.\partial_{p_k} H_1(q,p)\right|_{S_i}=0\qquad \text{for}\qquad k=\sigma_i,\sigma_{i+1}.
\]
\item[$\mathbf{H3}$] Fix $h>0$. For any value  {$h_i\in (0,h)$} there exists an open set $\mathcal{J}_{h_i}\subset \mathbb{T}^2$ with the property that when $(x_i, x_{i+1}, h_i)\in \mathcal{J}_+$, where
\[
\mathcal{J}_+=\mathcal{J}_{h_i}\times \{h_i\}.
\]
and  $h_{i+1}=h-h_i$,
\begin{itemize}
\item[$\mathbf{H3.1}$]
Consider the Melnikov potential
\begin{equation}\label{def:MelnikovPotential}
\begin{aligned}
\MoveEqLeft \mathcal{L}_i(x_i, x_{i+1}, h_{i}, h_{i+1},t)=\\
&\int_{-\infty}^{+\infty}H_{1, i}\left(q_{h_{i}}(s,x_i),p_{h_{i}}(s,x_i),q_{h_{i+1}}(s,x_{i+1}),p_{h_{i+1}}(s,x_{i+1}),q_0(s+t),p_0(s+t)\right)ds
\end{aligned}
\end{equation}
associated to the homoclinic of the torus
\begin{equation}\label{def:invtori}
\begin{split}
\TT_{\sigma_i,\sigma_{i+1},h_i,h_{i+1}}&=\{  E_{\sigma_{i}} =h_i, E_{\sigma_{i+1}} =h_{i+1}\,\text{ and }\, E_k=0\, \text{ for }\, k\neq \sigma_{i},\sigma_{i+1}\}.
\end{split}
\end{equation}
The map 
\[
t\in\mathbb{R}\to \mathcal{L}_j(x_i, x_{i+1}, h_{i}, h_{i+1},t)
\]
has a non-degenerate critical point $t$ which is locally  given by the implicit function theorem in the form $t=\tau(x_i, x_{i+1}, h_{i}, h_{i+1})$.
 
\item[$\mathbf{H3.2}$] Consider the Melnikov function
\begin{equation*}
\begin{aligned}
\MoveEqLeft  \mathcal{M}_i(x_i, x_{i+1}, h_{i}, h_{i+1},t) =\\
&\int_{-\infty}^{+\infty}\left\{H_{1, i},E_{i}\right\}\left(q_{h_{i}}(s,x_i),p_{h_{i}}(s,x_i),q_{h_{i+1}}(s,x_{i+1}),p_{h_{i+1}}(s,x_{i+1}),q_0(s+t),p_0(s+t)\right)ds
\end{aligned} 
\end{equation*}
associated to the homoclinic of the torus $\TT_{\sigma_i,\sigma_{i+1},h_i,h_{i+1}}$.

 The map
\[
 (x_i, x_{i+1})\in {\mathcal{J}_{h_1}} \subset \mathbb{T}^{{2}}\to \mathcal{M}_i(x_i,x_{i+1},h_{i},h_{i+1},\tau(x_i, x_{i+1}, h_{i}, h_{i+1}))
\]
is nonconstant and positive. 

Analogously, we assume that there exists a set $\mathcal{J}_-$ where the same hypohtesis is true but $\mathcal{M}_i$ is negative.
\end{itemize}

\item[$\mathbf{H4}$] Fix $h>0$. Consider the Melnikov potential
\begin{equation}\label{def:MelnikovPotential:per}
\widetilde{\mathcal{L}}_i (x_i, h,t_1, t_2)=\int_{-\infty}^{+\infty}H_{1, i}\left(q_{h}(s,x_i),p_{h}(s,x_i),q_{0}(s+t_1),p_{0}(s+t_1),q_0(s+t_2),p_0(s+t_2)\right)dt
\end{equation}
associated to the homoclinic of the periodic orbit
\begin{equation}\label{def:invtori1}
\PP_{\sigma_i}=\{   E_{\sigma_{i}}=h\,\text{ and }\, E_k=0\, \text{ for }\, k\neq \sigma_{i}\}.
\end{equation}
The map 
\[
(t_1,t_2)\in\mathbb{R}^2\to \widetilde{\mathcal{L}}_i(x_i, h,t_1, t_2)
\]
has a non-degenerate critical point $(t_1,t_2)$ which is locally  given by the implicit function theorem in the form ${(t_1, t_2)} =\widetilde\tau(x_i, h)$.

\end{itemize}

Note that the Hypotheses $\mathbf{H3}$ is the same as considered in the paper \cite{DelshamsLS06a} to prove Arnold diffusion  for nearly integrable Hamiltonian a priori unstable systems. The Hypotheses $\mathbf{H4}$ is the analog for the jumping regime (see Section \ref{sec:heuristics}). It is well known that they are $C^r$ (with $r\geq 4$) and $C^\infty$ generic. They are also generic in the analytic setting if one considers sufficiently small $h>0$ (see \cite{ChenLlave}). On the contrary the Hypothesis $\mathbf{H2}$ requires that certain subspaces are invariant and the dynamics on them is integrable. Unfortunately this hypothesis is not generic.



A family of perturbations which satisfy the hypohteses above when $m=1$ and for any instability path are those of the form
\[
H_1(q, p)=\sum_{N\geq 2}a_N\sum_{i\in\mathbb{Z}}f_N(q_{k},p_k,\ldots, q_{k+N},p_{k+N})\prod_{i=0}^N(1-\cos q_{k+i})^{\alpha_k} p_k^{\beta_k}.
\]
where 
\begin{itemize}
\item $a_N\in\mathbb{R}$ satisfy $|a_N|\lesssim \Gamma(N)$ such that $H_1$ has enough decay. Moreover $a_2\neq 0$, which is a necessary condition  so that $\mathbf{H3}$ and  $\mathbf{H4}$ are satisfied.
\item $\alpha_k\geq 1$ or $\beta_k\geq 2$ such that Hypothesis $\mathbf{H2}$ is satisfied.
\item $f_N$ is at least $C^r$, $r\geq 4$, and satisfies $\|f_N\|_{C^r}\leq 1$. The function $f_3$  is chosen generically so that  $\mathbf{H3}$ and  $\mathbf{H4}$ are satisfied. There is no extra requirement on $f_N$, $N\geq 4$.
\end{itemize}
Note that we have been able to choose a translation invariant example of perturbation. However the hypohteses above do not impose this restriction.

Theorems \ref{thm:MainFormal} and \ref{thm:MainEnergy} are consequence of the following.

\begin{theorem}\label{thm:technical}
Consider a Hamiltonian of the form \eqref{def:Hamiltonian} and assume that $H_1$ satisfies the  Hypotheses $\mathbf{H1}-\mathbf{H4}$. Then there exists $\e_0>0$ such that for all $\e\in (0, \e_0)$, for any $\eta>0$ small enough and any sequence
 \[
  \{\sigma_i\}_{i\geq 0}\subset \mathbb{Z}^m,\qquad |\sigma_{i+1}-\sigma_i|=1
 \]
 the following holds:
 \begin{itemize}
 \item There exist trajectories $(q(t),p(t))\in \ell^{\infty}(\Z^m; M)$ and an increasing sequence of times $\{t_i\}_{i\geq 0}$ such that
 \[
|  E_{\sigma_i}(q(t_i),p(t_i))-h|\leq\eta\qquad\text{and}\qquad |  E_{k}(q(t_i),p(t_i)|\leq\eta \quad\text{for}\quad k\neq \sigma_i.
 \]
\item For any fixed $j\in \Z^m$, there exist trajectories $(q(t),p(t))\in \Sigma_{j, \Gamma}$, which therefore satisfy $H(q(t), p(t))<\infty$ for all $t$, and an increasing sequence of times $\{t_i\}_{i\geq 0}$ such that
 \[
|  E_{\sigma_i}(q(t_i),p(t_i))-h|\leq\eta\qquad\text{and}\qquad |  E_{k}(q(t_i),p(t_i)|\leq\eta \quad\text{for}\quad k\neq \sigma_i.
 \]
\end{itemize}
\end{theorem}


We devote the rest of this section to describe the main steps  of the proof  of this theorem.



%

\subsection{Description of the proof of Theorem \ref{thm:technical}}
\subsubsection{Invariant manifolds.} 

The first step is that the flow associated to the Hamiltonian \eqref{def:Hamiltonian} fits the functional setting given in Section \ref{sec:functional}.

\begin{lem}\label{lem:welldefflow}
Consider the Hamiltonian $H$ in \eqref{def:Hamiltonian} and assume that it satisfies Hypothesis $\mathbf{H1}$. Fix any $\rr>0$. 
Then there exists $T>0$ such that for any  initial conditions $(q_0, p_0)\in   B_{\rho}(\ell^{\infty})\subset \ell^\infty(\ZZ^m;M)$,  there is a unique solution $(q(t), p(t))$ of the Cauchy problem associated to \eqref{eq:motions} defined for $|t|<T$.

Moreover denoting by $\Phi^t_H(q_0, p_0)=(q(t), p(t))$, we have $\Phi^t_H\in C^r_{\Gamma}( B_{\rho}(\ell^{\infty}))$ for all $|t|<T$ and there exist $C, \mu>0$ such that
\[
\| D \Phi^t_H(q, p) \|_{\Gamma}\leq C e^{\mu t}, \qquad q, p\in B_{\rho}(\ell^{\infty}), \qquad t\in (-T, T).
\] 
Moreover, fix $j\in\Z^m$. Then, if $q_0,p_0\in B_{\rho}(\Sigma_{j, \Gamma})$, 
one has that,  for $t\in (-T, T)$, $(q(t), p(t))\in \Sigma_{j, \Gamma}$.
\end{lem}

\begin{proof}
Since $X_H\in C^r_{\Gamma}(\mathcal{U})$ for any open subset $\mathcal{U}$ of the phase space $\ell^{\infty}(\Z^m;M)$ the proof follows by Proposition $8. 1$ in \cite{FontichLS15}.
\end{proof}

Once we know that the flow $\Phi_H^t$ is well defined both in  $ \ell^\infty(\ZZ^m;M)$ and in $\Sigma_{j, \Gamma}$, we can start developing an invariant manifolds theory for the invariant tori of the transition chain (see Section \ref{sec:heuristics}).

Recall that we have considered an  ``instability path'' 
\[
  \{\sigma_i\}_{i\geq 0}\subset \mathbb{Z}^m,\qquad |\sigma_{i+1}-\sigma_i|=1,
\]
and have defined the associated  sets of sites
\[
S_i=\left\{\sigma_i,\sigma_{i+1}\right\}\subset\ZZ^m.
\]


The Hypothesis $\mathbf{H2}$ implies that certain invariant tori of the unperturbed Hamiltonian \eqref{def:Hamiltonian} with $\eps=0$ are preserved.
In particular, this is case for the tori $\TT_{\sigma_i,\sigma_{i+1},h_1,h_2}$ and $\PP_{\sigma_i}$ introduced in \eqref{def:invtori} and \eqref{def:invtori1} respectively, which
are invariant under the flow associated to $H$ and the flow on these tori is a rigid rotation given by the integrable dynamics of \eqref{def:Hamiltonian} with $\eps=0$.

These tori have stable and unstable invariant manifolds which are, moreover, smooth with respect to parameters. This is stated in Theorem \ref{thm:toriinvman} below, which is a consequence of a more general invariant manifold theorem for invariant tori on lattices. This more general theory is explained in  Section \ref{sec:manifold}. 

To state  Theorem \ref{thm:toriinvman}, we introduce first good coordinates which allow to parameterize the invariant manifolds of the tori in \eqref{def:invtori}, \eqref{def:invtori1} as graphs.

%
%
%
To deal at the same time with the tori $\TT_{\sigma_i,\sigma_{i+1},h_1,h_2}$ and $\PP_{\sigma_i}$, we call $d$ to the dimension of the tori (which is either $d=2$ or $d=1$), $S$ to the ``activated sites'', that is either $S=\left\{\sigma_i, \sigma_{i+1}\right\}$ or $S=\left\{\sigma_i \right\}$ and we denote the torus by $\T_0$. Note that for  $\TT_{\sigma_i, \sigma_{i+1},h_1,h_2}$ we are assuming $h,h_1,h_2>0$, $h_1+h_2=h$ and for $\PP_{\sigma_j}$ we are assuming $h>0$.

Then, for $\de>0$ small enough,  we define the coordinates
\begin{equation*}
(x, y, \theta, r)\in \mathcal{M}_{\delta}:= B_{\delta}(\ell^{\infty}(\Z^m\setminus S; \R))\times B_{\delta}(\ell^{\infty}(\Z^m\setminus S; \R))\times \T^d\times B_{\delta}(\R^d)
\end{equation*}
 defined in a $\de$-neighborhood of $\T_0$, where
\begin{itemize}
\item  $(\theta_k, r_k)$ are the action-angle variables
 that are well defined in a neighborhood of the torus $\{ E_k=h_k \}$ for $h_k>0$ located at the tangential site $k\in S$.
\item  $(x_k, y_k)$ are cartesian coordinates which diagonalize the linearization of $E_k$  at the saddle $x_k=y_k=0$. That is,
\[
x_k=q_k+p_k, \qquad y_k=p_k-q_k, \qquad k\notin S.
\]
\end{itemize}

In these variables the equations of motion \eqref{eq:motions} are of the form
\begin{equation}\label{system}
\begin{cases}
\dot{x}_k=x_k+\mathbf{f}^k_1(\e; x, y, \theta, r) \qquad k\notin S,\\
\dot{y_k}=-y_k+\mathbf{f}^k_2(\e; x, y, \theta, r)\\
\dot{\theta}_k=\omega_k(r)+\mathbf{f}_3^k(\e; x, y, \theta, r) \qquad k\in S,\\
\dot{r}_k= \mathbf{f}_4^k(\e; x, y, \theta, r)
\end{cases}
\end{equation}
where $\omega_k(r)$ is the frequency associated to integrable Hamiltonian $E_k$ and
\[
\mathbf{f}^k_1(\e; x, y, \theta, r)=\mathbf{f}^k_2(\e; x, y, \theta, r):=\sin\left(\frac{x_k-y_k}{2} \right)-\left( \frac{x_k-y_k}{2} \right)-\e \partial_{q_k} H_1\left(\frac{x-y}{2} \right).
\]
Let us call
\[
\mathbf{f}_1:=(\mathbf{f}_1^k)_{k\in\Z\setminus S}, \qquad \mathbf{f}_2:=(\mathbf{f}_2^k)_{k\in\Z\setminus S}, \qquad \mathbf{f}_3:=(\mathbf{f}_3^k)_{k\in S}, \qquad \mathbf{f}_4:=(\mathbf{f}_4^k)_{k\in S}.
\]
Then, Hypohteses $\mathbf{H1}$ and $\mathbf{H2}$ imply
%
\begin{equation}\label{def:fs}
\begin{aligned}
\mathbf{f}_i(0; x, y, \theta, r)&=O_2(x, y), & i&=1,2\\
\mathbf{f}_i(0; x, y, \theta, r)&=0,  & i&=3,4\\
\mathbf{f}_i(\varepsilon; 0, 0, \theta, 0)&=0, & i&=1, 2, 3, 4.
\end{aligned}
\end{equation}
Since $X_{H_1}\in C^r_{\Gamma}(\ell^{\infty}(\Z; M))$ then the functions $\mathbf{f}_i$ are $C^r_{\Gamma}(\mathcal{M}_{\delta})$ for any $r\geq 1$. 

Now we are in position to state the theorem of existence of invariant manifolds for the tori \eqref{def:invtori}, \eqref{def:invtori1}.

{\begin{theorem}\label{thm:toriinvman}
Consider the equation \eqref{system} and assume $\mathbf{H1}$ - $\mathbf{H2}$. There exists $\e_0>0$ such that for all $\e\in (0, \e_0)$, any invariant torus $\T_0$ of those  in \eqref{def:invtori}, \eqref{def:invtori1} possesses  stable and unstable invariant manifolds $W^{s, u}=W^{s, u}_{\e}$ .
Moreover, 
they can be represented locally as graphs.  More precisely, there exists $\delta>0$ small enough and functions $\gamma_{\e}^{s}=\gamma^s_{\e}(x, \theta), \gamma_{\e}^{u}=\gamma_{\e}^u(y, \theta)\in C^2_\Gamma(B_{\delta}(\ell^{\infty})\times \T^d; \ell^{\infty}\times \R^d)$ such that
\begin{itemize}
\item  the local invariant manifolds are parameterized as
\[
\begin{split}
W^s=\{ (x, \gamma^s_y(x, \theta), \theta, \gamma^s_{r}(x, \theta))\},\\
W^u=\{(\gamma^u_x(y, \theta), y, \theta, \gamma^u_{r}(y, \theta))\}.
\end{split}
 \]
\item$\gamma^{s, u}_{\e}(0, \theta)=0$. Moreover, its $C^1_\Gamma$ norm is of order $\delta+\eps$ and
\[
\sup_{ B_{\delta}(\ell^{\infty})\times \T^d} \| \gamma^{s, u}_{\e}\|_{\ell^{\infty}\times \R^d}\le \mathcal{O}(\delta^2+\delta\eps).
\] 
\item $\gamma_{\e}^{s, u}$ is $C^2$ with respect to $\e$.
\item For all $j\in \Z^m$ 
\[
\gamma^{s, u}_{\e}\colon B_{\delta}(\Sigma_{j, \Gamma})\times \T^d\to \Sigma_{j, \Gamma}\times \R^d,
\]
 $\gamma^{s, u}_{\e}\in C^2_{\Gamma}(B_{\delta}(\Sigma_{j, \Gamma})\times \T^d)$ and its $C^1_\Gamma$ norm is of order $\delta+\eps$.
\end{itemize}
\end{theorem}
}

This theorem not only gives the existence of the invariant manifolds of the invariant tori but also give decay properties for them. Its proof is a consequence of a  general invariant manifolds theory for invariant tori which is developed in Section \ref{sec:manifold}.

\subsubsection{Transversal intersection between the invariant manifolds}\label{sec:transvheterosheuristic}

Theorem \ref{thm:toriinvman} gives the existence and regularity of the invariant manifolds of the tori in \eqref{def:invtori}. When $\eps=0$, the stable and unstable  invariant manifolds of these tori coincide creating a homoclinic manifold.
Next step is to prove that, for $0<\eps\ll 1$, they intersect transversally and that moreover the stable invariant manifold of one of these tori intersects transversally the unstable invariant manifold of ``nearby'' tori. 

Since we are in an infinite dimensional setting, we devote the next section to review the definition of transversality between Banach submanifolds. Note also that we are dealing with flows with a (formal) first integral and, therefore, we need an ``adapted'' definition of transversality. 

\paragraph{Transversality of Banach submanifolds}

To define transversality between Banach submanifolds, we start by reviewing the notion of direct sum of Banach subspaces. Later we use it to talk about Banach submanifolds and their tangent spaces.

Let us consider a Banach space $\cX$ and two Banach subspaces $\cX_1, \cX_2$. Then, $\cX$ is the  direct sum of  $\cX_1, \cX_2$, which we denote by 
\[
\cX=\cX_1\oplus\cX_2 
\]
if the map $T:\cX_1\times \cX_2\to \cX$ given by $T(v_1,v_2)=v_1+v_2$ is an isomorphism. Note that by the definition of Banach subspaces, $T$ is a continuous map and therefore, by the Open Mapping Theorem, its inverse $T^{-1}$ is continuous as well. The inverse map is just $T^{-1}=(\pi_1,\pi_2)$, where $\pi_i$ is the projection onto $\cX_i$, $i=1,2$, and, therefore, the projections are also continuous. Recall that the fact that $T$ is an isomorphism implies that $\cX_1\cap \cX_2=\langle 0 \rangle$.


The  direct sum can be defined in a more general setting considering vector subspaces of $\cX$ instead of Banach subspaces. Then one has to distinguish between algebraic direct sum and topological direct sum. Algebraic refers to direct sum as vector spaces (i.e. $T$ is an isomorphism but $T^{-1}$ may not be bounded\footnote{Note that $T$ is always bounded. On the contrary, if  $\cX_1, \cX_2$ are only vector subspaces  one cannot use the  Open Mapping Theorem and therefore $T^{-1}$ may not be bounded.}) whether topological refers to also requiring that $T^{-1}$ is a bounded map.

Since we are interested only in the case when $\cX_1$, $\cX_2$ are Banach subspaces the notions of algebraic and topological direct sum coincide and therefore, to simplify the exposition, we just talk about direct sums.


We use this concept to define transversality between submanifolds of Banach manifolds.

\begin{definition}\label{def:transversality}
Let us consider a Banach manifold $\cM$ modeled on a Banach space $\cX$ and a point $p\in\cM$. Assume that $\cM$ possesses two Banach submanifolds $\cN_1$, $\cN_2$ such that $p\in\cN_1\cap \cN_2$. 
Then, we say that $\cN_1$, $\cN_2$ intersect transversally at $p$ if and only if the Banach subspaces $T_p\cN_1, T_p\cN_2$ of $T_p\cM$ satisfy
\[
 T_p\cM=T_p\cN_1\oplus T_p\cN_2.
\]
\end{definition}

Note that in this paper we are dealing with flows. Therefore if we consider invariant manifolds by the flow, they cannot intersect transversally since the flow direction (the Banach subspace generated by the vector field) belongs to the tangent space of all invariant manifolds. For this reason we need to adapt the definition of transversality as follows. Note also that in this paper, we are dealing with (formal) Hamiltonian systems and therefore the associated vector fields have (formal) first integrals (see Definition \ref{def:firstintegralflows}). 

%
%
%
%
%
%
%

We introduce the following definition of transversality for invariant manifolds of the flow associated to the vector field $X$.
\begin{definition}\label{def:transversalityflows}
Fix $p\in\cM$  such that $X(p)\neq 0$ and  consider two Banach submanifolds $\cN_1$ and $\cN_2$ of $\cM$ such that $p\in\cN_1\cap\cN_2$ and such that both are invariant by the flow associated to $X$. Assume also that $X$ has a formal first integral $G$ in the sense of Definition \ref{def:firstintegralflows}.
Then, we say that $\cN_1,\cN_2$ intersect transversally at $p$ if 
\begin{enumerate}
\item They satisfy
\[
 T_p\cN_1\cap T_p\cN_2 =\langle X(p)\rangle
\]
where $\langle X(p)\rangle$ is the one dimensional invariant subspace generated by $X(p)$.
\item The map
\[
 T: T_p\cN_1\times T_p\cN_2 \to T_p\cM,\qquad T(v_1,v_2)=v_1+v_2
\]
is a linear continuous map whose image is equal to $\Ker dG(p)$
\end{enumerate}
\end{definition}
Note that Item 1 implies that $\Ker T\subset T_p\cN_1\times T_p\cN_2$ is one dimensional and generated by the vector $(X(p),-X(p))$.

\begin{remark}\label{rmk:complement}
This notion of transversality can be phrased in terms of (topological) direct sum as follows. Since $\langle X(p)\rangle\subset \Ker dG(p)$ is one dimensional, we know that there exists a complement. That is, there exists a Banach subspace $\cH_p$ of $\Ker dG(p)$ such that
\begin{equation}\label{def:complement}
 \Ker dG(p)=\langle X(p)\rangle\oplus\cH_p.
\end{equation}
Then, Definition \ref{def:transversalityflows} is equivalent to require that the Banach subspaces $\HH_p^i=T_p\cN_i\cap\HH$ satisfy
\[
 \HH_p^1\oplus \HH_p^2=\HH_p.
\]
\end{remark}

\paragraph{Transverse heteroclinic orbit between invariant tori}

The phase space we are considering is 
\[
\mathcal{M}=\ell^{\infty}(\Z^m\setminus S; \R)\times \ell^{\infty}(\Z^m\setminus S; \R)\times \T^d\times \R^d,
\]
whose tangent space at any point $z\in\mathcal{M}$ can be identified as
\[
 T_z\mathcal{M}=\ell^{\infty}(\Z^m\setminus S; \R)\times \ell^{\infty}(\Z^m\setminus S; \R)\times \R^d\times \R^d,
\]
or, in the $\Sigma_{j, \Gamma}$ case, the space
\[
\mathcal{M}_{j, \Gamma}=\Sigma_{j, \Gamma}\times \Sigma_{j, \Gamma} \times \T^d\times \R^d.
\]
Even if the Hamiltonian \eqref{def:Hamiltonian} may only be formally defined, its differential $d H$ and, therefore,  $\mathrm{Ker} d H$ are well defined (see Definition \ref{def:firstintegralflows}).

Then, given two tori $\T_1$ and $\T_2$ (not necessarily of the same dimension), we consider the unstable manifold of $\T_1$, denoted by $W^u(\T_1)$,  and the stable manifold of   $\T_2$, denoted by $W^s(\T_2)$. Note that, by construction, $T_zW^u(\T_1)$ and $T_z W^s(\T_2)$ are Banach subspaces of  $T_z\mathcal{M}$ and the same is true for $\mathrm{Ker} d H$. 


%

To prove that these invariant manifolds of nearby tori intersect transversally in the sense of Definition \ref{def:transversalityflows}, we need to impose the non-degeneracy conditions $\mathbf{H3}$-$\mathbf{H4}$ on certain  Melnikov functions associated to $H_1$. 
One should expect (under non-degeneracy hypotheses) plenty of transverse homoclinic/heteroclinic orbits.

This allows to construct a transition chain of hyperbolic tori.
These tori belong the invariant subspaces $\cV_i$  in \eqref{def:subspacesXj}.
Fix $h>0$ and the energy level\footnote{Note that we are fixing an energy level once we restrict to a finite dimensional subspace (where the Hamiltonian is a well defined function). This is not contradictory with the fact that in the infinite dimensional setting we deal with formal Hamiltonians in the sense that they may be unbounded but have a well defined differential.} $H=h$. Then, we define 
\begin{equation}\label{def:cylinderLambdaj}
 \Lambda_i=\left\{(q,p)\in \cV_i\cap H^{-1}(h): q_{\sigma_{i+2}}=p_{\sigma_{i+2}}=0 \right\}.
\end{equation}

\begin{theorem}\label{thm:transitiochainXj}
Fix $i\geq 0$. Assume that $H$ satisfies $\mathbf{H1}$--$\mathbf{H4}$. Then there exists $\varepsilon_0>0$ such that for any $\varepsilon\in (0,\varepsilon_0)$ there exists $N>0$ and a sequence of tori $\T_{i,k}\subset \Lambda_i$, $k=1\ldots N$  such that 
\[
 W_{\eps}^u(P_{\sigma_i}) \pitchfork  W_{\eps}^s(\T_{i,0}),\quad   W_{\eps}^u(\T_{i.k}) \pitchfork W_{\eps}^s(\T_{i, k+1})\quad \text{for}\quad k=0,\ldots N-1\quad \text{and}\quad W_{\eps}^u(\T_{i,N}) \pitchfork  W_{\eps}^s(P_{\sigma_{i+1}})
\]
where $ \pitchfork$ denotes transversal intersection in the sense of Definition \ref{def:transversalityflows}, $P_{\sigma_i}$ is the periodic orbit introduced in \eqref{def:invtori} and $\T_{i.k}$ are invariant tori of the form \eqref{def:invtori1}. This statement is true both in $\ell^\infty$-functional setting and in $\Sigma_{j,\Gamma}$-functional setting.
\end{theorem}

This theorem is proven in Section \ref{sec:MelnikovTheory}. Note that the transition chains for all $i$'s can be concatenated to build an infinite transition chain.

Note that Theorem \ref{sec:MelnikovTheory} contains both the Arnold regime and the Jumping regime explained in Section \ref{sec:heuristics}. Indeed, the cylinder  $\Lambda_i$ in \eqref{def:cylinderLambdaj} is not normally hyperbolic since it possesses the periodic orbits $P_{\sigma_i}$ and $P_{\sigma_{i+1}}$ introduced in \eqref{def:invtori}
whose hyperbolicity ``within $\Lambda_i$'' is as strong as the normal one. However, for any $\delta>0$,
\[
  \Lambda_{i,\delta}=\left\{(q,p)\in \Lambda_i: E_i(q,p)\in (\delta, h-\delta)\right\}
\]
is a normally hyperbolic invariant manifold both for $\varepsilon=0$ and for $0<\varepsilon\ll 1$.
Therefore, the proof of Theorem \ref{thm:transitiochainXj} will be done in two steps. First for the tori in  $\Lambda_{i,\delta} $ and then for the tori very close to the periodic orbit.

\subsubsection{The Lambda lemma and the shadowing argument}\label{subsec:Sketch:Shadow}
To prove Theorem \ref{thm:technical} it only remains to shadow the transition chain obtained in Theorem \ref{thm:transitiochainXj}. This is done by means of a Lambda lemma.
Let us rename as $\T_j$, $j=0, \dots, N$ the one and two dimensional tori of the chain.

We denote by  $| \cdot |_d$ the norm of $\R^d$. We recall that
\[
\T^d:=(\R/2\pi \Z)^d:=\{ [\theta] : \theta\sim \vartheta\,\, \mathrm{if\,\,and\,\,only\,\,if}\,\,\theta-\vartheta=2\pi k\,\,\mathrm{for\,\,some}\,\,k\in \Z^d \}.
\]
With abuse of notation we denote by
\[
d_{\T^d}(\theta, \theta')=\inf_{q\in [\theta], p\in [\theta']} |q-p|_d=:|\theta-\theta'|_d.
\]

\begin{theorem}{(Lambda Lemma)}\label{thm:Lambda}
Let $\Phi^t_{H}$ be the flow of the Hamiltonian system \eqref{def:Hamiltonian} and consider an invariant torus $\T_{j}$ on which the dynamics is quasi-periodic (i.e. a non-resonant rigid rotation). Then, the following two statements are satisfied.
\begin{enumerate}
\item Consider a Banach submanifold  $\Gamma\subset\mathcal{M}$ and assume that it intersects transversally in the sense of Definition \ref{def:transversalityflows} (with respect to the formal first integral $H$) the stable manifold $W_{\e}^s(\T_j)$. Then
\begin{equation}\label{def:lambdaaccumution}
W_{\e}^u(\T_0)\subset \overline{\bigcup_{t\geq 0} \Phi^t_H(\Gamma)},
\end{equation}
where the closure is meant with respect to the metric 
\[
d(w, \tilde{w}):=\| x-\tilde{x} \|_{\ell^{\infty}}+\| y-\tilde{y} \|_{\ell^{\infty}}+|\theta-\tilde{\theta}|_d+| r-\tilde{r}|_d.
\]
\item  Consider a Banach submanifold  $\Gamma\subset\mathcal{M}_{j,\Gamma}$ and assume that it intersects transversally in the sense of Definition \ref{def:transversalityflows} (with respect to the formal first integral $H$) the stable manifold $W_{\e}^s(\T_j)$. Then, \eqref{def:lambdaaccumution} is satisfied with respect to the metric
\[
d_{j, \Gamma}(w, \tilde{w}):=\| x-\tilde{x} \|_{{j, \Gamma}}+\| y-\tilde{y} \|_{{j, \Gamma}}+|\theta-\tilde{\theta}|_d+| r-\tilde{r}|_d.
\]
\end{enumerate}
\end{theorem}
This theorem is proven in section \ref{sec:lambdalemma}. The proof follows the techniques developed for finite dimensional maps in \cite{FontichM98} (see also \cite{Cresson97}). The statement in Section \ref{sec:lambdalemma} is more precise than the one stated above and in particular it implies $C^1$ convergence of the iterated of $\Gamma$ as for the classical Lambda lemma (more precisely the $C^1$ convergence is for a submanifold of $\Gamma$, see Section \ref{sec:lambdalemma} for details).

Note that, by Theorem \ref{thm:toriinvman}, the invariant manifolds $W_{\e}^{s,u}(\T_j)$ can be seen as both submanifolds of $\mathcal{M}$ and $\mathcal{M}_{j,\Gamma}$. This allows to rely on this Lambda lemma to perform a shadowing argument in both Banach manifolds. Finally, note that Theorem \ref{thm:Lambda} only depends on the metric but not on the choice of coordinates. That is, the theorem is also valid in $\ell^\infty$ (respectively $\Sigma_{j,\Gamma}$) in the original coordinates $(q,p)$.

Next lemma constructs an orbit which shadows the transition chain provided by Theorem \ref{thm:transitiochainXj}. 
\begin{lem}\label{lemma:shadowing}
Given $\{ \e_j \}_{j\in\mathbb{N}}$ a sequence of strictly positive numbers, we can find a point $p$ and an increasing sequence of numbers $\{ T_j\}_{j\in\mathbb{N}}$ such that 
\[
\Phi^{T_j}(p)\in \mathcal{U}_{\e_j}(\T_j)
\]
where $\mathcal{U}_{\e_j}(\T_j)$ are $\e_j$-neighborhoods of the tori $\T_j$ in the topology of the metric space $\ell^{\infty}$.

Moreover, fixed $j\in \Z^m$, we have the same statement considering the topology of the metric space $\mathcal{M}_{j, \Gamma}$.
\end{lem}

\begin{proof}
We give the proof in the $\ell^{\infty}$ topology. The proof in $\Sigma_{j,\Gamma}$ is analogous. Let $q\in W^s(\T_1)$. There exists a closed ball $B_1\subset \ell^{\infty}$ centered at $q$ such that
\[
\Phi^{T_1}(B_1)\subset \mathcal{U}_{\e_1}(\T_1)\subset\ell^{\infty}.
\]
By the Lambda Lemma Theorem \ref{thm:Lambda} we have
\[
W^s(\T_2)\cap B_1\neq \emptyset.
\]
Hence we can find a closed ball $B_2\subset B_1$ centered at a point of $W^s(\T_2)$ such that
\[
\begin{cases}
\Phi^{T_1}(B_2)\subset \mathcal{U}_{\e_1}(\T_1),\\
\Phi^{T_2}(B_2)\cap\mathcal{U}_{\e_2} (\T_2).
\end{cases}
\]
Then by induction it is possible to construct a sequence of closed nested balls $B_{j+1}\subset B_j\subset\dots$ such that
\[
\Phi^{T_j}(B_i)\subset \mathcal{U}_{\e_j}(\T_j), \quad i\le j.
\]
Since $\ell^{\infty}$ is a complete metric spaces, the Cantor's intersection Theorem ensures that the infinite sequence of closed nested balls $B_j$ has at least one point as intersection. This concludes the proof.
\end{proof}

This concludes the proof of Theorem \ref{thm:technical}. Indeed the orbit shadowing the transition chain visits arbitrarily small neighborhoods of the periodic orbits $P_{\sigma_j}$ at certain times. When they belong to such neighborhoods the energies $E_k$, $k\neq \sigma_j$ can be chosen to be smaller than $\eta$ whereas the energy $E_{\sigma_j}$ is $\eta$-close  to that of the periodic orbit. This is exactly the behavior stated in Theorem \ref{thm:technical}.

\section{Local invariant manifolds of invariant tori}\label{sec:manifold}

In this section we provide an invariant manifolds theory for invariant tori for both maps and flows with spatial structure on lattices. We consider the setting where the tori are finite dimensional whereas the invariant manifolds have infinite dimensions. First in Section \ref{sec:invmanmaps} we deal with maps and in Section \ref{sec:invmanflows} we deal with flows.

\subsection{Invariant manifolds of maps}\label{sec:invmanmaps}

In this section we provide abstract theorems of existence of local invariant manifolds of finite dimensional invariant tori for maps that are locally close to uncoupled maps. We consider only the case of invertible maps. Then it is sufficient to prove the result for the stable manifold. 

\medskip

Let $S\subset \Z^m$ with cardinality $d$ and $S^c:=\Z^m\setminus S$.
We recall the following notations from Section \ref{sec:functsettorus}
\begin{equation}\label{notation:linfS}
\begin{aligned}
&\ell^{\infty}_{S^c}:=\ell^{\infty}(\Z^m\setminus S; \R), \qquad \Sigma_{j, \Gamma, S^c}:=\Sigma_{j, \Gamma} (\Z^m\setminus S; \R),\\
& \R_S^d:=\ell^{\infty}(S; \R), \qquad \,\, \qquad \T^d_S:=\ell^{\infty}(S; \T).
\end{aligned}
\end{equation}
We consider the complete metric space 
\[
\mathcal{M}:=\ell_{S^c}^{\infty}\times \ell_{S^c}^{\infty}\times \T_S^d\times \R_S^d,
\]
and we denote its variables by
\[
w:=(x, y, \theta, r).
\]
Let $\delta>0$. We consider  maps
\begin{equation}\label{def:mapF}
F_{\nu}\colon \mathcal{M}_{\delta}:= B_{\delta} (\ell^{\infty}_{S^c})\times B_{\delta}(\ell^{\infty}_{S^c})\times \T_S^d\times B_{\delta}(\R_S^d)\subset \mathcal{M}\to \mathcal{M},
\end{equation}
which depend on a parameter $\nu\in (0, \mu)$ for some $\mu>0$, and are of the form
\[F_{\nu}(w) :=F_0(w)+f_{\nu}( w)
 \]
with 
\begin{equation}\label{def:F}
\begin{aligned}
F_0(w)&=(A_-(\theta) \,x, A_+(\theta)\, y, \theta+\omega(x, y, r), B(\theta)\,r),\\
f_{\nu}( w)&=\big(f_1(\nu; w), f_2(\nu; w), f_3(\nu; w), f_4(\nu; w) \big)
\end{aligned}
\end{equation}
where
\[
A_{\pm}(\theta)\in \mathcal{L}_{\Gamma}(\ell^{\infty}_{S^c}), \quad B(\theta)\in\mathcal{L}_{\Gamma}(\R_S^d)
\]
and $f_{1}, f_2 (\nu; \cdot)\colon \mathcal{M}_{\delta}\to \ell^{\infty}_{S^c}$, $f_3(\nu; \cdot)\colon\mathcal{M}_{\delta}\to \T_S^d$, $f_4(\nu; \cdot) \colon \mathcal{M}_{\delta}\to \R_S^d$ (following the notation and definitions in Section \ref{sec:functsettorus}).

Let us call
\begin{equation}\label{tB}
\tB_{\delta}:=B_{\delta} (\ell^{\infty}_{S^c})\times B_{\delta}(\ell^{\infty}_{S^c})\times B_{\delta}(\R_S^d).
\end{equation}
We assume that
\[
\T_0:=\{ x=0, y=0, r=0 \}
\]
is an invariant torus for the map $F_{\nu}$ for all $\nu\in(0,\mu)$ and we provide a theorem of existence of local invariant manifolds for $\T_0$
in class $C^2_{\Gamma}$ (and $C^2$-dependence with respect to the parameter $\nu$). We start with the Lipschitz case, then we deal with the $C^1_{\Gamma}$-regularity and eventually with the $C^2_{\Gamma}$ case. We follow a graph transform approach  and  provide full detailed proofs for  the Lipschitz, $C^1_\Gamma$ and $C^2_\Gamma$ settings. 

Since the torus $\T_0$ is fixed, along this section we can lighten the notation by denoting
\begin{equation}\label{notation:linfS2}
\ell^{\infty}=\ell^{\infty}_{S^c}, \qquad \T^d=\T^d_S, \qquad \R^d=\R^d_S, \qquad \Sigma_{j, \Gamma}:=\Sigma_{j, \Gamma, S^c}.
\end{equation}

\medskip

When we consider a function $\Psi \colon \mathcal{U}_X\times \mathcal{U}_Y\subseteq X\times Y\to Z$, $z=\Psi(x, y)$, where $X$, $Y$, $Z$ are complete metric spaces and $\mathcal{U}_X$, $\mathcal{U}_Y$ subsets of $X$ and $Y$ respectively, we denote by $\lip \Psi$ the Lipschitz constant of the function $\Psi$ and
\begin{equation}\label{def:lip}
\begin{aligned}
&\lip_x \Psi(y):=\inf_{x\neq x'} \frac{d_Y\big(\Psi(x, y), \Psi(x', y) \big)}{d_X(x, x' )},\\
&\lip_{x} \Psi:=\sup_{y\in \mathcal{U}_Y} \lip_x \Psi(y), \qquad \lip_{y} \Psi=\sup_{x\in \mathcal{U}_X} \lip_y \Psi(x).
\end{aligned}
\end{equation}

We will also prove the existence of invariant manifolds of $\T_0$ for maps of the form \eqref{def:F} on the complete metric space
\[
\mathcal{M}_{j, \Gamma}:=\Sigma_{j, \Gamma}\times \Sigma_{j, \Gamma}\times \T^d_S \times \R_S^d\subset \mathcal{M}.
\]

\subsubsection{Lipschitz invariant manifolds}

We consider a non-negative continuous function $L(\delta, \mu)$ such that $L(0, 0)=0$.
We assume that there exist constants $\lambda>1, \beta\geq 1, \K, K_{\theta}>0$ such that:
\begin{itemize}
\item[{$\mathbf{(H0)_{\mathrm{lip}}}$}] We have
\begin{equation}\label{bound:betalambda}
\lambda^{-1} \beta (1+K_{\theta})<1.
\end{equation}
\item[{$\mathbf{(H1)_{\mathrm{lip}}}$}] The functions $A_{\pm}, B, \omega$ and $f_{\nu}$ in \eqref{def:F} are Lipschitz, namely 
\[
A_{\pm}\in \lip(\T^d; \ell^{\infty}), \quad B\in \lip(\T^d; \R^d), \quad \omega\in \lip(\tB_{\delta}; \T^d), \quad f_{\nu}\in \lip(\mathcal{M}_{\delta}; \mathcal{M})
\]
and 
\begin{equation}\label{def:Kf}
\lip \,\omega, \lip\, A_{\pm}, \lip \,B, \lip f_3\le \K.
\end{equation}
Moreover the linear operators $A_{\pm}, B$ satisfy
\[
\| A_-(\theta)\|_{\mathcal{L}_{\Gamma}(\ell^{\infty})}, \| A_+(\theta)^{-1} \|_{\mathcal{L}_{\Gamma}(\ell^{\infty})}\le \lambda^{-1}, \qquad 
\| B^{-1}(\theta) \|_{\mathcal{L}_{\Gamma}(\R^d)}\le \beta \qquad \forall \theta\in\T^d.
\]
\item[$\mathbf{(H2)_{\mathrm{lip}}}$] For $j=1, 2, 4$ we have that $f_j$ is $L(\delta, \mu)$-Lipschitz with respect to $w$. Moreover, $f_j(\nu; 0, 0, \theta, 0)=0$ and
\begin{equation}\label{hyp:H3}
\begin{aligned}
\| f_{k}(\nu; x, y, \theta, r)- f_{k}(\nu; x, y, \theta', r) \|_{\ell^{\infty}} &\le \K (\|  x \|_{\ell^{\infty}}+\| y \|_{\ell^{\infty}}+|r|_d)\,| \theta-\theta'|_d, \quad k=1, 2,\\
| f_{4}(\nu; x, y, \theta, r)- f_{4}(\nu; x, y, \theta', r) |_d &\le \K (\|  x \|_{\ell^{\infty}}+\| y \|_{\ell^{\infty}}+|r|_d)\,| \theta-\theta'|_d.
\end{aligned}
\end{equation}
\item[$\mathbf{(H3)_{\mathrm{lip}}}$] The function $f_3$ is $K_{\theta}$-Lipschitz with respect to $\theta$.
\end{itemize}
These three hypotheses are sufficient to have invariant manifolds of the invariant torus. If one also wants them to be Lipschitz with respect to the parameter $\nu$, one has to impose also the following.
\begin{itemize}
\item[$\mathbf{(H4)_{\mathrm{lip}}}$] We have
\begin{align*}
\|  f_k(\nu; x, y, \theta, r)-f_k(\nu'; x, y, \theta, r) \|_{\ell^{\infty}} &\le \K (\|  x \|_{\ell^{\infty}}+\| y \|_{\ell^{\infty}}+|r|_d) |\nu-\nu'| \qquad k=1, 2,\\
|  f_4(\nu; x, y, \theta, r)-f_4(\nu'; x, y, \theta, r) |_{d} &\le \K (\|  x \|_{\ell^{\infty}}+\| y \|_{\ell^{\infty}}+|r|_d) |\nu-\nu'|
\end{align*}
and $f_3$ is $K_{\theta}$-Lipschitz with respect to $\nu$.
\end{itemize}

We observe that by assumption $\mathbf{(H2)_{\mathrm{lip}}}$ the torus $\T_0$ is invariant by $F_{\nu}$. To simplify the notation we denote $F_{\nu}$ by $F$ and $f_{\nu}$ by $f$.
\begin{theorem}\label{thm:lipcase}
Let $F\colon \mathcal{M}_{\delta}\to \mathcal{M}$ in \eqref{def:mapF} satisfy $\mathbf{(H0)_{\mathrm{lip}}}$-$\mathbf{(H3)_{\mathrm{lip}}}$. Then there exist $\delta_0>0$ and $\mu_0>0$ such that for all $\delta\in (0, \delta_0)$ and $\mu\in (0, \mu_0)$ the $F$-invariant torus $\T_0$ possesses a stable invariant manifold which can be represented as graph of a Lipschitz function $\gamma_{\nu}^s(x, \theta)\in \lip(B_{\delta}(\ell^{\infty})\times \T^d; \ell^{\infty}\times \R^d)$ that satisfies:
\begin{itemize} 
\item$\gamma^s_{\nu}(0, \theta)=0$. Moreover, its Lipschitz constant is of order $\delta+L(\delta, \mu)$ and
\[
\sup_{(x, \theta,\nu)\in B_{\delta}(\ell^{\infty})\times \T^d \times(0,\mu)} \| \gamma^s_\nu(x, \theta)\|_{\ell^{\infty}\times \R^d}\le \mathcal{O}(\delta^2+\delta\,L(\delta, \mu)).
\] 
\item The iterates of the points $(x, \theta, \gamma_{\nu}^s(x, \theta))$ tend to the torus exponentially fast with asymptotic rate bounded by $\lambda^{-1}$. 
\end{itemize}
Moreover if we also impose $\mathbf{(H4)_{\mathrm{lip}}}$, $\gamma_{\nu}^s$ depends in a Lipschitz way on $\nu\in(0,\mu)$. 
\end{theorem}

If one imposes decay properties on the map $F$, the invariant manifolds also have decay properties.

\begin{theorem}[\emph{$\Sigma_{j, \Gamma}$ case}]\label{thm:lipcaseSigma}
Let the map $F\colon \mathcal{M}_{j, \Gamma, \delta}\to \mathcal{M}_{j, \Gamma}$ in \eqref{def:mapF} satisfy the assumptions $\mathbf{(H0)_{\mathrm{lip}}}$-$\mathbf{(H4)_{\mathrm{lip}}}$ where $\mathcal{M}$, $\ell^{\infty}$ are replaced respectively by $\mathcal{M}_{j, \Gamma}$ and $\Sigma_{j, \Gamma}$. Then,
the same results of Theorem \ref{thm:lipcase} hold with $\gamma_{\nu}^s(x, \theta)\in \lip(B_{\delta}(\Sigma_{j, \Gamma})\times \T^d; \Sigma_{j, \Gamma}\times \R^d)$.
\end{theorem}

We devote the rest of this section to prove Theorems \ref{thm:lipcase}, \ref{thm:lipcaseSigma}. We first introduce some notations.
For functions $\psi\colon \dom\to Z$, where $Z$ is some Banach space with norm $\| \cdot \|$, we define
\[
\| \psi \|_0:=\sup_{(x, \theta)\in\dom} \| \psi(x, \theta) \|, \qquad \| \psi \|_1:=\sup_{\substack{(x, \theta)\in\dom,\\ x\neq 0}} \frac{\| \psi(x, \theta) \|}{\| x \|_{\ell^{\infty}}}.
\]
We shall denote
\begin{equation}\label{def:g}
g:=-(f_2, f_4)^t\colon \mathcal{M}_{\delta} \to \ell^{\infty}\times \mathbb{R}^d.
\end{equation}
Note that also $g$ satisfies $\bf{(H2)_{\mathrm{lip}}}$. We define
\[
C(\theta):=\begin{pmatrix}
A_+(\theta) & 0\\
0 & B(\theta)
\end{pmatrix}\in \mathcal{L}_{\Gamma}( \ell^{\infty}\times \mathbb{R}^d).
\]
Note that $\| C(\theta)^{-1} \|_{\mathcal{L}_{\Gamma}(\ell^{\infty}\times \mathbb{R}^d)}\le \beta$ and $C(\theta)^{-1}$ is $(\beta^2 \K)$-Lipschitz.

\medskip

\noindent\textbf{Proof of Theorems \ref{thm:lipcase} and \ref{thm:lipcaseSigma}}
We give full details for the proof of Theorem \ref{thm:lipcase}. Along such proof we make some remarks on how to adapt it for the proof of Theorem \ref{thm:lipcaseSigma}. Throughout the proof we assume withouth mentioning assumptions $\mathbf{(H0)_{\mathrm{lip}}}$-$\mathbf{(H4)_{\mathrm{lip}}}$.

We shall look for the stable manifold of $\T_0$ as a graph of a function of $(x, \theta)\in \dom$ taking values in $\ell^{\infty}\times\R^d$, which  is invariant by $F$. The invariance condition for the $\gr(\gamma)$ is
\[
\pi_{y, r} F(z, \gamma(z))=\gamma\big( \pi_{x, \theta} F(z, \gamma(z)) \big), \quad z:=(x, \theta)
\]
which is equivalent to
\begin{equation}\label{G}
\gamma(z)=G(\gamma)(z):=C(\theta)^{-1}\big( g(z, \gamma(z))+\gamma(h(z)) \big),
\end{equation}
where
\begin{equation}\label{h}
\begin{aligned}
& h(z):=(h_1(z), h_2(z))\colon \dom\to \ell^{\infty}\times \mathbb{T}^d,\\
&h_1(z)=A_-(\theta) x+f_1(z, \gamma(z)), \quad h_2(z)=\theta+\omega(x,  \gamma(z))+f_3(z, \gamma(z)).
\end{aligned}
\end{equation}
We introduce
\[
\Xi:=\left\{ \gamma \in C^0( \dom, \ell^{\infty}\times\R^d), \,\,\| \gamma \|_1<\infty \right\},
\]
which is a Banach space with the norm $\| \cdot \|_1$, and given $M>0$, $c>0$, we define the closed subset (recall \eqref{def:lip})
\begin{equation}\label{def:SigmacM}
\Xi_{c, M}:=\left\{ \gamma\in \Xi : \lip(\gamma)\le c,\,\,\lip_{\theta} \gamma(x) \le M \| x \|_{\ell^{\infty}}\,\,\,\,\forall x\in B_{\delta}(\ell^{\infty}) \right\}.
\end{equation}
\begin{remark}
Note that if $\gamma\in \Xi_{c, M}$ then
\begin{equation}\label{gamma0theta}
\gamma(0, \theta)=(0, 0) \quad \forall \theta\in \T^d\quad \text{ and } \quad \| \gamma(x, \theta) \|_{\infd}\le c \| x \|_{\ell^{\infty}}.
\end{equation}
Therefore $\| \gamma \|_0\le c \,\delta$.
\end{remark}
We shall look for a fixed point of the operator $G$ defined in \eqref{G} in the set $(\Xi_{c, M}, \| \cdot \|_1)$ for opportune $c, M$.\\
First let us prove two lemmas. 

\begin{lem}\label{lemtec}
If $\gamma\in \Xi_{c, M}$ with $c\in (0, 1]$ then
\begin{align*}
&\| f_j(z, \gamma(z)) \|_j \le L (1+c) \| x \|_{\ell^{\infty}}, \quad j=1, 2, 4, \\ 
&\| f_j(z, \gamma(z))-f_j(z', \gamma(z')) \|_j \le L (1+c) \| z-z'\|_{\infn}, \quad j=1, 2, 4,\\
&\| f_j(x, \theta, \gamma(x, \theta))-f_j(x, \theta', \gamma(x, \theta')) \|_j \le (\K (1+c)+ L M) \| x \|_{\ell^{\infty}} | \theta-\theta'|_d, \quad j=1, 2, 4,\\ 
& | f_3(z, \gamma(z))-f_3(z', \gamma(z')) |_d\le {\K (1+c)} \| z-z'\|_{\infn},
\end{align*}
where $\| \cdot \|_j=\| \cdot \|_{\ell^{\infty}}$ for $j=1, 2$, and $\| \cdot \|_4=| \cdot |_d$.
Then $g$ also satisfies The first, the second and the third bounds with a factor $2$.
\end{lem}
\begin{proof}
The first bound follows by the fact that $f(0, \theta, \gamma(0, \theta))=0$, $f_j$ with $j=1, 2, 4$ is $L$-Lipschitz by $\bf{(H2)_{\mathrm{lip}}}$ and $\gamma$ is $c$-Lipschitz. The second bound follows by similar arguments.
To prove the third bound we use $\bf{(H2)_{\mathrm{lip}}}$, \eqref{gamma0theta} and $\lip_{\theta} \gamma(x)\le M \| x \|_{\ell^{\infty}}$. The last one is a consequence of \eqref{def:Kf}.
\end{proof}


\begin{lem}\label{lemh}
Take $\gamma\in \Xi_{c, M}$ with $c\in (0, 1]$. Then for $z, z', (x, \theta')\in\dom$, the function $h$ introduced in \eqref{h} satisfies
\begin{align}
&\| h_1(z) \|_{\ell^{\infty}}\le (\lambda^{-1}+L(1+c)) \| x \|_{\ell^{\infty}}, \label{h1Lx}\\ \label{h1z}
&\| h_1(z)-h_1(z') \|_{\ell^{\infty}}\le (\lambda^{-1}+L(1+c)+{\K L}) \| z-z'\|_{\infn},\\
& \| h_1(x, \theta)-h_1(x', \theta) \|_{\ell^{\infty}}\le \big(\lambda^{-1}+L(1+c) \big) \| x-x'\|_{\ell^{\infty}}, \label{bound:lipxh1} \\  \label{bound:lipthetah1}
& \| h_1(x, \theta)-h_1(x, \theta') \|_{\ell^{\infty}}\le (\K\,(2+c)+LM) \| x \|_{\ell^{\infty}}\,\, | \theta-\theta'|_d,\\ \label{h2z}
& | h_2(z)-h_2(z') |_d\le (1+{2\K} (1+c)) \| z-z'\|_{\infn},\\
& | h_2(x, \theta)-h_2(x', \theta) |_d\le   2\K (1+c), \label{bound:lipxh2}\\  \label{bound:lipthetah2}
& | h_2(x, \theta)-h_2(x, \theta') |_d\le (1+K_{\theta}+2\,\K\, M \| x \|_{\ell^{\infty}})\,| \theta-\theta'|_d.
\end{align}
Moreover, if $\gamma\in\Xi_{c, M}$ we have
\begin{equation}\label{bound:gammah}
\| \gamma(h(z)) \|_0\le c \| h_1(z) \|_{\ell^{\infty}} \le c (\lambda^{-1}+L(1+c)) \| x \|_{\ell^{\infty}}.
\end{equation}
\end{lem}
\begin{proof}
The proof follows by Lemma \ref{lemtec} and assumptions $\bf{(H1)_{\mathrm{lip}}}$-$\bf{(H3)_{\mathrm{lip}}}$.
\end{proof}


\medskip

By choosing $\delta$ and $\mu$ small enough we have that
\[
\lambda^{-1}+L(\delta, \mu)(1+c)<1.
\]
Then by \eqref{h1Lx} we have that $h$ in \eqref{h} maps $\dom$ into itself. In particular $G(\gamma)$ in \eqref{G} is a well defined function on $\dom$.
\begin{remark}\label{rem:hjG}
{In the $\Sigma_{j, \Gamma}$ case we have that $h_1(x, \theta)$ maps $B_{\delta}(\Sigma_{j, \Gamma})\times \T^d$ into $\Sigma_{j, \Gamma}$ since, by Lemma \ref{lem:fm1}, $\| A_-(\theta)\, x \|_{j, \Gamma}\le \| A_-(\theta) \|_{\mathcal{L}_{\Gamma}(\ell^{\infty})} \| x \|_{j, \Gamma}$.}
\end{remark}


In the next lemma we prove that $G$ maps $\Xi_{c, M}$ into itself for opportune $c, M$.
\begin{lem}\label{lemball}
For any $c\in (0, 1]$ and $M$ such that 
\begin{equation}\label{condcM}
M>\frac{\beta\K\Big((1+c) (2+c)+{\,c\,(1+\beta  \lambda^{-1})}\Big)}{1-\beta\lambda^{-1}(1+K_{\theta})}
\end{equation}
there exist $\delta$ and $\nu$ small enough such that the map $G$ introduced in \eqref{G} satisfies $G(\Xi_{c, M})\subset \Xi_{c, M}$.
\end{lem}
\begin{proof}
Since $g, \gamma$ and $h$ are continuous then $G(\gamma)$ is continuous.
By using \eqref{gamma0theta} and Lemma \ref{lemtec} we have
\[
\| C(\theta)^{-1}\big(g(z, \gamma(z))+\gamma(h(z))\big) \|_{1}\stackrel{\eqref{bound:gammah}}{\le}  {\beta\,\Big(2 L(1+c) +c \big(\lambda^{-1}+L(1+c) \big) \Big)}
\]
which implies $\| G(\gamma) \|_1<\infty$.
Now we prove that $G(\gamma)$ is $c$-Lipschitz. By Lemma \ref{lemtec} we have that
\[
\| g(z, \gamma(z))-g(z', \gamma(z')) \|_{\infd}\le 2\,L (1+c) \| z-z' \|_{\infn}.
\]
Moreover,
\begin{align*}
\| \gamma(h(z))-\gamma(h(z'))\|_{\infd}  \le &\| \gamma(h_1(z), h_2(z))-\gamma(h_1(z'), h_2(z)) \|_{\infd}\\
&+\| \gamma(h_1(z'), h_2(z))-\gamma(h_1(z'), h_2(z')) \|_{\infd}\\
\le  &c \| h_1(z)-h_1(z') \|_{\ell^{\infty}}+M \| h_1(z') \|_{\ell^{\infty}} | h_2(z)-h_2(z') |_d,
\end{align*}
which, recalling \eqref{h1Lx}, \eqref{h1z}, \eqref{h2z}, implies
\begin{align*}
\lip (\gamma\circ h)\le &  c \big( \lambda^{-1}+L(1+c)+\K \delta \big)+M \| x \|_{\ell^{\infty}} (\lambda^{-1}+L(1+c)) (1+2 \K (1+c)).
\end{align*}
Then, we can deduce that
\begin{equation}\label{lipG}
\begin{aligned}
\lip(G)\le\, & { \K \beta^2\,L\,\Big(2 L(1+c) +c \big(\lambda^{-1}+L(1+c) \big) \Big)}+2\,\beta\,L (1+c)\\
&+\beta\Big( c \big( \lambda^{-1}+L(1+c)+\K \delta \big)+M \delta (\lambda^{-1}+L(1+c)) (1+2 \K (1+c))     \Big).
\end{aligned}
\end{equation}
Taking $\delta$ and $\mu$ small enough, we have that $\lip(G)\le \beta\lambda^{-1} c\stackrel{\eqref{bound:betalambda}}{<}c$. It remains to prove that
\begin{equation}\label{claim1}
\lip_{\theta} G(\gamma)(x)\le M \| x \|_{\ell^{\infty}}.
\end{equation}
Let us write $\tz:=(x, \theta')$. By Lemma \ref{lemtec} we have
\[
\| g(z, \gamma(z))-g(\tz, \gamma(\tz)) \|_{\infd}\le 2 (\K (1+c)+LM) \| x \|_{\ell^{\infty}} |\theta-\theta'|_d.
\]
By Lemma \ref{lemh} and \eqref{h1Lx},\eqref{h1z}, \eqref{h2z} we have
\begin{align*}
\| \gamma(h(z))-\gamma(h(\tz)) \|_{\infd}\le& c \| h_1(z)-h_1(\tz) \|_{\ell^{\infty}}+M \| h_1(z) \|_{\ell^{\infty}} | h_2(z)-h_2(\tz) |_d\\
\le& \Big(c (\K (2+c)+LM) \\
&+M \big(\lambda^{-1}+L(1+c) \big) \big(1+K_{\theta}+2 \K \, M \| x \|_{\ell^{\infty}}\big)\Big) \| x \|_{\ell^{\infty}}\,\, | \theta-\theta'|_d.
\end{align*}
Then, we obtain
\begin{equation}
\begin{aligned}
\lip_{\theta} G(\gamma)(x)\le \,&  \Big[\K\,\beta^2\,\,\Big(2 L(1+c) +c \big(\lambda^{-1}+L(1+c) \big) \Big)\,\,+2 \beta (\K (1+c)+LM) \Big] \| x \|_{\ell^{\infty}}\\
& +\beta\Big[c (\K (2+c)+LM)+M \big(\lambda^{-1}+L(1+c) \big) \big(1+K_{\theta}+2\K M \| x \|_{\ell^{\infty}}\big)\Big] \| x \|_{\ell^{\infty}}\,\,.
  \end{aligned}
\end{equation}
The coefficient in the r.h.s of the above inequality tends to 
{\[\K\,\beta^2 c \lambda^{-1}+ \beta \K (2+c) (1+c)+\beta c \K+\beta M \lambda^{-1} \big(1+K_{\theta})  \]} as $\delta, \mu\to 0$. Then by \eqref{condcM} and taking $\de$ and $\mu$ small enough, we get \eqref{claim1}.
\end{proof}

\begin{remark}\label{rem:GjG}
{In the $\Sigma_{j, \Gamma}$ case we just have to ensure that if $\gamma\in C^0( B_{\delta}(\Sigma_{j, \Gamma})\times \T^d; \Sigma_{j, \Gamma}\times\R^d)$ then $G(\gamma)$ in \eqref{G} maps $ B_{\delta}(\Sigma_{j, \Gamma})\times \T^d $ into $ \Sigma_{j, \Gamma}\times\R^d$. This holds by Lemma \ref{lemma:SigmajGammazero} and taking $\delta$ and $\mu$ small enough.}
\end{remark}


Next we prove that $G$ is a contraction on $\Xi_{c, M}$.
\begin{lem}
If $c$ and $M$ satisfy the assumptions of Lemma \ref{lemball} then there exist $\delta$ and $\mu$ small enough such that $G\colon \Xi_{c, M}\to \Xi_{c, M}$ is a contraction. 
\end{lem}
\begin{proof}
Recalling that the function $f_2$ is $L$-Lipschitz, $\omega$ and $f_3$ are $\K$-Lipschitz we have
\begin{equation}\label{hgamma-gamma'}
\begin{aligned}
&\| h_1(\gamma)(z)-h_1(\tg) (z) \|_{\ell^{\infty}}\le L \| \gamma(z)-\tg(z) \|_{\infd} \\
&| h_2(\gamma)(z)-h_2(\tg) (z) |_d\le 2 \K  \| \gamma(z)-\tg(z) \|_{\infd}
\end{aligned}
\end{equation}
which imply
\begin{align*}
\| \gamma(h(\gamma)(z))- \gamma(h(\tg)(z))  \|_{\infd}&\le \| \gamma(h(\gamma)(z))-\gamma(h_1(\gamma)(z), h_2(\tg)(z)) \|_{\infd}\\
&+\| \gamma(h_1(\gamma)(z), h_2(\tg)(z)) - \gamma(h(\tg)(z)) \|_{\infd}\\
&\le \Big(2 \K M (\lambda^{-1} +L(1+c))  \| x \|_{\ell^{\infty}} +c L \Big)\,\| x \|_{\ell^{\infty}}\, \|\gamma-\tg \|_1.
\end{align*}
Since $h$ maps $\dom$ to itself we have (recall that $\gamma(0, \theta)=(0, 0)$)
\[
\| \gamma(h(\gamma))-\tg(h(\gamma))\|_{\infd}\le \|\gamma-\tg \|_1 \| h_1(\gamma) \|_{\ell^{\infty}}  \stackrel{\eqref{h1Lx}}{\le} (\lambda^{-1}+L(1+c)) \| x \|_{\ell^{\infty}}\,\|\gamma-\tg \|_1.
\]
By the definition of $g$ in \eqref{def:g} and the Hypothesis $\mathbf{(H2)_{\mathrm{lip}}}$ we have
\[
\| g(z, \gamma(z))-g(z, \tilde{\gamma}(z)) \|_{\infd}\le 2 L \| x \|_{\ell^{\infty}}\, \|\gamma-\tg\|_1.
\]
By combining these estimates, one can deduce that
\begin{align*}
\| G(\gamma)(z)-G(\tg)(z) \|_{\infd}\le & \beta \Big( L\,(2+c)+ \,2 \K M (\lambda^{-1} +L (1+c))  \| x \|_{\ell^{\infty}}\\
& +\,(\lambda^{-1}+L(1+c))\Big) \| x \|_{\ell^{\infty}}\,\|\gamma-\tg \|_1.
\end{align*}
Taking $\delta$ and $\mu$ small enough and using that $\beta\lambda^{-1}<1$ (see \eqref{bound:betalambda}) we conclude that $G$ is a contraction.
\end{proof}
\begin{remark}\label{rem:csmallL}
By the bound \eqref{lipG} we can see that $c$ can be taken of order $\delta+L(\delta, \mu)$.
\end{remark}
 To obtain the Lipschitz dependence of $\gamma^s_{\nu}$ on the parameter $\nu$ one can repeat the same proof by treating $\nu$ as an additional angle\footnote{Recall that the we are only interested for $0<\nu\ll 1$. Therefore, by using a bump function (recall that $\nu$ is one-dimensional) one can assume that the $\nu$ dependence is periodic.} of $\T_0$. This concludes the proof of Theorem \ref{thm:lipcase}.
 
\smallskip

 We note that the proof of Theorem \ref{thm:lipcase} relies on Lemmata \ref{lemtec}, \ref{lemh}. It is easy to see that these Lemmata hold (replacing the norms accordingly) also under the assumptions of Theorem \ref{thm:lipcaseSigma}, with the same proofs. Then, thanks to Remarks \ref{rem:hjG} and \ref{rem:GjG}, the proof of Theorem \ref{thm:lipcaseSigma} in the $\Sigma_{j, \Gamma}$ case follows.

%

\subsubsection{$C_{\Gamma}^1$ regularity of the invariant manifolds}
Let us denote by $v=(y, r)$. Recall the continuous function $L(\delta, \mu)$ and notations \eqref{notation:linfS}, \eqref{notation:linfS2}, \eqref{tB} introduced in the previous section. We define
\begin{equation}\label{defspaces}
\begin{aligned}
&E_s:=C^0(\dom; \mathcal{L}_s), \quad s=x, \theta, v \qquad \mathrm{with}\\
 &\mathcal{L}_x:=\mathcal{L}_{\Gamma}(\ell^{\infty}; \ell^{\infty}\times\R^d), \quad  \mathcal{L}_{\theta}:=\mathcal{L}_{\Gamma}(\R^d; \ell^{\infty}\times\R^d), \quad  \mathcal{L}_v:=\mathcal{L}_{\Gamma}(\ell^{\infty}\times\R^d)
 \end{aligned}
\end{equation}
and norms $\| \cdot \|_{\mathcal{L}_s}$, $s=x, \theta, v$.

 We assume that there exist constants $\K, K_{\theta}$ such that:
 \begin{itemize}
\item[$\mathbf{(H0)_{C^1}}$] We have
 \begin{equation}\label{assumption:constants}
 \beta \lambda^{-1} (1+K_{\theta})^2<1.  
 \end{equation}
\item[$\mathbf{(H1)_{C^1}}$] Assume $\bf{(H1)_{\mathrm{lip}}}$. The functions $A_{\pm}\in C^2(\T^d; \mathcal{L}_{\Gamma}(\ell^{\infty})), B\in C^2(\T^d; \mathcal{L}_{\Gamma}(\R^d)), \omega\in C^1(\mathtt{B}_{\delta}; \T^d)$ and $f=f_{\nu}\in C_{\Gamma}^1(\mathcal{M}_{\delta})$. Moreover, for all $\theta\in \T^d$, 
\begin{equation}
\begin{aligned}\label{def:KomegaKf}
&\| A_-(\theta)\|_{\mathcal{L}_{\Gamma}(\ell^{\infty})}, \| A_+(\theta)^{-1} \|_{\mathcal{L}_{\Gamma}(\ell^{\infty})} \le \lambda^{-1}, \qquad 
\| B^{-1}(\theta) \|_{\mathcal{L}_{\Gamma}(\R^d)}\le \beta,\\[2mm]
&\sup_{\substack{j=1, 2}} \| \partial_{\theta}^j A_{\pm}(\theta) \|_{\mathcal{L}_{\Gamma}( \ell^{\infty})},  \sup_{\substack{j=1, 2}} \| \partial_{\theta}^j B(\theta) \|_{\mathcal{L}_{\Gamma}(\R^d)} \le \K,
\end{aligned}
\end{equation}
and
\begin{equation}
\sup_{(x, y, r)\in \mathtt{B}_{\delta}}\|  D \omega(x, y, r) \|_{\mathcal{L}_{\Gamma}(\ell^{\infty}\times\ell^{\infty}\times \R^d; \R^d)}\le \K.
\end{equation}

\item[$\mathbf{(H2)_{C^1}}$] Assume $\bf{(H2)_{\mathrm{lip}}}$. For $j=1, 2, 4$ we have that $f_j(\nu; 0, 0, \theta, 0)=0$ and for all $w\in \mathcal{M}_{\delta}$
\begin{align*}
&\| D f_k (\nu; w) \|_{\mathcal{L}_{\Gamma}(T\mathcal{M}; \ell^{\infty})}\le L(\delta, \mu), \qquad k=1, 2, \\
&\| D f_4 (\nu; w) \|_{\mathcal{L}_{\Gamma}(T\mathcal{M}; \R^d)}\le L(\delta, \mu),
\end{align*}
where $T\mathcal{M}$ is the tangent space of $\mathcal{M}$, which is isomorphic to $\ell^{\infty}\times \ell^{\infty}\times \R^d\times \R^d$.
Moreover the derivatives with respect to $\theta$ have the following bounds:
\begin{equation}
\begin{aligned}\label{bound:dthetafk}
&\| \partial_{\theta} f_k(\nu; x, y, \theta, r) \|_{\mathcal{L}_{\Gamma}(\R^d; \ell^{\infty})}\le \K (\|  x \|_{\ell^{\infty}}+\| y \|_{\ell^{\infty}}+|r|_d)\,, \quad k=1, 2, \\
&\| \partial_{\theta} f_4(\nu; x, y, \theta, r) \|_{\mathcal{L}_{\Gamma}(\R^d)}\le \K (\|  x \|_{\ell^{\infty}}+\| y \|_{\ell^{\infty}}+|r|_d)\,.
\end{aligned}
\end{equation}
\item[$\mathbf{(H3)_{C^1}}$] Assume $\bf{(H3)_{\mathrm{lip}}}$. The function $f_3$ satisfies the following
\begin{align*}
\sup_{w\in \mathcal{M}_{\delta}}\| \partial_{\theta} f_3(w) \|_{\mathcal{L}_{\Gamma}(\R^d)} &\le K_{\theta},\\
\sup_{w\in \mathcal{M}_{\delta}}\|  D f_3(\nu; w) \|_{\mathcal{L}_{\Gamma}(\mathcal{M}; \R^d)} &\le \K.
\end{align*}
\item[$\mathbf{(H4)_{C^1}}$] For $j=1,2, 4$ the derivatives of the function $f_j$ are Lipschitz on $\cM_\de$ and \begin{align}
\lip_{x, v} \partial_s f_j &\le \K, \quad s=x, v, \theta, \nonumber\\
\lip_v \partial_{\theta} f_j & \le \K ,   \label{bound:sn} \\ 
 \lip_{\theta} \partial_s f_j &\le L(\delta, \mu), \quad s=x, v, \nonumber\\  \nonumber
\lip_{\theta} \partial_{\theta} f_j (x,y,r) &\le \K (\| x \|_{\ell^{\infty}}+\|  y \|_{\ell^{\infty}}+| r |_d),\\
\lip\, \partial_s\, \omega &\le \K, \quad s=x, v \nonumber.
\end{align}
Moreover, the derivatives of the function $f_3$ are Lipschitz on $\cM_\de$, more precisely
\begin{align*}
&\lip\,\, \partial_s f_3\le \K, \qquad s=x, y, \theta, r.
\end{align*}
\item[$\mathbf{(H5)_{C^1}}$] Assume $\bf{(H4)_{\mathrm{lip}}}$. The derivatives $\partial_{\nu} f_j$, $j=1, 2, 3, 4$, satisfy the same estimates of the derivatives with respect to the angles $\partial_{\theta} f_j$ appearing in $\mathbf{(H2)_{C^1}} - \mathbf{(H5)_{C^1}}$. 
\end{itemize}
\begin{remark}
The assumption $\mathbf{(H1)_{C^1}}$ could be weakened by requiring that $\partial_{\theta} A_{\pm}$, $\partial_{\theta} B$ are Lipschitz functions, instead of $C^1$. However our model \eqref{def:Hamiltonian} satisfies even stronger assumptions, so we make this choice to simplify the exposition.  
\end{remark}

\begin{theorem}\label{thm:C1case}
Assume $F\colon \mathcal{M}_{\delta}\to \mathcal{M}$ satisfies $\mathbf{(H0)_{C^1}}$-$\mathbf{(H4)_{C^1}}$. Then there exist $\delta_1>0$ and $\mu_1>0$ such that for all $\delta\in (0, \delta_1)$ and $\mu\in (0, \mu_1)$ the function $\gamma_{\nu}^s(x, \theta)$ given by Theorem \ref{thm:lipcase}, whose graph is the stable manifold of $\T_0$, has the following properties.
\begin{itemize}
\item It is $C^1_{\Gamma}(B_{\delta}(\ell^{\infty})\times \T^d; \ell^{\infty}\times \R^d)$ and
\begin{equation}\label{bound:gammaC1}
\| \gamma^s_\nu \|_{C^1_{\Gamma}(B_{\delta}(\ell^{\infty})\times \T^d)}\le \mathcal{O}(\delta+L).
\end{equation}
\item For all $j\in \Z^m$
\[
\gamma_{\nu}^s\colon B_{\delta}(\Sigma_{j, \Gamma})\times \T^d\to \Sigma_{j, \Gamma}\times \R^d
\]
and
\[
\| \gamma^s_\nu \|_{C^1_{\Gamma}(B_{\delta}(\Sigma_{j, \Gamma})\times \T^d)}\le \mathcal{O}(\delta+L).
\]
\end{itemize}
Moreover if we assume $\mathbf{(H5)_{C^1}}$ and that the regularity conditions stated above hold also considering $\nu$ as an additional angle
then $\gamma_{\nu}^s$ is $C^1_{\Gamma}$ with respect to $(x, \theta)$ and $C^1$ with respect to $\nu\in(0,\mu)$.
\end{theorem}

To prove Theorem \ref{thm:C1case}, we need the following preliminary lemmas.

\begin{lem}\label{rem:boundspartialzgammaz}
Recall $v:=(y, r)$ and \eqref{defspaces}. If $\gamma\in \Sigma_{c, M}$ with $c\in (0, 1]$ then by $\mathbf{(H2)_{C^1}}$-$\mathbf{(H3)_{C^1}}$, for $k=1, 2$ and $z\in B_{\delta}(\ell^{\infty})$
\begin{align*}
&\| (\partial_x f_k) (z, \gamma(z)) \|_{\mathcal{L}_{\Gamma}(\ell^{\infty})}, \quad \| (\partial_v f_k) (z, \gamma(z)) \|_{\mathcal{L}_{\Gamma}(\ell^{\infty}\times\R^d; \ell^{\infty})}\le L,\\
&\| (\partial_x f_4) (z, \gamma(z)) \|_{\mathcal{L}_{\Gamma}(\ell^{\infty}; \R^d)}, \quad \| (\partial_v f_4) (z, \gamma(z)) \|_{\mathcal{L}_{\Gamma}(\ell^{\infty}\times\R^d; \R^d)}\le L\\
&\| (\partial_x f_3) (z, \gamma(z)) \|_{\mathcal{L}_{\Gamma}(\ell^{\infty}; \R^d)}, \quad \| (\partial_v f_3) (z, \gamma(z)) \|_{\mathcal{L}_{\Gamma}(\ell^{\infty}\times\R^d; \R^d)}\le \K\\
& \| (\partial_{\theta} f_3) (z, \gamma(z))  \|_{\mathcal{L}_{\Gamma}(\R^d)}\le K_{\theta},\\
&\| (\partial_{\theta} f_k) (z, \gamma(z)) \|_{\mathcal{L}_{\Gamma}(\R^d; \ell^{\infty})}, \quad \| (\partial_{\theta} f_4) (z, \gamma(z)) \|_{\mathcal{L}_{\Gamma}(\R^d)} \le \K\,(1+c)\,\| x \|_{\ell^{\infty}}, \\
& \| (\partial_x \omega) (z, \gamma(z)) \|_{\mathcal{L}_{\Gamma}(\ell^{\infty}; \R^d)}, \quad  \| (\partial_v \omega) (z, \gamma(z)) \|_{\mathcal{L}_{\Gamma}(\ell^{\infty}\times\R^d; \R^d)}\le \K.
\end{align*}
In particular, for $z\in B_{\delta}(\ell^{\infty})$
\begin{equation}\label{bound:partialg}
\| (\partial_x g)(z, \gamma(z)) \|_{\mathcal{L}_{x}}, \quad \| (\partial_v g)(z, \gamma(z)) \|_{\mathcal{L}_{v}}\le 2L, \quad \mathrm{and} \quad \| (\partial_{\theta} g)(z, \gamma(z)) \|_{\mathcal{L}_{\theta}}\le 2 \K (1+c)\,\| x \|_{\ell^{\infty}}.
\end{equation}
\end{lem}
\begin{proof}
The bounds for $\partial_{\theta} f_k$, $k=1, 2, 4$ follow by \eqref{gamma0theta} and $\mathbf{(H2)_{C^1}}$. The other estimates follow by straightforward computations.
\end{proof}

\begin{lem}\label{rem:lippartial}
If $\gamma\in \Xi_{c, M}$ with $c\in (0, 1]$ then for $j=1,2, 4$
\begin{align*}
&\lip_{x, v} (\partial_s f_j(\id, \gamma)) \le \K (1+c), \qquad s=x, v, \theta,\\ 
&\lip_{\theta} (\partial_s f_j(z, \gamma(z)))  \le   L+{\K M \| x \|_{\ell^{\infty}},}  \qquad s=x, v,\\ 
&\lip_{\theta} (\partial_{\theta} f_j(z, \gamma(z))) \le \K \big((1+c)+M\big)\,\| x \|_{\ell^{\infty}}
\end{align*}
and for $s=x, v$
\begin{align*}
&\lip_x (\partial_s f_3(\id, \gamma)),\,\lip_x (\partial_s \omega (\id, \gamma)) \le \K (1+c).
\end{align*}
\end{lem}
\begin{proof}
It follows by \eqref{gamma0theta} and $\mathbf{(H2)_{C^1}}$.
\end{proof}

\noindent \textbf{Proof of Theorem \ref{thm:C1case}}
{The assumptions of Theorem \ref{thm:C1case} imply the Lipschitz assumptions of Theorem \ref{thm:lipcase}. Then if $c\in (0, 1]$ and $M$ satisfies \eqref{condcM} the map $G$ defined in \eqref{G} is a contraction on $\Xi_{c, M}$.  Let $\gamma^s:=\gamma^s_{\nu}$ be its unique fixed point. We want to prove that $\gamma^s$ is $C^1_{\Gamma}(B_{\delta}(\ell^{\infty})\times \T^d)$.}\\


Let $\widetilde{M}, \kappa_0, \kappa_1, \rho_0>0$ and 
\begin{equation}\label{def:tv}
\mathtt{v}:=(\tM, \kappa_0, \kappa_1, \rho_0).
\end{equation}
 We introduce the space (recall \eqref{defspaces})
\[
D \Xi:=\left\{ \Psi:=(\Psi_0, \Psi_1)\in E_x\times E_{\theta} : \| \Psi_0 \|_0<\infty,\,\,\,\| \Psi_1 \|_1<\infty \right\}
\]
where 
\[
\| \Psi_0\|_0:=\sup_{z\in \dom} \| \Psi_0(z) \|_{\mathcal{L}_x}, \qquad \| \Psi_1 \|_1:=\sup_{\substack{z\in \dom,\\ x\neq 0}} \frac{\| \Psi_1(z) \|_{\mathcal{L}_{\theta}}  }{\| x \|_{\ell^{\infty}}}.
\]
We equip $D\Xi$ with the weighted norm
\begin{equation}\label{def:normDSigma}
\| \Psi \|_{D\Xi}:=\alpha_0 \| \Psi_0 \|_0+\alpha_1 \| \Psi_1 \|_1
\end{equation}
with $\alpha_0, \alpha_1>0$ to be opportunely chosen later.
We consider the closed subset
\begin{equation}\label{def:DSigma}
\begin{aligned}
D \Xi_{c, M, \mathtt{v}}:=\big\{ &\Psi:=(\Psi_0, \Psi_1)\in D \Xi : { \lip_x (\Psi_0)\le \kappa_0},\,\,\lip_{\theta} \Psi_0\le \rho_0,\,\,\lip_x(\Psi_1)\le \kappa_1, \\
&\lip_{\theta} \Psi_1(x)\le \widetilde{M} \| x \|_{\ell^{\infty}},\,\,\| \Psi_0\|_0\le c, \,\,\|\Psi_1 \|_1\le M \big\}
\end{aligned}
\end{equation}
where $c, M$ are the constants chosen in Lemma \ref{lemball}.
 We shall use the following Fiber Contraction Theorem.
\begin{theorem}\label{thm:fibrecontraction}
Let $\Xi$ and $D\Xi$ be metric spaces, $D\Xi$ complete, and $\Gamma\colon \Xi\times D\Xi\to \Xi \times D\Xi$ a map of the form $\Gamma(\gamma, \Psi)=(G(\gamma), H(\gamma, \Psi))$. Assume that
\begin{itemize}
\item[(a)] $G\colon \Xi\to \Xi$ has an attracting fixed point $\gamma_{\infty}\in \Xi$ (i.e. $\lim_{k\to \infty} G^k(\gamma)=\gamma_{\infty}, \,\,\forall \gamma\in \Xi$).
\item[(b)] $\limsup_{n\to \infty} \mathrm{Lip} H(G^n(\gamma), \cdot)<1$ for each $\gamma\in \Xi$.
\item[(c)] $H$ is continuous with respect to $\gamma$ at $(\gamma_{\infty}, \Psi_{\infty})$, where $\Psi_{\infty}\in D\Xi$ is a fixed point for $H(\gamma_{\infty}, \cdot)$.
\end{itemize}
Then $(\gamma_{\infty}, \Psi_{\infty})$ is an attracting fixed point of $\Gamma$.
\end{theorem}
To apply the above theorem we look for $H$ such that $D [ G(\gamma)]=H(\gamma, D \gamma)$, where $G(\gamma)$ is the operator introduced in \eqref{G}. For that we differentiate formally $G(\gamma)$ and we substitute $\partial_x \gamma$, $\partial_{\theta}\gamma$ with $\Psi_0$ and $\Psi_1$.\\
In this way we obtain the map $H(\gamma, \Psi)=(H_0(\gamma, \Psi), H_1(\gamma, \Psi))$ defined by 
\begin{equation}\label{def:H0}
\begin{aligned}
H_0(\gamma, \Psi)(z)=&C(\theta)^{-1} \Big\{  (\partial_x g)\,(z, \gamma(z))+(\partial_v g)(z, \gamma(z)) \, \Psi_0(z)\\
&+\Psi_0(h (z))
\Big( A_-(\theta)+(\partial_x f_1)(z, \gamma(z))+(\partial_v f_1)(z, \gamma(z)) \Psi_0(z) \Big)\\
&+\Psi_1(h (z)) \Big( {(\partial_x \omega)(\pi_{(x, v)} \big(z, \gamma(z) \big))+(\partial_v \omega)(\pi_{(x, v)} \big(z, \gamma(z) \big))\,\Psi_0(z)}\\
&+  (\partial_x f_3)(z, \gamma(z))+(\partial_v f_3)(z, \gamma(z)) \Psi_0(z) \Big) \Big\},\\
H_1(\gamma, \Psi)(z)=&\partial_{\theta} \big(C(\theta)^{-1}\big)  \{ g(z, \gamma(z))+\gamma(h(z)) \} +C(\theta)^{-1} \Big\{  (\partial_{\theta} g)\,(z, \gamma(z))+(\partial_v g)(z, \gamma(z)) \, \Psi_1(z)\\
&+\Psi_0(h (z))
\Big( {\big(\partial_{\theta} A_-(\theta)\big) x}+ (\partial_{\theta} f_1)(z, \gamma(z))+(\partial_v f_1)(z, \gamma(z)) \Psi_1(z) \Big)\\
&+\Psi_1(h (z)) \Big( \id+{(\partial_v \omega)(\pi_{(x, v)} \big(z, \gamma(z)\big))\,\Psi_1(z)}\\
&+ (\partial_{\theta} f_3)(z, \gamma(z))+(\partial_v f_3)(z, \gamma(z)) \Psi_1(z) \Big) \Big\}.
\end{aligned}
\end{equation}
We prove the following: for an opportune choice of the parameters $\mathtt{v}=(\tM, \kappa_0,  \rho_0, \kappa_1)$ we have
\begin{itemize}
\item[(i)] $H\colon \Xi_{c, M}\times D \Xi_{c, M, \mathtt{v}}\to D\Xi_{c, M, \mathtt{v}}$ is well defined (see Lemma \ref{lem:Hwelldef}).
\item[(ii)] $H(\gamma, \cdot)\colon D \Xi_{c, M, \mathtt{v}}\to D \Xi_{c, M, \mathtt{v}}$ is a contraction (see Lemma \ref{lem:uniformcontraction}).
\item[(iii)] $H(\cdot, \Psi)\colon \Xi_{c, M}\to D\Xi_{c, M, \mathtt{v}}$ is continuous (see Lemma \ref{lemcont}).
\end{itemize}
Let us set
\begin{align*}
\mathfrak{A}(\kappa_1, \rho_0):=&\,\K \beta (1+c)^2 (2+c) +2 \K  \beta\,\lambda^{-1} (1+c) (\kappa_1+\rho_0)\\
\mathfrak{B}(\rho_0, \kappa_1, M):=&\,K \beta^2 \Big( (1+K) c \lambda^{-1}+ 2 K (1+c)(2+c)+2Kc+2 M \lambda^{-1} (1+K_{\theta}) \Big)\\
&+\beta \Big\{ 2 \K(1+c)+2 \K M (2+c)+c (K+\K(1+c)+\K M (2+c))\\& +2 \K\,M \lambda^{-1}+K (\rho_0+\kappa_1) (2+c) (1+K_{\theta})\Big\}.
\end{align*}

\begin{lem}\label{lem:Hwelldef}
Recall \eqref{assumption:constants}, \eqref{def:SigmacM}, \eqref{condcM}. If
\begin{equation} \label{condition:r}
\begin{aligned}
&\rho_0>\frac{\beta c \K(1+\beta \lambda^{-1})}{1-\beta\lambda^{-1} (1+K_{\theta})} \\
&\kappa_1\geq \frac{\beta c \K(1+\beta \lambda^{-1})+\beta \K (1+c) (2+c)}{1-\beta\lambda^{-1}(1+K_{\theta})},\\
&\kappa_0\geq  \frac{\mathfrak{A}(\kappa_1, \rho_0)}{1-\beta\lambda^{-2}} ,\\
&\widetilde{M}\geq \frac{\mathfrak{B}(\rho_0, \kappa_1, M)}{1-\beta \lambda^{-1} (1+K_{\theta})^2}.  
\end{aligned}
\end{equation}
Then, for any $\gamma\in \Xi_{c, M}$ and $\Psi\in D\Xi_{c, M, \mathtt{v}}$,  $H(\gamma, \Psi)\in D \Xi_{c, M, \mathtt{v}}$.
\end{lem}
\begin{remark}
We observe that in order to fulfill the conditions \eqref{condition:r} for given $c$ and $M$, it is sufficient to fix (in this order) $\rho_0, \kappa_1$, $\kappa_0$ and $\widetilde{M}$.
\end{remark}
\begin{proof}[Proof of Lemma \ref{lem:Hwelldef}]
We observe that $H(\gamma, \Psi)\in E_x\times E_{\theta}$ since, by assumption $\mathbf{(H1)_{C^1}}$, $H_0(\gamma, \Psi)$ and $H_1(\gamma, \Psi)$ are compositions of continuous functions.
We prove that $\| H_0(\gamma, \Psi) \|_0\le c$.  By \eqref{h1Lx}, \eqref{bound:partialg} and Lemma \ref{rem:boundspartialzgammaz}, for $z\in B_\de(\ell^\infty)\times\T^d$,
\[
\|   H_0(\gamma, \Psi)(z) \|_{\mathcal{L}_x}\le \beta \Big\{ 2L(1+c)+c (\lambda^{-1}+L(1+c))+2 M \K (\lambda^{-1}+L(1+c)) (1+c)\| x \|_{\ell^{\infty}}   \Big\}.
\]
Hence, by taking $\delta$ and $\mu$ small enough and recalling that $\beta  \lambda^{-1}<1$ (see \eqref{bound:betalambda}), we get  $\| H_0(\gamma, \Psi) \|_0\le c$. Now we prove that $\lip_x H_0(\gamma, \Psi)\le \kappa_0$. We remark that by Lemma \ref{lemh} we have (recall \eqref{h})
\begin{equation}
\begin{aligned}\label{bound:lipxh}
&\lip_x (\Psi_0\circ h)\le {\kappa_0\,\big( \lambda^{-1}+L(1+c) \big)+2 \K \rho_0 (1+c) },\\
&\lip_x (\Psi_1\circ h)\le \kappa_1 \big( \lambda^{-1}+L(1+c)\big)+2\K \widetilde{M} \delta (\lambda^{-1}+L(1+c)) (1+c).
\end{aligned}
\end{equation}
By Hypothesis $\mathbf{(H2)_{C^1}}$, the bounds \eqref{bound:lipxh}, \eqref{h1Lx}, Lemmas \ref{rem:boundspartialzgammaz} and \ref{rem:lippartial} we have
\begin{equation}
\begin{aligned}
\lip_x H_0(\gamma, \Psi)&\le \beta \Big\{ 2 \K (1+c)^2+2 \kappa_0 L+\big[\kappa_0 \big( \lambda^{-1}+L(1+c) \big)^2+2\rho_0\,\K (1+c) \big]( \lambda^{-1}+L(1+c))\\
&+c \big( \K (1+c)^2+L \kappa_0 \big) \\
&+2\K (1+c)\Big[ \kappa_1 \big( \lambda^{-1}+L(1+c)\big)+2\widetilde{M} \K \delta (\lambda^{-1}+L(1+c)) (1+c) \Big]\\
&+M \delta (\lambda^{-1}+L(1+c))  \big( 2\K(1+c)^2+2 \K \kappa_0 \big)\Big\}.
\end{aligned}
\end{equation}
Taking $\delta$ and $\mu$ small enough the right hand side of the above inequality becomes
\begin{align*}
&\K \beta (1+c)^2 (2+c)+\kappa_0\,\beta\,\lambda^{-2} +2 \K (\kappa_1+\rho_0) \beta\,\lambda^{-1} (1+c).
\end{align*}
Hence by the choice of $\kappa_1$ in \eqref{condition:r} we get that $\lip_x H_0(\gamma, \Psi)\le \kappa_0$. By Lemma \ref{lemh}, for $x\in B_\delta(\ell^\infty)$,
\begin{equation}\label{bound:psihtheta}
\begin{aligned}
\lip_{\theta} (\Psi_0\circ h)(x) &\le {\kappa_0 \| x \|_{\ell^{\infty}}  \Big( \K (2+c)+LM \Big)+\rho_0\,\Big(1+K_{\theta}+2 \K M \| x \|_{\ell^{\infty}} \Big)},\\
\lip_{\theta} (\Psi_1\circ h)(x) &\le  \Big[\kappa_1 (\K (2+c)+LM) +\widetilde{M} (\lambda^{-1}+L(1+c))  (1+K_{\theta}+2 \K M \| x \|_{\ell^{\infty}})\Big] \| x \|_{\ell^{\infty}}.
\end{aligned}
\end{equation}
By Hypothesis $\mathbf{(H4)_{C^1}}$, \eqref{bound:psihtheta} and Lemma \ref{rem:lippartial}
\begin{equation}
\begin{aligned}
\lip_{\theta} H_0(\gamma, \Psi)&\le \K \beta^2 \Big\{ 2L(1+c)+c(\lambda^{-1}+L(1+c))\\
&+2 M \K \delta (\lambda^{-1}+L(1+c))   (1+c) \Big\}\\
&+\beta \Big\{ 2 L(1+c)+2 \K M \delta (1+c) +2L \rho_0\\
&+{\kappa_0\Big[ \Big( \K (2+c)+LM \Big) \delta+\rho_0\,\Big(1+K_{\theta}+2 \K M \delta \Big)} \Big](\lambda^{-1}+L(1+c))\\
&+c  (\K+ L(1+c)+ \K M \delta (1+c)+L \rho_0)\\
&+2 \K (1+c)\Big[\kappa_1 \delta (\K (2+c)+LM) +\widetilde{M} \delta (\lambda^{-1}+L(1+c))  (1+K_{\theta}+2 \K M \delta)\Big] \\
&+M \delta (\lambda^{-1}+L(1+c))  \big( \K (1+2 M \delta)(1+c)+2\K \rho_0 \big)   \Big\}.
\end{aligned}
\end{equation}
Taking $\delta$ and $\mu$ small enough in the right hand side of the above inequality becomes
\[
\rho_0\beta \lambda^{-1} (1+K_{\theta})+\beta c \K(1+\beta \lambda^{-1}).
\]
Hence by the choice of $\rho_0$ in \eqref{condition:r}, we get that $\lip_{\theta} H_0(\gamma, \Psi)\le \rho_0$.

Now we prove that $\| H_1(\gamma, \Psi) \|_1\le M$. By Lemma \ref{lemtec}, for all $z\in B_\de(\ell^\infty)\times\T^d$, 
\begin{equation}\label{bound:g0}
\| g(z, \gamma(z)) \|_{\ell^{\infty}\times \R^d}\le 2 L (1+c) \| x \|_{\ell^{\infty}}. 
\end{equation}
Then by \eqref{h1Lx}, \eqref{bound:gammah}
\begin{align*}
\| H_1(\gamma, \Psi)(z) \|_{\mathcal{L}_{\theta}} &\le \K \beta^2\, [2L (1+c)+c (\lambda^{-1}+L\,(1+c))] \| x \|_{\ell^{\infty}}\\
&+\beta \Big\{  2 \K (1+c)+2LM+c(\K (2+c)+LM)\\
&+M(\lambda^{-1}+L(1+c)) (1+K_{\theta}+2 \K  M \| x \|_{\ell^{\infty}}) \Big\} \| x \|_{\ell^{\infty}}.
\end{align*}
Taking $\delta$ and $\mu$ small enough, the right hand side of the above inequality becomes
\[
c \K \beta^2 \lambda^{-1}+M\beta \lambda^{-1} (1+K_{\theta})+\beta \K ( c+ (1+c)(2+c)).
\]
Hence by \eqref{condcM} we get that $ \| H_1(\gamma, \Psi) \|_1\le M$. 

Now we prove that $\lip_x H_1(\gamma, \Psi)\le \kappa_1$. We observe that by Lemma \ref{lemh}, for all $z\in B_\de(\ell^\infty)\times\T^d$,
\[
\lip_x \gamma(h(z))\le c \big(\lambda^{-1}+L(1+c) \big)+2 M \K \delta(\lambda^{-1}+L(1+c))  (1+c).
\]
Then by $\mathbf{(H4)_{C^1}}$, bound \eqref{bound:lipxh}, Lemmas \ref{rem:boundspartialzgammaz} and \ref{rem:lippartial}
\begin{align*}
\lip_x H_1(\gamma, \Psi)&\le \K \beta^2 \Big(2  L(1+c)+c (\lambda^{-1}+L(1+c))+2 M \K \delta(\lambda^{-1}+L(1+c))(1+c) \Big)\\
&+\beta \Big\{ 2 \K (1+c) (1+M \delta)+2 \kappa_1 L\\
&+\Big[{\kappa_0 \Big( \lambda^{-1}+L(1+c)\Big)+2 \K \rho_0\,(1+c)} \Big](\K (2+c) +LM) \delta\\
&+c \Big( \K+\K (1+c)(1+M \delta)+L \kappa_1 \Big)\\
&+(\lambda^{-1}+L(1+c)) \Big[ \Big( \kappa_1+2\widetilde{M} \K \delta(1+c) \Big)(1+2 M \K  \delta+K_{\theta})\\
& +M \delta
\Big( 2 \K \kappa_1+\K (1+c)(1+2 M \delta) \Big) \Big] \Big\}.
\end{align*}
Taking $\delta$ and $\mu$ small enough, the right hand side of the above inequality becomes
\[
\beta c \K (1+\beta \lambda^{-1})+\beta \K (1+c) (2+c)+\beta \kappa_1 \lambda^{-1} (1+K_{\theta}).
\]
Hence by the choice of $\kappa_1$ in \eqref{condition:r} we get that $\lip_x H_1(\gamma, \Psi)\le \kappa_1$.
Now we prove that $\lip_{\theta} H_1(\gamma, \Psi)(x)\le \widetilde{M} \| x \|_{\ell^{\infty}} $.
We observe that 
\begin{equation*}
\lip_{\theta} \partial_{\theta} \big(C(\theta)\big)^{-1}\le  \K \beta^2 (1+2 \beta \K).
\end{equation*}
 By Hypothesis $\mathbf{(H2)_{C^1}}$, Lemma \ref{rem:lippartial}, bounds \eqref{bound:gammah}, \eqref{bound:psihtheta}, \eqref{bound:g0} the function $H_1$ in \eqref{def:H0} satisfies
\begin{align*}
&\lip_{\theta} H_1(\gamma, \Psi)(x) \le \K \beta^2 (1+2 \beta \K)\Big( 2L (1+c)+c (\lambda^{-1}+L(1+c))  \Big) \| x \|_{\ell^{\infty}}\\
&+2 \K \beta^2 \Big(  2 \K (1+c) +(2+c) LM+\K c (2+c)+M (\lambda^{-1}+L(1+c)) (1+K_{\theta}+2 \K M \| x \|_{\ell^{\infty}})  \Big)\| x \|_{\ell^{\infty}}\\
&+\beta\Big\{ 2 \K(1+c)+2 \K M+2 \K (1+c) M+(2L+2\K M \| x \|_{\ell^{\infty}}) M+2 L\widetilde{M}\\
&+\Big[\kappa_0 \| x \|_{\ell^{\infty}} \Big( \K (2+c)+LM \Big)+\rho_0\,\Big(1+K_{\theta}+2 M \K \| x \|_{\ell^{\infty}} \Big)\,\Big]\,(\K (2+c)+LM)  \\
&+c\big(\K+\K(1+c)+\K M+(L+\K M \| x \|_{\ell^{\infty}}) M+L \widetilde{M}  \big) \\
&+\Big[\kappa_1 (\K (2+c)+LM) +\widetilde{M} (\lambda^{-1}+L(1+c)) (1+K_{\theta}+2 M \K \| x \|_{\ell^{\infty}})\Big] \,(1+K_{\theta}+2 M \K \| x \|_{\ell^{\infty}}) \\
&+M (\lambda^{-1}+L(1+c)) \Big(  \K (1+c) (1+M \| x \|_{\ell^{\infty}})+\K M^2 \| x \|_{\ell^{\infty}}+2  \widetilde{M} \K \| x \|_{\ell^{\infty}} \Big) \Big\}  \| x \|_{\ell^{\infty}} .
\end{align*}
Taking $\delta$ and $\mu$ small enough, the right hand side of the above inequality becomes
\begin{align*}
&\K \beta^2 \Big( (1+2 \beta \K) c \lambda^{-1}+ 2 \K (1+c)(2+c)+2 c \K+2 M \lambda^{-1} (1+K_{\theta}) \Big)\\
&+\beta \Big\{ 2 \K (1+c)+2 \K M (2+c)+\rho_0 (1+K_{\theta}) \K (2+c)\\
&+c (\K+\K (1+c)+\K M ) + \K (1+c)\,M \lambda^{-1}\Big\}+\beta \kappa_1 \K (2+c) (1+K_{\theta}) +\beta (1+K_{\theta})^2 \widetilde{M} \lambda^{-1}.
\end{align*}
By \eqref{condition:r} we conclude.
\end{proof}

\begin{lem}\label{lem:uniformcontraction}
Assume that $\delta$ and $\mu$ are small enough and that $\alpha_0, \alpha_1$ satisfy
\begin{equation}\label{cond:alpha01}
\frac{\alpha_0}{\alpha_1}>\frac{ \beta \K (2+c)}{1-\beta\lambda^{-1}}.
\end{equation}
 Then $H(\gamma, \cdot)$ in \eqref{def:H0} is a contraction, uniformly with respect to $\gamma$.
\end{lem}

\begin{proof}
By Lemma \ref{rem:boundspartialzgammaz} we have
\begin{align*}
\| H_0(\gamma, \Psi)(z)-H_0(\gamma, \Psi')(z) \|_{\mathcal{L}_x} \le&\, \beta \Big\{ \Big( L(2+c)+2 M \K (\lambda^{-1}+L(1+c))\| x \|_{\ell^{\infty}}  \Big) \| \Psi_0(z)-\Psi_0'(z)\|_{\mathcal{L}_x}\\
&+ \| \Psi_0(h(z))- \Psi'_0(h(z)) \|_{\mathcal{L}_x} \Big( \lambda^{-1}+{L(1+c)} \Big)\\
&+ \| \Psi_1(h(z))- \Psi'_1(h(z)) \|_{\mathcal{L}_{\theta}}  \,2K (1+c) \Big\} ,\\
\| H_1(\gamma, \Psi)(z)-H_1(\gamma, \Psi')(z) \|_{\mathcal{L}_{\theta}}\le&\, \beta \Big\{ \Big(L (2+c)+2KM (\lambda^{-1}+L(1+c)) \| x \|_{\ell^{\infty}}\Big) \| \Psi_1(z)-\Psi_1'(z) \|_{\mathcal{L}_{\theta}} \\
&+ \| \Psi_0(h(z))- \Psi'_0(h(z)) \|_{\mathcal{L}_x} \big( K (2+c)+LM \big) \| x \|_{\ell^{\infty}} \\
&+ \| \Psi_1(h(z))- \Psi'_1(h(z)) \|_{\mathcal{L}_{\theta}} \Big( 1+K_{\theta}+2K M \| x \|_{\ell^{\infty}} \Big)  \Big\}.
\end{align*}
Now we observe that, for all $z\in \dom$ (recall the bound \eqref{h1Lx}),
\[
 \| \Psi_0(h(z))- \Psi'_0(h(z)) \|_{\mathcal{L}_x} \le \| \Psi_0-\Psi'_0 \|_0, \qquad  \| \Psi_1(h(z))- \Psi'_1(h(z)) \|_{\mathcal{L}_{\theta}}\le  \| \Psi_1- \Psi'_1 \|_{1} (\lambda^{-1}+L(1+c)) \| x \|_{\ell^{\infty}}.
 \]
 Then,
 \begin{align*}
\| H_0(\gamma, \Psi)-H_0(\gamma, \Psi') \|_0 \le& \,\| \Psi_0-\Psi_0'\|_0\, \beta \Big( \lambda^{-1}+ L(3+2c)+2KM\delta (\lambda^{-1}+L(1+c))  \Big)\\
&+ \| \Psi_1- \Psi'_1 \|_{1} \beta\, (\lambda^{-1}+L(1+c))\,2K\delta\, (1+c) \\
\| H_1(\gamma, \Psi)-H_1(\gamma, \Psi') \|_1 \le&\, \| \Psi_1-\Psi_1'\|_1\, \beta \Big[L (2+c)+(\lambda^{-1}+L(1+c)) \, \Big( 1+K_{\theta}+4KM \delta \Big)  \Big]\\
&+ \| \Psi_0- \Psi'_0 \|_{0}\,\, \beta\,\big(K (2+c)+LM \big) .
 \end{align*}
 Therefore (recall \eqref{def:normDSigma})
 \begin{align*}
 \| H(\gamma, \Psi)-H(\gamma, \Psi') \|_{D\Sigma} &\le \| \Psi_0-\Psi'_0 \|_0 \Big( \alpha_0 (\beta \lambda^{-1}+\mathcal{O}(L)+\mathcal{O}(\delta)) +\alpha_1 \beta (\K (2+c)+\mathcal{O}(L)) \Big)\\
 &+ \| \Psi_1-\Psi'_1 \|_1 \Big( \alpha_0 \mathcal{O}(\delta)+\alpha_1 \big( \beta \lambda^{-1} (1+K_{\theta})+\mathcal{O}(L)+\mathcal{O}(\delta) \big)  \Big).
 \end{align*}
Hence, taking $\delta$ and $\mu$ small enough, we require that
 \[
 \alpha_0 \beta \lambda^{-1}+\alpha_1 \beta K (2+c)<\alpha_0, \qquad \alpha_1 \beta \lambda^{-1} (1+K_{\theta})<\alpha_1
 \]
 to prove that $H(\gamma, \cdot)$ is a contraction as a map on the space $D\Xi$. By \eqref{cond:alpha01} and \eqref{assumption:constants} the first and the second inequality respectively hold. 
\end{proof}

\begin{lem}\label{lemcont}
The function $\gamma\in (\Xi_{c, M}, \| \cdot \|_1)\mapsto H(\gamma, \Psi)\in (D\Xi, \| \cdot \|_{D\Xi})$ is continuous.
\end{lem}
\begin{proof}
We have to see that $\alpha_0 \| H_0(\gamma, \Psi)-H_0(\gamma', \Psi) \|_0+\alpha_1 \| H_1(\gamma, \Psi)-H_1(\gamma', \Psi) \|_1$ is small if $\|\gamma-\gamma'\|_1$ is small. Decomposing the difference in telescopic form, the more delicate term to control is the one containing $\| \Psi_1 (h(\gamma'))-\Psi_1(h(\gamma))\|_1$. By \eqref{hgamma-gamma'},  Lemma \ref{rem:boundspartialzgammaz} we have (recall the definitions \eqref{h}, \eqref{def:DSigma})
\begin{align*}
\| x \|_{\ell^{\infty}}^{-1} \| \Psi_1 (h(\gamma'))(z)-\Psi_1(h(\gamma))(z)\|_{\mathcal{L}_{\theta}} &\le  \| x \|_{\ell^{\infty}}^{-1}  \| \Psi_1 (h_1(\gamma'), h_2(\gamma'))(z)-\Psi_1(h_1(\gamma), h_2(\gamma'))(z)\|_{\mathcal{L}_{\theta}}\\
&+\| x \|_{\ell^{\infty}}^{-1}  \| \Psi_1 (h_1(\gamma), h_2(\gamma'))(z)-\Psi_1(h_1(\gamma), h_2(\gamma))(z)\|_{\mathcal{L}_{\theta}}\\
&\le \Big(\kappa_1 L+2\widetilde{M} \K \delta\,(\lambda^{-1}+L(1+c)) \Big)  \| \gamma-\gamma'\|_1.
\end{align*}
The other terms can be estimated analogously.

\end{proof}

\begin{remark}\label{rem:smallDgamma}
By the definition \eqref{def:DSigma} and Remark \ref{rem:csmallL} we have that
\[
\sup_{(z,\nu)\in\dom\times (0,\mu)}\| D \gamma^{s}_\nu(z) \|_{{\mathcal{L}_{\Gamma}(\ell^{\infty}\times \R^d)}}=\mathcal{O}(\delta+L(\de,\mu)).
\]
\end{remark}

We conclude the proof of Theorem \ref{thm:C1case} by applying the Fiber Contraction Rheorem \ref{thm:fibrecontraction}.

Let $c, M, \widetilde{M}, \kappa_0, \rho_0, \kappa_1$ satisfy $c\in (0, 1]$,  \eqref{condcM}, \eqref{condition:r}.
Now we prove that $\Gamma:=(G, H)\colon \Xi_{c, M}\times D \Xi_{c, M, \mathtt{v}}\to \Xi_{c, M}\times D \Xi_{c, M, \mathtt{v}}$, with $G$ defined in \eqref{G} and $H$ in \eqref{def:H0} satisfies the assumptions of the above theorem. The hypothesis (a) holds because $\gamma^s$ is an attracting fixed point of $G$ on $\Xi_{c, M}$. The hypothesis (b)-(c) hold by Lemmas \ref{lem:uniformcontraction} and \ref{lemcont} respectively. Then there exists an attracting fixed point $(\gamma_{\infty}, \Psi_{\infty})$ for $\Gamma$  and, by uniqueness, $\gamma_{\infty}=\gamma^s$. Now we recall that by the definition of $H$ in \eqref{def:H0} we have that, taking for instance $\gamma_0=0, \Psi_0=0$, the iterates $\Gamma^j(\gamma_0, \Psi_0)=(\gamma_j, \Psi_j)$ are such that $\Psi_j=D \gamma_j\in D \Xi_{c, M, \mathtt{v}}$. Then by  the definition of $E_x$, $E_{\theta}$ in \eqref{defspaces} and the definition of the space $D\Xi_{\mathtt{c}, M, \mathtt{v}}$ in \eqref{def:DSigma} the function $\gamma_j$ {belongs to the ball of radius $c$ \footnote{Provided that $M \delta\le c$, hence for $\delta$ small enough.} of  $C^1_{\Gamma}(\dom)$ for all $j\geq 0$.}\\
 Since $\gamma_j$ and $\Psi_j$ converge in the uniform $C^0$-topology we have that $\Psi_{\infty}=D \gamma_{\infty}=D \gamma^s$. This proves that $\gamma^s$ is $C^1$. By Lemma \ref{lem:limitdecay} we conclude that $\gamma^s\in C^1_{\Gamma}$. By Remark \ref{rem:smallDgamma} we obtain the bound \eqref{bound:gammaC1}. This concludes the proof of the first item of Theorem \ref{thm:C1case}.
 
 We recall that $\gamma^s(0, \theta)=0$. Moreover, denoting by  $\gamma_{y}^s$ the $y$ component of  $\gamma^s$, by Lemma \ref{lem:fm1}, it satisfies $\gamma_{y}^s(\cdot , \theta)\colon \Sigma_{j, \Gamma}\to \Sigma_{j, \Gamma}$. This concludes the proof of the second item of Theorem \ref{thm:C1case}.

\subsubsection{$C_{\Gamma}^2$ regularity of the invariant manifolds}\label{subsec:C2case}
Recall the definitions \eqref{defspaces} and  \eqref{def:klinearmaps}. Let us define
\begin{align*}
&\mathcal{L}^2_{xx}:=\mathcal{L}^2_{\Gamma}(\ell^{\infty};  \ell^{\infty}\times \R^d),\qquad &E_{xx}:=C^0(B_{\delta}(\ell^{\infty})\times \T^d; \mathcal{L}^2_{xx}),\\
&\mathcal{L}^2_{\theta \theta}:=\mathcal{L}^2_{\Gamma}(\R^d; \ell^{\infty}\times \R^d),\qquad &E_{\theta \theta}:=C^0(B_{\delta}(\ell^{\infty})\times \T^d; \mathcal{L}^2_{\theta \theta}),\\
&\mathcal{L}^2_{x \theta}:=\mathcal{L}^2_{\Gamma}(\ell^{\infty}, \R^d; \ell^{\infty}\times \R^d),\qquad &E_{x\theta}:=C^0(B_{\delta}(\ell^{\infty})\times \T^d;
 \mathcal{L}^2_{x\theta}).
\end{align*}
 We assume that there exist constants $\K, K_{\theta}$ such that 
 \begin{itemize}
\item[$\mathbf{(H0)_{C^2}}$]  We have
 \begin{equation}\label{condition:lambdastrong}
\lambda^{-1} \beta (1+K_{\theta})^3<1.
\end{equation}
\item[$\mathbf{(H1)_{C^2}}$] Assume $\mathbf{(H1)_{C^1}}$. The functions $A_{\pm}, B\in C^3_{{\Gamma}}(\T^d; \mathcal{L}_{\Gamma}(\ell^{\infty})), \omega\in C^2_{\Gamma}(\tB(\delta); \T^d)$ and $f=f_{\nu}$ are $C_{\Gamma}^2(\mathcal{M}_{\delta})$ with respect to $w$. Moreover for all $\theta\in \T^d$
\begin{equation}\label{def:KomegaKf2}
\begin{aligned}
&\| A_-(\theta)\|_{\mathcal{L}_{\Gamma}(\ell^{\infty})}, \| A_+(\theta)^{-1} \|_{\mathcal{L}_{\Gamma}(\ell^{\infty})} \le \lambda^{-1}, \qquad 
\| B^{-1}(\theta) \|_{\mathcal{L}_{\Gamma}(\R^d)}\le \beta,\\[2mm]
&\sup_{\substack{j=1, 2}} \| \partial_{\theta}^j A_{\pm}(\theta) \|_{\mathcal{L}_{\Gamma}( \ell^{\infty})},  \sup_{\substack{j=1, 2}} \| \partial_{\theta}^j B(\theta) \|_{\mathcal{L}_{\Gamma}(\R^d)} \le \K,
\end{aligned}
\end{equation}
and
\begin{align*}
\sup_{w\in \mathcal{M}_{\delta}}    \sup_{j=1, 2}   \|  D^j f_3(\nu; w) \|_{\mathcal{L}^j_{\Gamma}(T\mathcal{M}; \R^d)}&\le \K,\\[2mm]
\sup_{(x, y, r)\in \tB_{\delta}} \sup_{j=1, 2}  \|  D^j \omega(x, y, r) \|_{\mathcal{L}^j_{\Gamma}(\ell^{\infty}\times\ell^{\infty}\times \R^d; \R^d)}&\le \K.
\end{align*}

\item[$\mathbf{(H2)_{C^2}}$]
Assume $\mathbf{(H2)_{C^1}}$.
For $w\in\cM_\de$, we have
\begin{align*}
\| D^2 f_k(\nu; w) \|_{\mathcal{L}^2_{\Gamma}(T \mathcal{M}; \ell^{\infty})} & \le \K, \quad k=1, 2,\\
\| D^2 f_4(\nu; w) \|_{\mathcal{L}^2_{\Gamma}(T \mathcal{M}; \R^d)} & \le \K,\\
\| \partial^2_{\theta} f_k(\nu; x, y, \theta, r) \|_{\mathcal{L}^2_{\Gamma}(\R^d; \ell^{\infty})} &\le \K (\|  x \|_{\ell^{\infty}}+\| y \|_{\ell^{\infty}}+|r|_d)\,, \quad k=1, 2, \\
\| \partial^2_{\theta} f_4(\nu; x, y, \theta, r) \|_{\mathcal{L}^2_{\Gamma}(\R^d)} &\le \K (\|  x \|_{\ell^{\infty}}+\| y \|_{\ell^{\infty}}+|r|_d)\,.
\end{align*}

\item[$\mathbf{(H3)_{C^2}}$] Assume $\mathbf{(H3)_{C^1}}$-$\mathbf{(H4)_{C^1}}$. The second order derivatives of $f_{\nu}$ are Lipschitz on $\cM_\de$, more precisely
\begin{align*}
 \lip\,\, \partial^2_{s, s'} f_j &\le  \K, \qquad s, s'=x, y, \theta, r, \qquad  j=1, 2, 3, 4,\\
 \lip_{\theta} \partial_{\theta}^2 f_j(w) &\le \K (\| x \|_{\ell^{\infty}}+\| y\|_{\ell^{\infty}}+|r|_d  ),  \qquad j=1, 2, 4, \\
 \lip\,\, \partial^2_{s, s'} \omega &\le \K, \qquad s, s'=x, y, r. 
\end{align*}
\item[$\mathbf{(H4)_{C^2}}$] Assume $\mathbf{(H5)_{C^1}}$. The function $f$ is $C^2$ with respect to $\nu$ and we have 
\begin{align*}
\|  \partial^2_{\nu} f_k(\nu; x, y, \theta, r) \|_{\ell^{\infty}} &\le \K (\|  x \|_{\ell^{\infty}}+\| y \|_{\ell^{\infty}}+|r|_d), \quad   k=1, 2,\\
 |  \partial^2_{\nu} f_4(\nu; x, y, \theta, r) |_{d} &\le \K (\|  x \|_{\ell^{\infty}}+\| y \|_{\ell^{\infty}}+|r|_d),\\
 \sup_{w\in \mathcal{M}_{\delta}} | \partial^2_{\nu} f_3(\nu; w) |_d &\le \K.
\end{align*}
\end{itemize}

\begin{theorem}\label{thm:C2case}
Let $F\colon \mathcal{M}_{\delta}\to \mathcal{M}$ satisfies $\mathbf{(H0)_{C^2}}$-$\mathbf{(H3)_{C^2}}$. Then there exist $\delta_2>0$ and $\mu_2>0$ such that for all $\delta\in (0, \delta_2)$ and $\mu\in (0, \mu_2)$ the function $\gamma_{\nu}^s(x, \theta)$ given by Theorem \ref{thm:lipcase}, whose graph is the stable manifold of $\T_0$, is $C^2_{\Gamma}(\dom)$.  {If $F$ also satisfies $\mathbf{(H4)_{C^2}}$}, then $\gamma_{\nu}^s$ is $C^2$ with respect to $\nu\in (0,\mu)$.

Moreover for all $j\in \Z^m$
\[
\gamma_{\nu}^s\colon B_{\delta}(\Sigma_{j, \Gamma})\times \T^d\to \Sigma_{j, \Gamma}\times \R^d.
\]
and
\[
\| \gamma_\nu^s \|_{C^2_{\Gamma}(B_{\delta}(\Sigma_{j, \Gamma})\times \T^d)}\le C
\]
for some $C=C(j)>0$.

\end{theorem}

\medskip

\noindent\textbf{Proof of Theorem \ref{thm:C2case}}
Recall \eqref{def:SigmacM}, \eqref{def:DSigma}.
Let us define the set
\begin{equation}\label{def:Sigma2cMv}
\begin{aligned}
\Xi_{c, M, \mathtt{v}}^2:=&\{ \gamma\in  C^1_{\Gamma}( \dom, \ell^{\infty}\times\R^d) : \gamma(0, \theta)=0,\,\,\lip_x \partial_x \gamma\le \kappa_0,\,\,\lip_x \partial_{\theta} \gamma\le \kappa_1,\\
& \| D \gamma \|_0\le c,\,\,\| \partial_{\theta} \gamma \|_1\le M,\,\,\lip_{\theta} (\partial_x \gamma)\le \rho_0,\,\,\lip_{\theta} \,(\partial_{\theta} \gamma)(x)\le \widetilde{M}\,\| x \|_{\ell^{\infty}}\,\,\,\forall x\in B_{\delta}(\ell^{\infty})   \}
\end{aligned}
\end{equation}
where
\[
\mathtt{v}:=(\tM, \kappa_0, \kappa_1, \rho_0).
\]
We assume that $c$ and $M$ satisfy the assumptions of Lemma \ref{lemball} and that $\mathtt{v}$ satisfies the conditions in Lemma \ref{lem:Hwelldef}.\\
We observe that $\Xi^2_{c, M, \mathtt{v}}\subset \Xi_{c, M}$. We also introduce
\begin{align*}
D \Xi^2_{\mathtt{v}, \widetilde{\mathtt{v}}}:=& \{ \Upsilon:=(\Upsilon_0, \Upsilon_1, \Upsilon_2)\in E_{xx}\times E_{x \theta}\times E_{\theta \theta}, \,\,\,\,\lip_x \Upsilon_j\le \wkappa_j ,\,\,j=0, 1, 2,\,\,\,\,\lip_{\theta} \Upsilon_j\le \wrho_j,\,\,j=0,1,\\
&\,\,\lip_{\theta} \Upsilon_2(x)\le \widehat{M} \| x \|_{\ell^{\infty}}\,,\,\, \| \Upsilon_0\|_0\le \kappa_0, \, \| \Upsilon_1 \|_0\le  \rho_0,\,\,\,\,\| \Upsilon_2 \|_1\le \widetilde{M}  \}
\end{align*}
with
\[
 \widetilde{\mathtt{v}}:=(\widetilde{\kappa_0}, \widetilde{\kappa_1}, \widetilde{\kappa_2}, \widetilde{\rho_0}, \widetilde{\rho_1}).
\]
Without loss of generality we have further assumed that $\rho_0\geq \kappa_1$.
We consider
\begin{equation}\label{def:normDS}
\| \Psi \|_{D\Xi^2_{\mathtt{v}, \widetilde{\mathtt{v}}}}:=\widetilde{\alpha_0} \| \Psi_0\|_0+\widetilde{\alpha_1} \| \Psi_1 \|_0+\eta \| \Psi_2\|_1
\end{equation}
where $\widetilde{\alpha_0}, \widetilde{\alpha_1}$ and $\eta$ are positive constants to be chosen later.\\

We look for $\mathcal{H}$ such that $D^2 [ G(\gamma)]=\mathcal{H}(\gamma, D \gamma, D^2 \gamma)$ (recall the definition of $G$ in \eqref{G}). For that we differentiate formally $G(\gamma)$ and we substitute $\partial^{2-j}_x \partial_{\theta}^j \gamma$ with $\Upsilon_j$.
We define 
$
\cH_j:=\partial_{\theta}^j \partial_{x}^{2-j} G(\gamma).
$
We have (recall \eqref{def:H0})
\begin{equation}\label{def2:H0}
\begin{aligned}
\cH_0:=&\,C(\theta)^{-1} \Big\{   (\partial_{x}^2 g)(z, \gamma(z))+2(\partial_{v x}^2 g)(z, \gamma(z))  \partial_x \gamma+ (\partial_{v}^2 g)(z, \gamma(z)) (\partial_x \gamma)^2+(\partial_{v} g)(z, \gamma(z))  \Upsilon_0 \\
 &\,+\Upsilon_0 (h(z)) (\partial_x h_1)^2+2\Upsilon_1 (h(z)) \partial_x h_1\,\partial_x h_2+\Upsilon_2 (h(z)) (\partial_x h_2)^2+(\partial_x \gamma)(h(z)) \partial_x^2 h_1\\ 
&\,+(\partial_{\theta} \gamma)(h(z)) \partial_x^2 h_2  \Big\},\\
\cH_1:=&\,\partial_{\theta} (C(\theta)^{-1}) \Big\{  (\partial_x g) (z, \gamma(z))+(\partial_v g) (z, \gamma(z)) \partial_x \gamma+(\partial_x \gamma) (h(z)) \partial_x h_1+(\partial_{\theta} \gamma) (h(z)) \partial_x h_2  \Big\}\\ 
&\,+C(\theta)^{-1} \Big\{  (\partial_{x \theta}^2 g) (z, \gamma(z))+(\partial_{v \theta}^2 g)(z, \gamma(z)) \partial_x \gamma+(\partial_{xv}^2 g)(z, \gamma(z)) \partial_{\theta} \gamma+(\partial_v^2 g)(z, \gamma(z)) \partial_x \gamma\,\partial_{\theta} \gamma\\
&\,+(\partial_v g) (z, \gamma(z)) \Upsilon_1+\Upsilon_0 (h(z)) \partial_x h_1 \partial_{\theta} h_1+\Upsilon_1 (h(z)) \partial_x h_2 \partial_{\theta} h_1+\Upsilon_1 (h(z)) \partial_x h_1 \partial_{\theta} h_2\\
&\,+\Upsilon_2 (h(z)) \partial_x h_2 \partial_{\theta} h_2+(\partial_x \gamma)(h(z)) \partial_{x \theta}^2 h_1+(\partial_{\theta} \gamma) (h(z)) \partial_{x\theta}^2 h_2 \Big\},\\
\cH_2:=&\,\partial^2_{\theta} (C(\theta)^{-1}) \Big\{  g(z, \gamma(z))+\gamma(h(z))  \Big\}\\
&\,+2\partial_{\theta} \big(C(\theta)^{-1}\big)  \{  (\partial_{\theta} g)(z, \gamma(z))+(\partial_{v} g)(z, \gamma(z)) \partial_{\theta} \gamma+(\partial_x \gamma)(h(z))\,\partial_{\theta} h_1+(\partial_{\theta} \gamma) (h(z)) \partial_{\theta} h_2 \}\\
&\,+C(\theta)^{-1} \Big\{  (\partial_{\theta}^2 g)(z, \gamma(z))+2(\partial_{v \theta}^2 g)(z, \gamma) \partial_{\theta} \gamma+(\partial_{ v}^2 g)(z, \gamma(z)) (\partial_{\theta} \gamma)^2+(\partial_v g)(z, \gamma) \Upsilon_2\\
&\,+\Upsilon_0(h(z)) (\partial_{\theta} h_1)^2+2 \Upsilon_1 (h(z)) \partial_{\theta} h_1 \partial_{\theta} h_2+{\partial_x \gamma(h(z)) \partial_{\theta}^2 h_1}+\Upsilon_2 (h(z)) (\partial_{\theta} h_2)^2\\
&\,+(\partial_{\theta} \gamma) (h(z)) (\partial_{\theta}^2 h_2) \Big\}.
\end{aligned}
\end{equation}

We have to prove that: for $\mathtt{v}$ in Lemma \ref{lem:Hwelldef} and an opportune choice of $\widetilde{\mathtt{v}}$ 
\begin{itemize}
\item[(i)] $G\colon \Xi_{c, M, \mathtt{v}}^2\to \Xi_{c, M, \mathtt{v}}^2$ is well defined and has an attracting fixed point $\gamma_{\infty}$.
\item[(ii)] $H\colon \Xi^2_{c, M, \mathtt{v}}\times D \Xi_{\mathtt{v}, \widetilde{\mathtt{v}}}^2\to D \Xi_{\mathtt{v}, \widetilde{\mathtt{v}}}^2$ is well defined (see Lemma \ref{lem:ball2}).
\item[(iii)] $H(\cdot, \Psi)\colon \Xi^2_{c, M, \mathtt{v}}\to D \Xi_{\mathtt{v}, \widetilde{\mathtt{v}}}^2$ is continuous (see Lemma \ref{lem:cont2}).
\item[(iv)] $H(\gamma, \cdot)\colon D \Xi_{\mathtt{v}, \widetilde{\mathtt{v}}}^2\to D \Xi_{\mathtt{v}, \widetilde{\mathtt{v}}}^2$ is a contraction (see Lemma \ref{lem:contraction2}).
\end{itemize}

By Lemma \ref{lemball} if $\gamma\in \Xi_{c, M, \mathtt{v}}^2\subset \Xi_{c, M}$ then $G(\gamma)\in \Xi_{c, M}$. We need to prove that $G(\gamma)\in \Xi_{c, M, \mathtt{v}}^2$. By the definition \eqref{def:Sigma2cMv} if $\gamma\in \Xi_{c, M, \mathtt{v}}^2$ then $D\gamma\in D \Xi_{c, M, \mathtt{v}}$ (see \eqref{def:DSigma}). In the previous section we also proved that $H\colon \Xi_{c, M}\times D \Xi_{c, M, \mathtt{v}}\to D \Xi_{c, M, \mathtt{v}}$ is well defined, where $H$ is defined by $DG(\gamma)=H(\gamma, D(\gamma))$. Then $DG(\gamma)\in D \Xi_{c, M, \mathtt{v}}$ and so $G(\gamma)\in \Xi_{c, M, \mathtt{v}}$. 
By applying the fibre contraction theorem \ref{thm:fibrecontraction} to $\Gamma=(G, H)$ as we did in the previous section, we find that if $\gamma\in \Xi^2_{c, M, \mathtt{v}}$ the sequence of iterates $G^n(\gamma)$ converges to a $C^1$ function in $\Xi^2_{c, M, \mathtt{v}}$. Thus $G$ has an attracting fixed point and the item (i) has been proved.\\

From now on in this section we denote by $\tC$ any constant that depends on $\beta, c, M, \mathtt{v}, \K$ (see \eqref{def:DSigma}, \eqref{def:KomegaKf2}).

We state first some lemmas. The proof of the first one is straightforward using the assumed hypotheses and the particualr form of the functions $h_1$ and $h_2$ in \eqref{h}.
\begin{lem}\label{rmk:hderivatives}
For $\gamma\in \Xi^2_{c, M, \mathtt{v}}$ and $(x, \theta)\in B_{\delta}(\ell^{\infty})\times \T^d$, the functions $h_1, h_2$ introduced in \eqref{h} have the following estimates
\[
\begin{aligned}
&\| \partial_x h_1(x, \theta) \|_{\mathcal{L}_{\Gamma}(\ell^{\infty})}\le \lambda^{-1}+\cO(L), \\
&\| \partial_x h_2 (x, \theta) \|_{\mathcal{L}_{\Gamma}(\ell^{\infty}; \R^d)}\le 2 \K  (1+c),  \\ 
& \| \partial_{\theta} h_1(x, \theta) \|_{\mathcal{L}_{\Gamma}(\R^d; \ell^{\infty})}\le (\K(2+c)+2 LM) \| x \|_{\ell^{\infty}},\\    \label{h2theta}
&  \| \partial_{\theta} h_2(x, \theta) \|_{\mathcal{L}_{\Gamma}(\R^d)}  \le 1+K_{\theta}+\cO (\delta),\\ 
& \| \partial_x^2 h_1(x, \theta) \|_{\mathcal{L}^2_{\Gamma}(\ell^{\infty})}  \le \tC+\tC \| \Upsilon_0\|_0,\\
& \| \partial_{x \theta}^2 h_1(x, \theta) \|_{\mathcal{L}^2_{\Gamma}(\ell^{\infty}, \R^d; \ell^{\infty})}\le \tC+\cO(L+\delta)+L \|\Upsilon_1 \|_0,\\
&\| \partial_{\theta}^2 h_1(x, \theta) \|_{\mathcal{L}^2_{\Gamma}(\R^d; \ell^{\infty})}\le \cO(L+\delta)+\tC \| \Upsilon_2\|_0,\\ 
& \| \partial_x^2 h_2(x, \theta) \|_{\mathcal{L}^2_{\Gamma}(\ell^{\infty}; \R^d)}  \le \tC+2 K \| \Upsilon_0\|_0,\\
& \| \partial_{x \theta}^2 h_2(x, \theta) \|_{\mathcal{L}^2_{\Gamma}(\ell^{\infty}, \R^d; \R^d)}   \le \cO(L+\delta)+\tC +2\K \| \Upsilon_1\|_0,\\
&\| \partial_{\theta}^2 h_2(x, \theta) \|_{\mathcal{L}^2_{\Gamma}(\R^d)}   \le \cO(L+\delta)+\K +2 \K \| \Upsilon_2\|_0.
\end{aligned}
\]
\end{lem}

\begin{lem}\label{lem:ball2}
There exist functions $g_0, g_1, f_0, f_1$ and a constant $\mathtt{C}>0$ such that the following holds:
If 
\begin{align}
&\wkappa_0\geq \frac{ g_0(\wkappa_1, \wkappa_2,  \wrho_0, \wrho_1 )}{1-\lambda^{-3} \beta} \label{cond:kzero}  \\  \label{cond:kUno}
&\wkappa_1\geq \frac{ g_1( \wrho_1,  \wkappa_2)}{1-\beta\lambda^{-1} (1+K_{\theta})^2 }, \qquad \wkappa_2\geq \frac{\mathtt{C}}{1-\beta \lambda^{-1} (1+K_{\theta})^2}\\
&\wrho_0\geq \frac{f_0( \wrho_1, \wkappa_1,  \wkappa_2)}{1-\beta (1+K_{\theta}) \lambda^{-2}}, \qquad \wrho_1\geq \frac{\tC}{1-\beta\lambda^{-1} (1+K_{\theta})^2}  \label{cond:rozero}  \\
&\widehat{M}\geq \frac{f_1(\wkappa_2, \wrho_1)}{1-\beta\lambda^{-1}(1+K_{\theta})^3}
\end{align}
 then, for any $\gamma\in \Xi^2_{c, M, \mathtt{v}}$ and $\Psi\in D\Xi^2_{\mathtt{v}, \widetilde{\mathtt{v}}}$,  $\cH(\gamma, \Upsilon)\in D \Xi^2_{\mathtt{v}, \widetilde{\mathtt{v}}}$.
\end{lem}
The proof of this lemma is postponed to the Appendix \ref{sec:appendix}.
We observe that in order to fulfill the above conditions is sufficient to fix the parameters in the following order: $\wkappa_2, \wrho_1, \widehat{M}, \wkappa_1, \wrho_0, \wkappa_0$.

\begin{lem}\label{lem:cont2}
The function $\cH(\cdot, \Upsilon)\colon \Xi^2_{c, M, \mathtt{v}}\to D \Xi_{\mathtt{v}, \widetilde{\mathtt{v}}}^2$ is continuous.
\end{lem}

\begin{proof}
We have to prove that, given $\gamma, \gamma'\in \Xi^2_{c, M, \mathtt{v}}$, we can made small the difference
\begin{align*}
\| \cH(\gamma, \Upsilon)-\cH(\gamma', \Upsilon) \|_{D \Xi_{\mathtt{v}, \widetilde{\mathtt{v}}}^2}=&\widetilde{\alpha}_0 \| \cH_0(\gamma, \Upsilon)-\cH_0(\gamma', \Upsilon)\|_0+\widetilde{\alpha}_1 \| \cH_1(\gamma, \Upsilon)-\cH_1(\gamma', \Upsilon)\|_0\\
&+\eta \| \cH_2(\gamma, \Psi)-\cH_2(\gamma', \Psi)\|_1
\end{align*}
provided that 
\[
\| \gamma-\gamma'\|_{*}:=\max \{ \| D \gamma-D\gamma'\|_0, \| \partial_{\theta} \gamma-\partial_{\theta} \gamma' \|_1 \}
\]
is small.
The proof follows the proof of Lemma $5.5$ in \cite{FontichM98}. The only difference is in the following estimates
\begin{equation}\label{partialthetafj}
\begin{aligned}
\| \partial_{\theta}^l f_j(z, \gamma(z))-\partial_{\theta}^l f_j(z, \gamma'(z))\| &\le \chi_1(\gamma, \gamma') \| x \|_{\ell^{\infty}} \qquad l=1, 2, \qquad j=1,2, 4,  \\
\| \partial_{\theta}^l h_1( \gamma)(z)-\partial_{\theta}^l h_1( \gamma')(z)\| &\le \chi_2(\gamma, \gamma') \| x \|_{\ell^{\infty}} \qquad l=1, 2,  \\
\| \partial_{\theta} \gamma\big(h( \gamma) \big)(z)-\partial_{\theta} \gamma'\big(h( \gamma') \big)(z)  \|&\le \chi_3(\gamma, \gamma') \| x \|_{\ell^{\infty}} 
\end{aligned}
\end{equation}
where $\chi_i(\gamma, \gamma')$ are functions that go to zero as $\| \gamma-\gamma'\|_{*}$ tends to zero.
To prove the first bound in \eqref{partialthetafj} it is sufficient to use bounds \eqref{bound:sn} and the fact that, by definition, 
\[
\|  \gamma(z) - \gamma'(z) \|_{\ell^{\infty}\times \R^d}\le \| D \gamma-D \gamma'\|_0 \| x \|_{\ell^{\infty}} \le  \| \gamma-\gamma' \|_* \| x \|_{\ell^{\infty}}.
\]

We show how to prove the second bound for $l=2$. The other bounds are similar. By $\bf{(H2)_{C^2}}$ we have
\begin{align*}
\| \partial_{\theta}^2 h_1( \gamma)(z)-\partial_{\theta}^2 h_1&( \gamma')(z)\|_{\mathcal{L}_\Gamma^2(\R^d; \ell^{\infty})} \le \lip_v (\partial_{\theta}^2 f_1(\id, \gamma) ) \|  \gamma(z) - \gamma'(z) \|_{\ell^{\infty}\times \R^d}\\
&+\lip_v (\partial_{\theta v}^2 f_1(\id, \gamma) ) \| \partial_{\theta} \gamma \|_0 \|  \gamma(z) - \gamma'(z) \|_{\ell^{\infty}\times \R^d}\\
&+{\| \partial_{\theta v}^2 f_1(\id, \gamma) \|_{\mathcal{L}_{\Gamma}^2(\ell^{\infty}\times \R^d, \R^d; \ell^{\infty})}} \| \partial_{\theta} \gamma(z)-\partial_{\theta} \gamma'(z)\|_{\ell^{\infty}\times \R^d}\\
&+\lip_v (\partial_{\theta v}^2 f_1(\id, \gamma)) \| \partial_{\theta} \gamma \|_0^2 \|  \gamma(z) - \gamma'(z) \|_{\ell^{\infty}\times \R^d} \\
&+ \| \partial_v^2 f_1(\id, \gamma) \|_{\mathcal{L}_{\Gamma}^2(\ell^{\infty}\times \R^d, \ell^{\infty})} (\| \partial_{\theta} \gamma\|_0+\|  \partial_{\theta} \gamma' \|_0) \| \partial_{\theta} \gamma(z)-\partial_{\theta} \gamma'(z) \|_0\\
&+\| D^2 f_1(\id, \gamma)\|_{\mathcal{L}_{\Gamma}^2(T \mathcal{M}; \ell^{\infty})} \|\Upsilon_2\|_0 \| \gamma(z)-\gamma'(z)\|_{\ell^{\infty}\times \R^d}\\
\le& \mathtt{C}  \| x \|_{\ell^{\infty}}   \| \gamma-\gamma'\|_*.
\end{align*}
To prove the third bound in \eqref{partialthetafj} we use the triangle inequality. We consider just the most problematic term, that is 
\begin{align*}
\|   \partial_{\theta} \gamma'(h(\gamma)(z))-\partial_{\theta} \gamma'(h(\gamma')(z))    \|_{\mathcal{L}_{\Gamma}(\R^d; \ell^{\infty}\times \R^d)} &\le \lip_x (\partial_{\theta} \gamma'(h(\gamma)) ) \| h_1(\gamma)-h_1(\gamma')  \|_0+\widehat{M} \| h_1 \|_0 \| h_2(\gamma)-h_2 (\gamma') \|_0\\
&\stackrel{\eqref{hgamma-gamma'}, \eqref{h1Lx}}{\le} \cO(L+\delta)  \| \gamma-\gamma'\|_0
\end{align*}
for all $z\in \dom$.
\end{proof}

\begin{lem}\label{lem:contraction2}
For $\delta$ and $\mu$ small enough $\Upsilon\mapsto \cH(\gamma, \Upsilon)$ is a contraction on $D \Xi_{\mathtt{v}, \widetilde{\mathtt{v}}}^2$, uniform respect to $\gamma\in \Xi^{2}_{c, M, \mathtt{v}}$.
\end{lem}
\begin{proof}
By Lemmas \ref{rem:boundspartialzgammaz}, \ref{rmk:hderivatives} the functions $\cH_0$, $\cH_1$, $\cH_2$ introduced in \eqref{def2:H0} satisfy
\begin{align*}
\| \cH_0(\gamma, \Upsilon)-\cH_0(\gamma, \Upsilon')\|_0 \le &\big(\beta\lambda^{-2}+ \mathcal{O}(L) \big) \| \Upsilon_0-\Upsilon_0'\|_0+\tC \| \Upsilon_1-\Upsilon_1'\|_0+\mathcal{O}(\delta) \| \Upsilon_2-\Upsilon_2'\|_1,\\
\| \cH_1(\gamma, \Upsilon)-\cH_1(\gamma, \Upsilon') \|_0 \le&\mathcal{O}(L+\delta) \| \Upsilon_0-\Upsilon_0'\|_0+ (\beta \lambda^{-1} (1+K_{\theta})+\mathcal{O}(L+\delta)) \| \Upsilon_1-\Upsilon_1'\|_0\\
&+\cO(\delta) \| \Upsilon_2-\Upsilon_2'  \|_1,\\
\| \cH_2(\gamma, \Upsilon)-\cH_2(\gamma, \Upsilon') \|_1 \le& \Big(\beta \lambda^{-1} (1+K_{\theta})^2+\cO(L+\delta)  \Big) \| \Upsilon_2-\Upsilon_2'\|_1\\
&+\mathcal{O}(\delta) \| \Upsilon_0-\Upsilon_0'\|_0+\tC \| \Upsilon_1-\Upsilon_1'\|_0
\end{align*}
and, recalling \eqref{def:normDS},
\begin{align*}
\| \cH(\gamma, \Upsilon)-\cH(\gamma, \Upsilon') \|_{D\Sigma} \le& \widetilde{\alpha}_0\, (\beta \lambda^{-2}+\cO(L+\delta)) \| \Upsilon_0-\Upsilon'_0\|_0\\
&+\widetilde{\alpha}_1 \Big(\beta \lambda^{-1} (1+K_{\theta})+\tC_1 \frac{\widetilde{\alpha}_0}{\widetilde{\alpha}_1} +\tC_2 \frac{\eta}{\widetilde{\alpha}_1}  \Big) \| \Upsilon_1-\Upsilon'_1\|_0\\
&+\eta \Big(\beta \lambda^{-1} (1+K_{\theta})^2 \Big) \| \Upsilon_2-\Upsilon'_2\|_1
\end{align*}
for some constants $\tC_1, \tC_2>0$.
By \eqref{condition:lambdastrong} and choosing $\widetilde{\alpha}_0, \widetilde{\alpha}_1$ and $\eta$ such that
\begin{align*}
\tC_1 \frac{\widetilde{\alpha}_0}{\widetilde{\alpha}_1} +\tC_2 \frac{\eta}{\widetilde{\alpha}_1}&<1-\beta \lambda^{-1} (1+K_{\theta})
\end{align*}
we get the thesis.
\end{proof}

\subsection{Invariant manifolds for flows}\label{sec:invmanflows}

In Section \ref{sec:invmanmaps} we have proved the existence of $C^2_{\Gamma}$ invariant manifolds of invariant tori of maps under certain hypotheses (see Theorems \ref{thm:lipcase}, \ref{thm:C1case} and \ref{thm:C2case}). We devote this section to state an analogous theorem for flows. 

Let us consider a $C^2_\Gamma$ vector field $\XX_\nu$ defined on $\mathcal{M}_\de$. We assume that it is of the form  $\XX_\nu=\XX_0+\cF_\nu$ with 
\[
 \XX_0(w)=\begin{pmatrix} \cA_-(\theta) x\\\cA_+(\theta) y\\ \tilde\omega(x,y,r)\\ \cB(\theta)r\end{pmatrix}
\]
and 
\[\cF_{\nu}( w)=\big(\cF_1(\nu; w), \cF_2(\nu; w), \cF_3(\nu; w), \cF_4(\nu; w) \big)
\]
where (recall the notation \eqref{notation:linfS})
\[
\cA_{\pm}(\theta)\in \mathcal{L}_{\Gamma}(\ell^{\infty}_{S^c}), \quad \cB(\theta)\in\mathcal{L}_{\Gamma}(\R_S^d)
\]
and $\cF_{1}, \cF_2 (\nu; \cdot)\colon \mathcal{M}_{\delta}\to \ell^{\infty}_{S^c}$, $\cF_3(\nu; \cdot)\colon\mathcal{M}_{\delta}\to \T_S^d$, $\cF_4(\nu; \cdot) \colon \mathcal{M}_{\delta}\to \R_S^d$.
We also assume that
\[
\T_0:=\{ x=0, y=0, r=0 \}
\]
is an invariant torus for the flow associated to the vector field map $F_{\nu}$ for all $\nu\in(0,\mu)$ and we define
\begin{equation}\label{def:omega0}
 \tilde\omega_0=\tilde\omega(0,0,0).
\end{equation}
We provide a theorem of existence of local invariant manifolds for $\T_0$

Let us first start by stating the needed hypotheses. As in Section \ref{sec:invmanmaps}, we consider a non-negative continuous function $L(\delta, \mu)$ such that $L(0, 0)=0$.
We assume that there exist constants $\tilde\lambda, \tilde\beta>0$,  $\tilde\K$, $\tilde K_{\theta}>0$ such that

\begin{itemize}

\item[{$\mathbf{(H0)_f}$}]
We have
 \begin{equation}\label{bound:betalambda:flow}
\tilde\lambda- \tilde\beta>0.
\end{equation}
\item[{$\mathbf{(H1)_f}$}] 
The functions $\cA_{\pm}\in C^3(\T^d; \mathcal{L}_{\Gamma}(\ell^{\infty})), \cB\in C^3(\T^d; \mathcal{L}_{\Gamma}(\R^d)), \tilde\omega\in C^2(\mathtt{B}_{\delta}; \T^d)$ and $\mathcal{F}=\mathcal{F}_{\nu}\in C_{\Gamma}^2(\mathcal{M}_{\delta})$. Moreover, for all $\theta\in \T^d$, 
\[
\begin{split}
\| e^{\int_0^t \cA_-(\theta+\tilde\omega_0 s)ds}\|_{\mathcal{L}_{\Gamma}(\ell^{\infty})}, \| e^{\int_0^{-t} \cA_+(\theta+\tilde\omega_0 s)ds} \|_{\mathcal{L}_{\Gamma}(\ell^{\infty})}&\le e^{-\tilde{\lambda} t}, \qquad \text{ for all }\quad t\geq 0\quad \text{and}\quad \theta\in\T^d\\
\|e^{\int_0^t \mathcal{B}(\theta+\tilde\omega_0 s)ds} \|_{\mathcal{L}_{\Gamma}(\R^d)}&\le e^{\tilde{\beta}|t|}, \qquad \text{ for all }\quad t\in \mathbb{R}\quad \text{and}\quad \theta\in\T^d
\end{split}
\]
and
\begin{equation*}
\sup_{\substack{j=1, 2,3}} \| \partial_{\theta}^j \cA_{\pm}(\theta) \|_{\mathcal{L}_{\Gamma}( \ell^{\infty})},  \sup_{\substack{j=1, 2,3}} \| \partial_{\theta}^j \cB(\theta) \|_{\mathcal{L}_{\Gamma}(\R^d)},\sup_{(x, y, r)\in \mathtt{B}_{\delta}, j=1,2}\|  D \tilde{\omega}(x, y, r) \|_{\mathcal{L}_{\Gamma}(\ell^{\infty}\times\ell^{\infty}\times \R^d, \R^d)} \le \tilde{\K}.
\end{equation*}
Note that in the case that the matrices $ \cA_\pm$ and $\cB$ are constant (as happens in \eqref{system}), the exponential matrices above are just $e^{\cA_\pm t}$ and $e^{\cB t}$.
\item[$\mathbf{(H2)_f}$] For $j=1, 2, 4$ we have that $\cF_j(\nu; 0, 0, \theta, 0)=0$ and for all $w\in \mathcal{M}_{\delta}$
\begin{align*}
\| D \cF_k (\nu; w) \|_{\mathcal{L}_{\Gamma}(T\mathcal{M}; \ell^{\infty})} &\le L(\delta, \mu), \qquad k=1, 2, \\
\| D \cF_4 (\nu; w) \|_{\mathcal{L}_{\Gamma}(T\mathcal{M}; \R^d)} &\le L(\delta, \mu)\\
\| D^2 \cF_k (\nu; w) \|_{\mathcal{L}^2_{\Gamma}(T\mathcal{M}; \ell^{\infty})} &\le \tilde{\K}, \qquad k=1, 2\\
\| D^2 \cF_4 (\nu; w) \|_{\mathcal{L}^2_{\Gamma}(T\mathcal{M}; \R^d)} &\le \tilde{\K}, \qquad k=1, 2\end{align*}
where $T\mathcal{M}$ is the tangent space of $\mathcal{M}$, which is isomorphic to $\ell^{\infty}\times \ell^{\infty}\times \R^d\times \R^d$.
Moreover the derivatives with respect to $\theta$ have the following bounds
\begin{equation*}
\begin{aligned}
\| \partial_{\theta} \cF_k(\nu; x, y, \theta, r) \|_{\mathcal{L}_{\Gamma}(\R^d; \ell^{\infty})} &\le \K (\|  x \|_{\ell^{\infty}}+\| y \|_{\ell^{\infty}}+|r|_d)\,, \quad k=1, 2, \\
\| \partial_{\theta} \cF_4(\nu; x, y, \theta, r) \|_{\mathcal{L}_{\Gamma}(\R^d;\R^d)} &\le \K (\|  x \|_{\ell^{\infty}}+\| y \|_{\ell^{\infty}}+|r|_d)\\
\| \partial_{\theta}^2 \cF_k(\nu; x, y, \theta, r) \|_{\mathcal{L}^2_{\Gamma}(\R^d;\ell^{\infty})} &\le \K (\|  x \|_{\ell^{\infty}}+\| y \|_{\ell^{\infty}}+|r|_d)\,, \quad k=1, 2, \\
\| \partial_{\theta}^2 \cF_4(\nu; x, y, \theta, r) \|_{\mathcal{L}_{\Gamma}^2(\R^d;\R^d)} &\le \K (\|  x \|_{\ell^{\infty}}+\| y \|_{\ell^{\infty}}+|r|_d)\,.
\end{aligned}
\end{equation*}
and
\[
\begin{split}
\| \partial_{\theta s} \cF_k(\nu; x, y, \theta, r) \|_{\mathcal{L}^2_{\Gamma}(\R^d, Y;\ell^{\infty})} &\le L(\de,\mu), \quad k=1, 2, \\
\| \partial_{\theta s} \cF_4(\nu; x, y, \theta, r) \|_{\mathcal{L}^2_{\Gamma}(\R^d, Y;\R^d)} &\le L(\de,\mu).
\end{split}
\]
where $s=x,y,r$ and $Y=\ell^\infty$ (when $s=x,y$) and  $Y=\R^d$ (when $s=r$).
\item[$\mathbf{(H3)_f}$] The function $\cF_3$ satisfies the following
\begin{align*}
\sup_{w\in \mathcal{M}_{\delta}}\| \partial_{\theta} \cF_3(w) \|_{\mathcal{L}_{\Gamma}(\R^d)} &\le \tilde{K}_{\theta},\\
\sup_{w\in \mathcal{M}_{\delta}, j=1,2}\|  D^j \cF_3(\nu; w) \|_{\mathcal{L}^j_{\Gamma}(T\mathcal{M}, \R^d)} &\le \tilde{\K}.
\end{align*}

\item[$\mathbf{(H4)_f}$] The second order derivatives are Lipschitz on $\cM_\de$, namely
\begin{align*}
 \lip\,\, \partial^2_{s, s'} \cF_j &\le  \tilde{\K}, \qquad s, s'=x, y, \theta, r, \qquad  j=1, 2, 3, 4,\\
 \lip_{\theta} \partial_{\theta}^2 \cF_j(x,y,r) &\le \tilde{\K} (\| x \|_{\ell^{\infty}}+\| y\|_{\ell^{\infty}}+|r|_d  ),  \qquad j=1, 2, 4, \\
 \lip\,\, \partial^2_{s, s'} \omega &\le \tilde{\K}, \qquad s, s'=x, y, r. 
\end{align*}
\item[$\mathbf{(H5)_f}$] The function $\cF$ is $C^2$ with respect to $\nu$ and we have 
\begin{align*}
\|  \partial^2_{\nu} \cF_k(\nu; x, y, \theta, r) \|_{\ell^{\infty}} &\le \K (\|  x \|_{\ell^{\infty}}+\| y \|_{\ell^{\infty}}+|r|_d), \quad   k=1, 2,\\
 |  \partial^2_{\nu} \cF_4(\nu; x, y, \theta, r) |_{d} &\le \K (\|  x \|_{\ell^{\infty}}+\| y \|_{\ell^{\infty}}+|r|_d),\\
 \sup_{w\in \mathcal{M}_{\delta}} | \partial^2_{\nu} \cF_3(\nu; w) |_d &\le \K
\end{align*}
Moreover, the Lipschitz constant of the second derivatives of $\cF_k$ satisfy also $\mathbf{(H4)_f}$ treating $\nu$ as an extra component of the angle $\theta$.
\end{itemize}
%
%
\begin{theorem}\label{thm:invmanflows}
Let $\cX_\nu$ be a $C^2_\Gamma$ vector field defined on $\mathcal{M}_{\delta}$ of the form  $\XX_\nu=\XX_0+\cF_\nu$. Assume that it satisfies $\mathbf{(H0)_f}$-$\mathbf{(H4)_f}$. Then there exist $\delta_0>0$ and $\mu_0>0$ such that for all $\delta\in (0, \delta_0)$ and $\mu\in (0, \mu_0)$ the $\cX_\nu$-invariant torus $\T_0$ possesses a stable invariant manifold which can be represented as graph of a $C^2_\Gamma$ function $\gamma_{\nu}^s(x, \theta)\in C^2_\Gamma(B_{\delta}(\ell^{\infty})\times \T^d; \ell^{\infty}\times \R^d)$ that satisfies
\begin{itemize} 
\item$\gamma^s_{\nu}(0, \theta)=0$. Moreover, its $C^1_\Gamma$ norm is of order $\delta+L(\delta, \mu)$ and
\[
\sup_{(x, \theta,\nu)\in B_{\delta}(\ell^{\infty})\times \T^d\times (0,\mu)} \| \gamma^s_{\nu}(x, \theta)\|_{\ell^{\infty}\times \R^d}\le \mathcal{O}(\delta^2+\delta\,L(\delta, \mu)).
\] 
\item The iterates of the points $(x, \theta, \gamma_{\nu}^s(x, \theta))$ tend to the torus exponentially fast with asymptotic rate bounded by $e^{\tilde\lambda t}$ as $t\to-\infty$. 
\end{itemize}
If we also impose $\mathbf{(H5)_{f}}$, $\gamma_{\nu}^s$ is also $C^2$ with respect to $\nu$. 
Moreover for all $j\in \Z^m$
\[
\gamma_{\nu}^s\colon B_{\delta}(\Sigma_{j, \Gamma})\times \T^d\to \Sigma_{j, \Gamma}\times \R^d,
\]
 $\gamma_{\nu}^s\in C^2_{\Gamma}(B_{\delta}(\Sigma_{j, \Gamma})\times \T^d)$ and its $C^1_\Gamma$ norm is of order $\delta+L(\delta, \mu)$.

\end{theorem}

Note that this theorem implies easily Theorem \ref{thm:toriinvman}. Indeed, it is straightforward to verify Hypotheses $\mathbf{(H1)_f}$-$\mathbf{(H4)_f}$ for the vector field \eqref{system} (see also \eqref{def:fs}). Note also that one can also verify Hypothesis $\mathbf{(H5)_f}$ with respect to the parameter $\eps$. 

We devote the rest of the section to deduce Theorem \ref{thm:invmanflows} from Theorems \ref{thm:lipcase}, \ref{thm:C1case}, \ref{thm:C2case}.

\begin{proof}
Denote by $\Phi^t$ the flow defined by the vector field $\cX_\nu$. 
It is easy to check by a classical Picard iteration argument that for any fixed $T>0$ there exists $\de>0$ small enough so that 
\[
 \Phi^t:\mathcal{M}_\de\to \mathcal{M}_{C\de}
\]
for some $C>0$ independent of $\de$ and $t\in [-T,T]$.

We take $T\gg 1$ and we write $\Phi^T$ in a particular form so that Hypotheses $\mathbf{(H1)_f}$-$\mathbf{(H5)_f}$ can be verified. First note, that $\cX_\nu$ can be written as $\cX_\nu=\tilde\cX_0+\tilde \cF_\nu$ where
\[
\wt\XX_0(w)=\begin{pmatrix} \cA_-(\theta) x\\\cA_+(\theta) y\\ \tilde\omega_0\\ \cB(\theta)r\end{pmatrix}
\]
(see \eqref{def:omega0}) and $\tilde\cF_k(\nu; w)=\cF_k(\nu; w)$ for $k=1,2,4$ and 
\[
\tilde\cF_3(\nu; w)=\cF_3(\nu; w)+\tilde\omega(x,y,r)-\tilde\omega(0,0,0)
\]
One can easily check that $\tilde\cX_0$ and $\tilde \cF_\nu$ also satisfy Hypotheses $\mathbf{(H1)_f}$-$\mathbf{(H5)_f}$.

We use this rewriting of $\cX_\nu$ to write $\Phi^t$ in a particular form. First note that, denoting by $\Phi_0^t$ the flow of $\tilde\cX_0$, one has that 
\[
 \Phi^t_0(x,y,r,\theta)=\begin{pmatrix} e^{\int_0^t\cA_-(\theta+\tilde\omega_0 s)ds} x\\ e^{\int_0^t\cA_+(\theta+\tilde\omega_0 s)ds} y\\ \theta+\tilde\omega_0 t\\ e^{\int_0^t\cB(\theta+\tilde\omega_0 s)ds}r\end{pmatrix}.
\]
Then, following the notation in Section \ref{sec:invmanmaps} and 
applying Duhamel formula, one can write $\Phi^T$ as
\[
\Phi^T= F^T_0(w)+f^T_{\nu}( w)
\]
where $F^T_0(w)=\Phi_0^T(w)$ and 
\[
\begin{split}
 f^T_{1}( w)&= \int_0^T e^{\int_t^T\cA_-(\theta+\tilde\omega_0 s)ds}\tilde\cF_1(\nu; \Phi^t(w))dt\\
f^T_{2}( w)&= \int_0^T e^{\int_t^T\cA_+(\theta+\tilde\omega_0 s)ds}\tilde\cF_2(\nu; \Phi^t(w))dt\\
f^T_{3}( w)&=\int_0^T\tilde\cF_3(\nu;\Phi^t(w))dt\\
f^T_{4}( w)&= \int_0^T e^{\int_t^T\cB(\theta+\tilde\omega_0 s)ds}\tilde\cF_4(\nu; \Phi^t(w))dt.
 \end{split}
\]
Fixing $T\gg 1$ and $\de>0$ small enough, it is straightforward to check that $F_0^T$ and $f_\nu^T$ satisfy the Hypotheses $\mathbf{(H1)_f}$-$\mathbf{(H5)_f}$.

In particular, for $T>0$ large enough, \eqref{condition:lambdastrong} is satisfied. Indeed, on the one hand
\[
 \lambda^{-1} \leq e^{-\tilde\lambda T}, \qquad \beta \leq e^{\tilde\beta T}
\]
and on the other hand
\[
 K_\theta \lesssim T\tilde K_\theta, \qquad  K \lesssim T\tilde K.
\]
Therefore, there exists $T^*$ such that  for $T\geq T^*$ one has the inequality \eqref{condition:lambdastrong}. Then, Theorems \ref{thm:lipcase}, \ref{thm:C1case}, \ref{thm:C2case} imply the existence of the torus stable invariant manifold for the  map $\Phi^T$ for $T\geq T^*$. We denote this parameterization by $\gamma^T$
. Note that it is defined in $B_{{\delta}_T}(\ell^\infty)\times\mathbb{T}^d$ for some $\delta_T>0$ which may depend on $T$.

Then, it only remains to show that $\gamma^T$ is independent of $T$. This would imply that $\gamma=\gamma^T$ is invariant by the flow $\Phi^t$. Note that it is enough to show that there exists $0<\eta\ll1$ so that for any $T_1,T_2\in [T^*, T^*+\eta]$ one has $\gamma^{T_1}=\gamma^{T_2}$ since this implies that the vector field is tangent to the invariant manifold.

First note that using the uniqueness of the invariant manifold, for any $n\in \mathbb{N}$, $\gamma^T=\gamma^{nT}$ since $\gamma^T$ is both invariant under $\Phi^T$ and under its $n$-iterate $\Phi^{nT}$. Reasoning analogously, one has that $\gamma^T=\gamma^{qT}$ for any $q\in\mathbb{Q}$.

Note moreover, that it is easy to check that there exists $\delta_0$ such that for any $T\in [T^*, T^*+\eta]$, $\gamma^T$ is defined in $B_{{\delta}_0}(\ell^\infty)\times\mathbb{T}^d$.
Then, the family of parameterizations  $\gamma^T$ for  $T\in [T^*, T^*+\eta]$  are defined in a common domain. Moreover, by Hypothesis\footnote{Note that Hypothesis $\bf{(H5)_f}$ only admits dependence on the parameter $\nu$ on $f_\nu$ but not on $F_0$. This is not the case for the parameter $T$ and the map $\Phi^T=F_0^T+f_\nu^T$ since the two terms in the sum depend on $T$. However, note that one can just fix $T_0$ and  define $\tilde F_0=F_0^{T_0}$ and $\tilde f_\nu=f_\nu^T+F_0^{T}-F_0^{T_0}$ accordingly. The new maps satisfy the same hypothesis as before and also \textbf{H5} for the parameter $T$} $\mathbf{(H5)_f}$ they are $C^2$ with respect to $T$ and $\gamma^T$ coincides with $\gamma^{T^*}$ for all $T=qT^*$ with $q\in\mathbb{Q}$. Thus, we can conclude that $\gamma$ is independent of $T$. This completes the proof of Theorem \ref{thm:invmanflows}.
\end{proof}

\section{Transverse intersection of invariant manifolds and construction of the transition chain}\label{sec:MelnikovTheory}

We devote this section to prove Theorem \ref{thm:transitiochainXj}. 
As explained in Section \ref{sec:transvheterosheuristic}, recall that we deal with formal Hamiltonians. However, one can restrict to the invariant (under the flow of \eqref{eq:motions}) subspaces $\cV_i$ in \eqref{def:subspacesXj} where the energy is well defined. Recall that  $\cV_i$ is invariant thanks to 
%
%
Hypotheses \textbf{H1} and \textbf{H2}.

Within the invariant subspaces $\cV_i$ and at a fixed energy level $h$, there are the invariant cylinders introduced in \ref{def:cylinderLambdaj},
\[
 \Lambda_i=\left\{(q,p)\in \cV_i\cap H^{-1}(h): q_{\sigma_{i+2}}=p_{\sigma_{i+2}}=0 \right\}.
\]
Note that these cylinders are not normally hyperbolic since they possess the periodic orbits $P_{\sigma_i}$ and $P_{\sigma_{i+1}}$ introduced in \eqref{def:invtori}
whose hyperbolicity within $\Lambda_i$ is as strong as the normal one. However, for any $\delta>0$,
\begin{equation}\label{def:nhim}
  \Lambda_{i,\delta}=\left\{(q,p)\in \Lambda_i: E_i(q,p)\in (\delta, h-\delta)\right\}
\end{equation}
is a normally hyperbolic invariant manifold both for $\varepsilon=0$ and for $0<\varepsilon\ll 1$.

We construct the chain of transverse heteroclinics connecting different invariant tori of either dimension one or two given by Theorem \ref{thm:transitiochainXj}  in two steps. First we construct ``pieces'' of the transition chain ``within'' $\cV_i$. Then, we show that the heteroclinic connections are transverse in the infinite dimensional setting in the sense of Definition \ref{def:transversalityflows}.

\subsection{Transversality within the 6 dimensional invariant subspaces}\label{sec:chain6dimspace}
To analyze the existence of transverse heteroclinic orbits in the subspace  $\cV_i$ we distinguish two ``regimes''. First we analyze the invariant tori in the normally hyperbolic invariant cylinder $\Lambda_{i,\delta}$ in \eqref{def:nhim}. Later we analyze the invariant tori in $\delta$-neighborhoods of the periodic orbits $P_{\sigma_{i}}$ and $P_{\sigma_{i+1}}$ (see  \eqref{def:invtori}). 

We consider the intersection of the invariant manifolds of the invariant tori with the invariant subspace $\cV_i$ which we denote by
\[
  W_{\eps,\cV_i}^u(\T) =W_{\eps}^u(\T)\cap \cV_i,
\]
where $\T$ denotes any of the tori in \eqref{def:invtori}.

\begin{lem}\label{lemma:TransitionArnold}
Fix $\delta>0$ small. Assume that $H$ satisfies $\mathbf{H1}$,$\mathbf{H2}$,$\mathbf{H3}$. Then, there exists $\varepsilon_0>0$ such, that for any $\varepsilon\in (0,\varepsilon_0)$, there exist $K>0$ and a sequence of {two-dimensional} tori $\T_{i,k}\subset \Lambda_i$, $k=0\ldots K$ 
such that 
\[
 E_{\sigma_{i}}(\T_{i,0})\in (0,\delta)\qquad \text{and}\qquad  E_{\sigma_{i}}(\T_{i,K})\in (h,h-\delta)
\]
and 
\[
 W_{\eps,\cV_i}^u(\T_{i.k}) \pitchfork_{\cV_i}  W_{\eps,\cV_i}^s(\T_{i,k+1})
\]
where  $\pitchfork_{\cV_i}$ denotes transversality in the sense of Defintion \ref{def:transversalityflows} applied to the invariant subspace $\cV_i$.
\end{lem}
This lemma is a consequence of Melnikov Theory. Its prove goes back to \cite{Arnold64}. A more modern proof can be found in \cite{DelshamsLS06a} which relies on the so called scattering map. Note that in these papers, the dynamics of the unperturbed Hamiltonian on  the cylinder is given already in action angle coordinates. Even if this is not the case in the present setting, the proof follows the same lines as \cite{DelshamsLS06a} since the scattering map is defined independently of the choice of coordinates.

Hypothesis $\mathbf{H3.1}$ ensures that the invariant manifolds of the normally hyperbolic invariant cylinder $\Lambda_{j,\delta}$ are transverse. This allows to define scattering maps locally at these transverse intersections.  Hypothesis $\mathbf{H3.2}$ ensures that these scattering maps are such that the image of the level sets of the energy $E_j$ is transverse to the level sets. This implies that there are heteroclinic connections between ``close enough'' tori.

Lemma \ref{lemma:TransitionArnold} give a transition chain that ``connects'' invariant tori which are $\delta$-close to the periodic orbits $P_{{\sigma_i}}$ and $P_{\sigma_{i+1}}$. The next lemma extend the transition chain to reach these periodic orbits (see Figure \ref{fig:chain}).

\begin{lem}\label{lemma:TransitionPeriodic}
Fix $\delta>0$ small. Assume that $H$ satisfies $\mathbf{H1}$,$\mathbf{H2}$,$\mathbf{H3}$ and $\mathbf{H4}$. Then there exists $\varepsilon_0>0$ such that for any $\varepsilon\in (0,\varepsilon_0)$ there exist $K'>0$ and a sequence of tori $\T_{i,k}'\subset \Lambda_i$, $k=1\ldots K'$ where $\T_{i,K'}'=\T_{i,0}$ is the torus obtained in Lemma \ref{lemma:TransitionArnold}  such that 
\[
 W_{\eps,\cV_i}^u(P_{\sigma_j}) \pitchfork_{\cV_i}  W_{\eps,\cV_i}^s(\T_{i,0}')\qquad \text{and}\qquad  W_{\eps,\cV_i}^u(\T_{i,k}) \pitchfork_{\cV_i}  W_{\eps,\cV_i}^s(\T_{i,k+1})\qquad \text{for}\qquad k=0,\ldots K'.
\]
\end{lem}
Note that for Lemma \ref{lemma:TransitionPeriodic} one cannot apply the scattering map technology since in this setting ${\Lambda_i}$ is not a normally hyperbolic cylinder anymore. Instead, we use the classical Arnold approach \cite{Arnold64} to deal directly with the invariant manifolds of the periodic and the invariant tori and look for their intersections. Note that this approach could also have been used in Lemma \ref{lemma:TransitionArnold}. We have used instead the scattering map to emphasize that that lemma deals with the ``classical'' a priori unstable setting.

\begin{proof}[Proof of Lemma \ref{lemma:TransitionPeriodic}]
  We show how to connect $P_{\sigma_i}$ with the torus $\T_{i,0}'$ (to be chosen).   
  Note that the stable/unstable invariant manifolds of these objects are three dimensional within a five dimensional energy level. Therefore, to analyze the breakdown of the homoclinic channels and the possible connections between different objects we have to fix a two dimensional section transverse to the unperturbed homoclinic manifold to $P_{\sigma_i}$. 
 
 We consider a section transverse to the homoclinic manifolds of the periodic orbit $P_{\sigma_i}$ and the invariant tori ``close to it''. 
 Let us call $\Pi=\Pi(x_i, \tau_1, \tau_2)$ the $2$-dimensional affine subspace passing through a given point $\tilde{z}_0=\tilde{z}_0(0, x_i, \tau_1, \tau_2)$ of the homoclinic manifold of $P_{\sigma_i}$ and spanned by the vectors $\nabla E_{\sigma_{i+1}}(\tilde{z}_0)$, $\nabla E_{\sigma_{i+2}}(\tilde{z}_0)$, with
 \begin{equation}\label{def:homomanifold}
 \begin{aligned}
\tilde{z}_0(t, x_i, \tau_1, \tau_2) &=\Phi_0^t(\tilde{z}_0(0, x_i, \tau_1, \tau_2))\\
&=\big(q_h(t, x_j), p_h(t, x_j), q_0(\tau_1+t), p_0(\tau_1+t), q_0(\tau_2+t), p_0(\tau_2+t) \big),
\end{aligned}
\end{equation}
where $(q_h(t, x_i), p_h(t, x_i))$ is the periodic orbit contained in $\{ E_{\sigma_i}=h, E_{\sigma_{i+1}}=E_{\sigma_{i+2}}=0 \}$ such that $q_h(0, x_i)=x_i\in \T$ and $(q_0(\tau_k), p_0(\tau_k))$ is the homoclinic orbit of the pendulum $\sigma_{i+k}$ with $k=1, 2$.
The section $\Pi(x_i, \tau_1, \tau_2)$ is transverse to the homoclinic manifold of $P_{\sigma_i}$ and, for $\e>0$ small enough, also to the invariant manifolds $W_{\e, \mathcal{V}_i}^{s, u}(P_{\sigma_i})$.

Fix $T>0$. Since the invariant manifolds are regular with respect to parameters, for $0<\e\ll 1$ and any
\begin{equation}\label{def:DomainMelnikov}
(x_i, \tau_1, \tau_2)\in \T\times [-T, T]^2
\end{equation}
there are points $\tilde{z}^{s, u}=\tilde{z}^{s, u}(x_i, \tau_1, \tau_2)$ which belong to the intersection of $W_{\e, \mathcal{V}_i}^{s, u}(P_{\sigma_i})$ with $\Pi(x_i, \tau_1, \tau_2)$.
In particular these points satisfy
\[
\tilde{z}^{s, u}=\tilde{z}^{s, u}(x_i, \tau_1, \tau_2)=\tilde{z}_0(0, x_i, \tau_1, \tau_2)+\mathcal{O}_{C^1}(\e).\footnote{Since we are in a finite dimensional setting decay plays no role.}
\]
To measure the distance between these points in $\Pi$ we use the energies of the pendulums $\sigma_{i+1}, \sigma_{i+2}$
\[
d_{\sigma_i+k}(\tilde{z}^u, \tilde{z}^s)=E_{\sigma_{i+k}}(\tilde{z}^u)-E_{\sigma_{i+k}}(\tilde{z}^s), \qquad k=1, 2.
\]
Let us denote by
\[
\tilde{z}^{s, u}(t, x_i, \tau_1, \tau_2)=\Phi_{\e}^t(\tilde{z}^{s, u}(x_i, \tau_1, \tau_2)).
\]
Then
\begin{align*}
E_{\sigma_{i+k}}(\tilde{z}^s(x_i, \tau_1, \tau_2))&=E_{\sigma_{i+k}}(\tilde{z}^s(t, x_i, \tau_1, \tau_2))-\int_0^t \left( \frac{d}{d \mathtt{t}} E_{\sigma_{i+k}}(\tilde{z}^s(\mathtt{t}, x_i, \tau_1, \tau_2))\right) \,d\mathtt{t}\\
&=E_{\sigma_{i+k}}(\tilde{z}^s(t, x_i, \tau_1, \tau_2))-\e\int_0^t \{ E_{\sigma_{i+k}}, H_{1, i} \} \circ \Phi_{\e}^{\mathtt{t}}(\tilde{z}^s(x_i, \tau_1, \tau_2))  \,d\mathtt{t}
\end{align*}
Since the forward iterates of $\tilde{z}^s$ tend to the periodic orbit $P_{\sigma_{i}}$ and $E_{\sigma_{i+k}}(P_{\sigma_{i}})=0$, $k=1,2$, when we let $t$ tends to $+\infty$ we get
\[
E_{\sigma_{i+k}}(\tilde{z}^s(x_i, \tau_1, \tau_2))=\e\int_0^{+\infty} \{  H_{1, i}, E_{\sigma_{i+k}} \} \circ \Phi_{\e}^{\mathtt{t}}(\tilde{z}^s(x_i, \tau_1, \tau_2))  \,d\mathtt{t}
\]
and reasoning analogously 
\[
E_{\sigma_{i+k}}(\tilde{z}^u(x_i, \tau_1, \tau_2))=\e\int_{-\infty}^{0} \{  H_{1, i}, E_{\sigma_{i+k}} \} \circ \Phi_{\e}^{\mathtt{t}}(\tilde{z}^u(x_i, \tau_1, \tau_2))  \,d\mathtt{t}.
\]
Since the perturbation $H_1$ vanishes on the cylinder $\tilde{\Lambda}_i$, by hyperbolicity we have
\begin{equation}\label{bound:hyp}
\begin{aligned}
\|\Phi^{\mathtt{t}}_{\e}(\tilde{z}^s(x_i, \tau_1, \tau_2))-\Phi^{\mathtt{t}}_0(\tilde{z}_0(0, x_i, \tau_1, \tau_2))\|\le C \e e^{-\nu \mathtt{t}} \qquad \forall \mathtt{t}\geq 0
\end{aligned}
\end{equation}
for some $C>0$ and $\nu>0$ independent of $\e$. The same estimate holds for the derivatives in $x_i$, $\tau_1$ and $\tau_2$.
Reasoning in the same way for $\Phi^{\mathtt{t}}_{\e}(z^u(x_i, \tau_1, \tau_2))$ for $\mathtt{t}\le 0$ we conclude that
\[
E_{\sigma_{i+k}}(\tilde{z}^u)-E_{\sigma_{i+k}}(\tilde{z}^s)=\e \mathcal{M}_{\sigma_{i+k}}(h, x_i, \tau_1, \tau_2)+\mathcal{O}(\e^2)
\]
where $\mathcal{M}_{\sigma_{i+k}}(x_j,t_1,t_2)$ is the Melnikov function
\[
\mathcal{M}_{\sigma_{i+k}}(h, x_i, \tau_1, \tau_2)=\int_{-\infty}^{+\infty} \{  H_{1, i}, E_{\sigma_{i+k}} \} \circ \Phi_{0}^{\mathtt{t}}(\tilde{z}_0(0, x_i, \tau_1, \tau_2))  \,d\mathtt{t}
\]
and $H_{1,j}$ is the Hamiltonian introduced in \eqref{def:Hamj}.
Note that $\mathcal{M}_{\sigma_{i+1}}=\partial_{t_1}\widetilde{\mathcal{L}}_j$ and 
 $\mathcal{M}_{\sigma_{i+2}}=\partial_{t_2}\widetilde{\mathcal{L}}_j$ (see \eqref{def:MelnikovPotential:per}).

Hypothesis $\mathbf{H4}$ ensures the existence of  transverse critical points of $\widetilde{\mathcal{L}}_j$. Then, the Implicit Function Theorem gives the existence of non-degenerate zeros of the distances $d_{\sigma_{j+1}}$ and $d_{\sigma_{j+2}}$ which are  $\varepsilon$-close to these critical points. These zeros correspond to transverse homoclinic orbits to $P_{\sigma_j}$.

We show that, from the existence of transverse homoclinic orbits to $P_{\sigma_j}$, one can construct heteroclinic connections between $P_{\sigma_j}$ and $2$-dimensional tori $\T_{\sigma_i, \sigma_{i+1}}=\{ E_{\sigma_i}=h_i, E_{\sigma_{i+1}}=h-h_i, E_{\sigma_{i+2}}=0 \}$ with $h-h_i=\mathcal{O}(\e)$. A similar idea,  in a different context, was used in \cite{GiulianiGMS21}.


For $\e>0$ small enough, since $h-h_i=\mathcal{O}(\e)$, the section $\Pi(x_i, \tau_1, \tau_2)$ is transversal to the homoclinic manifold of $\T_{\sigma_i, \sigma_{i+1}}$ and to the invariant manifolds $W_{\e, \mathcal{V}_i}^{s, u}(\T_i)$. 
Let us define with
\begin{align*}
\Phi_0^t(z_0(0, x_i, \tau_1, \tau_2)) &=z_0(t, x_i, \tau_1, \tau_2)\\
&=\big(q_{h_i}(t, x_i), p_{h_i}(t, x_i), q_{h-h_i}(\tau_1+t), p_{h-h_i}(\tau_1+t), q_0(\tau_2+t), p_0(\tau_2+t)   \big),
\end{align*}
where $(q_{h-h_i}(\tau_1), p_{h-h_i}(\tau_1))$  is the periodic orbit contained in $\{ E_{\sigma_i}=h, E_{\sigma_{i+1}}=h-h_i, E_{\sigma_{i+2}}=0 \}$ such that $q_{h-h_i}(0)=\pi$.
Let us call $z^{s, u}=z^{s, u}(x_i, \tau_1, \tau_2)$ points in the intersection $W_{\e, \mathcal{V}_i}^{s, u}(\T_i)\cap \Pi(x_i, \tau_1, \tau_2)$. To find heteroclinic connections between the periodic orbit $P_{\sigma_i}$ and the torus $\T_{\sigma_i, \sigma_{i+1}}$ we measure the distance
\[
d_{\sigma_{i+k}}(\tilde{z}^u, z^s)=E_{\sigma_{i+k}}(\tilde{z}^u)-E_{\sigma_{i+k}}(z^s) \qquad k=1, 2.
\]
By reasoning as before we get
\begin{align*}
E_{\sigma_{i+k}}(\tilde{z}^u)&=\e\int_{-\infty}^0 \{ H_{1, i}, E_{\sigma_{i+k}} \}\circ \Phi^{\mathtt{t}}_{0}(\tilde{z}_0)\,d\mathtt{t}+\mathcal{O}(\e^2), \qquad k=\sigma_{i+1}, \sigma_{i+2},\\
E_{\sigma_{i+k}}(z^s)&=\begin{cases}
h-h_i+\e\int_{0}^{+\infty} \{ H_{1, i}, E_{\sigma_{i+1}} \}\circ \Phi^{\mathtt{t}}_{0}(z_0)\,d\mathtt{t}+\mathcal{O}(\e^2) \qquad k=1,\\[2mm]
\e\int_{0}^{+\infty} \{ H_{1, i}, E_{\sigma_{i+2}} \}\circ \Phi^{\mathtt{t}}_{0}(z_0)\,d\mathtt{t}+\mathcal{O}(\e^2) \qquad \qquad \qquad k=2.
\end{cases}
\end{align*}
Note that the term $h-h_i$ in $E_{\sigma_{i+1}}(z^s)$ arises from the fact that the iterates of $z^s$ tend to the torus $\T_{\sigma_{i+1}, \sigma_{i+2}}$ and then
\[
\lim_{t\to +\infty} E_{\sigma_{i+1}}(z^s(t, \tau_0, \tau_1, \tau_2))=h-h_i.
\]
Now we claim that
\[
E_{\sigma_{i+k}}(z^s)-E_{\sigma_{i+k}}(\tilde{z}^s)=\mathcal{O}(\e^2\log\e), \qquad k=1, 2.
\]
Indeed, by uniform hyperbolicity and the fact that the vector field of $H_1$ vanishes at the torus, the integral decays exponentially and so
\[
\int_{c |\log(\e)|}^{+\infty} \left|\{ H_{1, i}, E_{\sigma_{i+k}} \}\circ \Phi^{\mathtt{t}}_{0}(z_0) \right|\,d\mathtt{t}=\mathcal{O}(\e)
\]
for opportune $c>0$ independent of $\e$, and, since $h-h_i=\mathcal{O}(\e)$,
\[
\Phi^{\mathtt{t}}_{0}({z}_0)=\Phi^{\mathtt{t}}_{0}(\tilde{z}_0)+\mathcal{O}_{C^1}(\e ), \qquad \forall \mathtt{t}\in [-c |\log\e|, c |\log(\e)|].
\]
and therefore
\[
\int_{0}^{c |\log(\e)|} \left|\{ H_{1, i}, E_{\sigma_{i+k}} \}\circ \Phi^{\mathtt{t}}_{0}(z_0) -\{ H_{1, i}, E_{\sigma_{i+k}} \}\circ \Phi^{\mathtt{t}}_{0}(\tilde z_0)\right|\,d\mathtt{t}=\mathcal{O}(\e\log\eps).
\]
Then, we can conclude that
\begin{align*}
E_{\sigma_{i+1}}(\tilde{z}^u)-E_{\sigma_{i+1}}(z^s)&=
h-h_i+\e  \mathcal{M}_{\sigma_{i+1}}(h, x_i, \tau_1, \tau_2)+\mathcal{O}(\e^2\log\e),\\
E_{\sigma_{i+2}}(\tilde{z}^u)-E_{\sigma_{i+2}}(z^s)&=\e  \mathcal{M}_{\sigma_{i+2}}(h, x_i, \tau_1, \tau_2)+\mathcal{O}(\e^2\log\e).
\end{align*}
Then if $h-h_i\le \kappa \e$, for an opportune small $\kappa>0$, one can proceed as in the homoclinc case explained above to find zeros of the above equations, which give rise to heteroclinic connections between the periodic orbit $P_{\sigma_i}$ and nearby tori.
Proceeding analogously one can connect two nearby tori which are close to the periodic orbit $P_{\sigma_i}$. Indeed, note that by changing the parameterization of the unperturbed invariant manifold, one can rewrite the Melnikov potential \eqref{def:MelnikovPotential} as
\begin{align*}
\MoveEqLeft[4]{\mathcal{L}_j'(x_j,h_{j},, t_{j+1},h_{j+1},t_{j+2})=}\\
&\int_{-\infty}^{+\infty}H_{1,j}\left(q_{h_{j}}(s,x_j),p_{h_{j}}(s,x_j),q_{h_{j+1}}(s+t_{j+1},0),p_{h_{j+1}}(s+t_{j+1},0),q_0(s+t),p_0(s+t_{j+2})\right)dt.
\end{align*}
The difference between the two Melnikov potentials is on how one parameterizes the invariant torus, and more particularly its $(j+1)$ component. 
Note that in the $(q_{j+1}, p_{j+1})$ coordinates now we parameterize the invariant manifold by the initial time instead of using an initial condition. This implies that whereas $\cL_j$ is $2\pi$-periodic in $x_j$, $\cL_j'$ is $T$ periodic where $T$ is the period of the periodic orbit defined implicitly by $E_{\sigma_{j+1}}=h_{j+1}$. It can be easily checked that, for $h_{j+1}>0$  small, the Melnikov potential $\cL_j'$ is close to  \eqref{def:MelnikovPotential:per}.

Then Hypothesis \textbf{H3} implies that, for each $h_j\in (0,h)$ (and taking $h_{j+1}=h-h_j$) the function
\[
 (t_{j+1},t_{j+2})\to \mathcal{L}_j'(x_j,h_{j}, t_{j+1},h_{j+1},t_{j+2})
\]
has non-degenerate critical points. Then, proceeding as in the previous case one can prove that the torus $E_{j}=h_{j}$, $E_{j+1}=h-h_{j}$ with $h_j\in (h,h-\delta)$,  has transverse homoclinic connections and from them, construct heteroclinic connections to nearby tori.
\end{proof}

Lemmas \ref{lemma:TransitionArnold} and Lemma \ref{lemma:TransitionPeriodic} imply the following corollary.

\begin{corollary}\label{coro:transitiochainXj}
Assume that $H$ satisfies $\mathbf{H3}$ and $\mathbf{H4}$. Then there exists $\varepsilon_0>0$ such that for any $\varepsilon\in (0,\varepsilon_0)$ there exists $N>0$ and a sequence of tori $\T_{i,k}\subset \Lambda_i$, $k=1\ldots N$  such that 
\begin{align*}
 &W_{\eps,\cV_i}^u(P_{\sigma_i}) \pitchfork_{\cV_i}  W_{\eps,\cV_i}^s(\T_{i,0}),\,   W_{\eps,\cV_i}^u(\T_{i,k}) \pitchfork_{\cV_i}  W_{\eps,\cV_i}^s(\T_{i,k+1})\,\, \text{for}\,\, k=0,\ldots N-1\\
 & \text{and}\quad W_{\eps,\cV_i}^u(\T_{i,N}) \pitchfork_{\cV_i}  W_{\eps,\cV_i}^s(P_{\sigma_{i+1}}),
\end{align*}
where  $\pitchfork_{\cV_i}$ denotes transversality in the sense of Definition \ref{def:transversalityflows} applied to the invariant subspace $\cV_i$.
\end{corollary}

Note that this corollary gives a sequence of transverse heteroclinic connections. However {the transversality holds when} they are considered as orbits in the invariant subspace $\cV_i$ (within the energy level). To apply the Lambda Lemma given by Theorem \ref{thm:Lambda}, one needs that these heteroclinic connections are transverse in the whole phase space in the sense of Definition \ref{def:transversalityflows}. This is proven in next section.

\subsection{Transversality in the full phase space}\label{sec:fulltrans}

In this section we complete the proof of Theorem \ref{thm:transitiochainXj} by  proving the following.
\begin{proposition}\label{theorem:transitionchain}
Assume that $H$ satisfies $\mathbf{H1}$, $\mathbf{H2}$, $\mathbf{H3}$ and $\mathbf{H4}$. Then there exists $\varepsilon_0>0$ such that, for any $\varepsilon\in (0,\varepsilon_0)$, there exists $N>0$ and a sequence of tori $\T_{i,k}\subset \Lambda_i$, $k=1\ldots N$  such that 
\[
 W^u(P_i) \pitchfork  W^s(\T_{i,0}),\,\,   W^u(\T_{i,k}) \pitchfork  W^s(\T_{i,k+1})\quad \text{for}\quad k=0,\ldots N-1, \,\,\text{and}\,\,
 W^u(\T_{i,N}) \pitchfork  W^s(P_{i+1})
\]
where $\pitchfork$ denotes transversality in the whole phase space $\cM$ in the sense of Definition \ref{def:transversalityflows}.
\end{proposition}
Note that the only difference between this theorem and Corollary \ref{coro:transitiochainXj} is that the transversality refers to different spaces. In the corollary is within (the energy level) in 6 dimensional invariant subspace $\cV_i$ whereas in Proposition \ref{theorem:transitionchain} is  in the infinite dimensional phase space (always  in the sense of Definition \ref{def:transversalityflows}). See \cite{DelhamsGKP10} for a similar analysis in a finite dimensional setting.

\begin{proof}[Proof of Proposition \ref{theorem:transitionchain}]
To prove this proposition, let us consider $\gamma(t)$, one of the heteroclinic orbits connecting two of the tori given by Corollary \ref{coro:transitiochainXj}. To simplify the notation, in this proof we denote these tori by $\T_1$ and $\T_2$. They are characterized by $E_{\sigma_{i}}(\T_\ell)=h_\ell$, $\ell=1,2$, for some $h_1,h_2$, which are $\eps$-close, and they also satisfy $E_{\sigma_{i+1}}(\T_\ell)=h-h_\ell$.

Recall that we have defined 
\[
 W^{u,s}_{\eps,\cV_i}(\T_\ell)=W_{\eps}^{u,s}(\T_\ell)\cap\cV_i,\qquad \ell=1,2.
\]
Corollary \ref{coro:transitiochainXj} implies that  $W^{u}_{\eps,\cV_i}(\T_1)$ and  $W^{s}_{\eps,\cV_i}(\T_2)$ intersect transversally along $\gamma(t)$ within $\cV_i$ in the sense of Definition \ref{def:transversalityflows}. That is, for $t\in\R$,
\begin{equation}\label{def:TransvXj}
 T_{\gamma(t)}W^{u}_{\eps,\cV_i}(\T_1)\cap  T_{\gamma(t)} W^{s}_{\eps,\cV_i}(\T_2)=\langle \dot\gamma(t)\rangle\quad \text{and} \quad T_{\gamma(t)}W^{u}_{\eps,\cV_i}(\T_1)+  T_{\gamma(t)} W^{s}_{\eps,\cV_i}(\T_2)= \Ker dH|_{\cV_i}(\gamma(t))
\end{equation}
where $\langle \dot\gamma(t)\rangle$ is the one dimensional vector space generated by $\dot\gamma(t)$ and  $\Ker dH|_{\eps,\cV_i}(\gamma(t))$ is just the tangent space of  the energy level of $H$ restricted to $\cV_i$ at the point $\gamma(t)$ (recall that even if the Hamiltonian $H$ may be only formal, it becomes well defined when restricted to the finite dimensional subspace $\cV_i$ (but in any case its differential is well defined, even if it is not restricted).

Denoting
\begin{equation}\label{def:z}
z=(q_{\sigma_{i}},p_{\sigma_{i}},q_{\sigma_{i+1}},p_{\sigma_{i+1}}, q_{\sigma_{i+2}}, p_{\sigma_{i+2}}),
\end{equation}
the invariant manifolds $W^{u}_{\eps,\cV_i}(\T_1)$ and  $W^{s}_{\eps,\cV_i}(\T_2)$ can be parameterized as 
\begin{equation}\label{def:paramXjinvman}
\begin{split}
z^{u}(x_i,x_{i+1},t,h_1,\eps)&=z_0(x_i,x_{i+1},t,h_1)+\eps z_1^{u,s}(x_i,x_{i+1},t,h_1,\eps)\\
z^{s}(x_i,x_{i+1},t,h_2,\eps)&=z_0(x_i,x_{i+1},t,h_2)+\eps z_1^{u,s}(x_i,x_{i+1},t,h_2,\eps)
\end{split}
\end{equation}
where $z_0$ is the unperturbed homoclinic respectively of $\T_1$ and $\T_2$, that is,
\begin{equation}\label{def:paramXjinvman0}
z_0(x_i,x_{i+1},t,h_\ell)=\left(q_{h_\ell}(0,x_i),p_{h_\ell}(0,x_i),q_{h-h_\ell}(0,x_i),p_{h-h_\ell}(0,x_i),q_0(t),p_0(t)\right)
\end{equation}
(recall that $h_1$ and $h_2$ are $\eps$-close). Fix $T>0$. Then, for $\eps$ small enough, the parameterizations \eqref{def:paramXjinvman} are defined for 
\begin{equation}\label{def:domaininvman}
x_i,x_{i+1}\in\T\quad \text{and}\quad t\in (-\infty, T]\quad \text{ (unstable manifold) and}\quad t\in [-T, +\infty)\quad \text{ (stable manifold)}.
\end{equation}
Now we prove that the ``full'' invariant manifolds $W^{u}(\T_1)$ and  $W^{s}(\T_2)$ also intersect transversally along $\gamma(t)$ in the same sense as in \eqref{def:TransvXj} but in the whole phase space. To this end we ``enlarge'' the parameterizations \eqref{def:paramXjinvman} to parameterize (a portion of) $W^{u}(\T_1)$ and  $W^{s}(\T_2)$ instead of $W^{u}_{\eps,\cV_i}(\T_1)$ and  $W^{s}_{\eps,\cV_i}(\T_2)$.

To this end, we consider the Moser Normal Form coordinates \cite{Moser1956}
\[
 (q_k, p_k)=\Phi(\mathtt{u}_k, \mathtt{v}_k)
\]
for the integrable Hamiltonians $E_k$ with $k\neq \sigma_{i},\sigma_{i+1},\sigma_{i+2}$.  Then, 
\[
E_i\circ \Phi(\mathtt{u}_i,\mathtt{v}_i)= \cE_i(\tu_i \tv_i)=\tu_i \tv_i+\mathcal{O}_2\left(\tu_i \tv_i\right).
\]
We introduce the notation
\[
 \cI_i=\Z^m\setminus\{\sigma_i,\sigma_{i+1},\sigma_{i+2}\},\qquad \tu=\{\tu_k\}_{k\in\cI_i},\qquad \tv=\{\tv_k\}_{k\in\cI_i}.
\]
Then, the Hamiltonian \eqref{def:Hamiltonian} is transformed into 
\[
 \cH\left(z, \tu, \tv\right)=\sum_{k\in \cI_i}\cE_k(\tu_k \tv_k)+\sum_{k=\sigma_{i},\sigma_{i+1},\sigma_{i+2}}E_k(q_k,p_k)+\eps\cH_1 \left(z, \tu, \tv\right)
\]
where $\cH_1$ is the perturbation Hamiltonian  $H_1$ expressed in coordinates $(z, \tu, \tv)$, see \eqref{def:z} (To simplify the notation from now on in this section we reorder the variables to ``group'' the $\tu_k's$ and $\tv_k's$).

Recall that the tori $\T_1, \T_2$ are invariant both for the perturbed and unperturbed flows and are now characterized as
\[
\T_\ell: E_{\sigma_{i}}(q_{\sigma_{i}},p_{\sigma_{i}})=h_\ell, \quad  E_{{\sigma_{i+1}}}(q_{{\sigma_{i+1}}},p_{{\sigma_{i+1}}})=h-h_\ell,\quad q_{{\sigma_{i+2}}}=p_{{\sigma_{i+2}}}=\tu_k=\tv_k=0, \quad \,k\in\cI_i,\ell=1,2
\]
(and $h_1$ and $h_2$ are $\eps$-close). Analogously, the invariant subspaces $\cV_i$ are now defined as
\[
 \cV_i=\left\{\tu_k=0,\tv_k=0\quad\text{for}\quad k\in\cI_i\right\}.
\]
For the unperturbed Hamiltonian ($\eps=0$), the stable and unstable invariant manifolds of $\T_1$, $\T_2$  are parameterized by
\begin{equation}\label{def:Gamma_0}
 \begin{split}
\Gamma_0^u(x_i,x_{i+1}, t, \tu,h_1)&=\left(z_0(x_i,x_{i+1},t,h_1), \tu,0\right)\\
\Gamma_0^s(x_i,x_{i+1}, t, \tv,h_2)&=\left(z_0(x_i,x_{i+1},t,h_2), 0, \tv\right)
\end{split}
\end{equation}
where $z_0$ is the parameterization given in \eqref{def:paramXjinvman0}.

Note that the homoclinic manifold
\[
 \Gamma_0^u(x_i,x_{i+1}, t,0,h_1)=\Gamma_0^s(x_i,x_{i+1}, t, 0,h_2)=\left(z_0(x_i,x_{i+1},t,h_2), 0,0\right)
\]
is already transverse in the $(\tu, \tv)$ directions but it is not transverse in the $z=(q_i,p_i,q_{i+1},p_{i+1}, q_{i+2}, p_{i+2})$ directions.

Since the invariant manifolds $W_{\eps}^u(\T_1)$ and $W_{\eps}^s(\T_2)$ are regular with respect to $\eps$, one can consider parameterizations of the perturbed invariant manifolds close to \eqref{def:Gamma_0} (that is, parameterize the perturbed invariant manifolds as graphs with respect to the unperturbed ones).

Fix $\de>0$ and consider $T>0$ (see \eqref{def:domaininvman}) and recall the notation
\[
B_{\delta} (\ell^{\infty}_{S^c})=\left\{u:S^c\subset \mathbb{Z}^m\to\mathbb{R}:\|u\|_\infty\leq \de\right\}.
\]
where $S^c=\mathbb{Z}^m\setminus S$ and $S\subset \mathbb{Z}^m$ is a finite set (see \eqref{def:linftysc}). In this section, we take 
\[
 S=\{\sigma_{i},\sigma_{i+1},\sigma_{i+2}\}.
\]
Since Theorem \ref{thm:toriinvman} gives the existence of the invariant manifolds of the invariant tori and their regularity with respect to parameters, we can ensure that,  for $t\in (-\infty,T]$ and $\tu\in B_{\delta} (\ell^{\infty}_{S^c})$,
the unstable invariant manifold $W^u_{\e}(\T_1)$ has a parameterization of the form 
%
\[
 \Gamma_\eps^u(x_i,x_{i+1}, t, \tu)=\left(z^{u}(x_i,x_{i+1},t)+\eps F^u_z(x_i,x_{i+1},t, \tu), \tu+\eps  F^u_{\tu} (x_i,x_{i+1},t, \tu),\eps F^u_{\tv}(x_i,x_{i+1},t, \tu)\right)
\]
where $z^u$ is the parameterization introduced in \eqref{def:paramXjinvman} and  $F^u_z$, $F^u_{\tu}$, $F^u_{\tv}$  are some  $C^2$ functions\footnote{Theorem \ref{thm:toriinvman} proves the stronger statement that they are $C^2_\Gamma$. However, for this section it is enough to use $C^2$ regularity.}. These functions  depend on $h_1$ and $\eps$. We omit this dependence to avoid cluttering the notation.  Note that the fact that $\cV_i$ is invariant and $W_{\eps,\cV_j}^u$ is parameterized by \eqref{def:paramXjinvman} implies that 
\[
 F^u_\ast(x_i,x_{i+1},t,0)=0\qquad \text{for}\qquad \ast=z,u,v.
\]
Analogously, for $t\in [T,+\infty)$ and $v\in B_\infty(\rho)$, one has a paramaterization of $W_{\e}^s(\T_2)$ of the form 
\[
 \Gamma_\eps^s(x_i,x_{i+1}, t, \tv)=\left(z^{s}(x_i,x_{i+1},t)+\eps F^s_z(x_i,x_{i+1},t, \tv), \eps  F^s_{\tu} (x_i,x_{i+1},t, \tv),\tv+\eps F^s_{\tv}(x_i,x_{i+1},t, \tv )\right).
\]
for some $C^2$ functions $F^s_z$, $F^s_{\tu}$, $F^s_{\tv}$ which satisfy
\[
 F^s_\ast(x_i,x_{i+1},t,0)=0\qquad \text{for}\qquad \ast=z,u,v.
\]
The proof of Proposition \ref{theorem:transitionchain} is a consequence of the particular form of these parameterizations. Indeed, in the $\cV_i$ ``directions'', that is $z=(q_{\sigma_{i}},p_{\sigma_{i}},q_{\sigma_{i+1}},p_{\sigma_{i+1}}, q_{\sigma_{i+2}}, p_{\sigma_{i+2}})$, the $\tu$ and $\tv$ are small in $C^1$ norm since 
\[
 F^u_z(x_i,x_{i+1},t, \tu)=\mathcal{O}_{\ell^\infty}( \tu), \quad  F^s_z(x_i,x_{i+1},t, \tv)=\mathcal{O}_{\ell^\infty}( \tv).
\]
Therefore, the transversality obtained in Section \ref{sec:chain6dimspace} still holds.

In the $(q_k,p_k)$, $k\in \cI_i$, directions or, equivalently $(\tu,\tv)$ directions it is enough to check that, for any fixed $(x_i,x_{i+1}, t)$,  $\tu=\tv=0$ is a transverse zero of the function
\[
 \mathcal{G}(\tu,\tv)=\begin{pmatrix}
                      \tu+\eps  F^u_{\tu} (x_i,x_{i+1},t, \tu)-\eps  F^s_{\tu} (x_i,x_{i+1},t, \tv)\\\eps F^u_{\tv}(x_i,x_{i+1},t, \tu)-\tv-\eps F^s_{\tv}(x_j,x_{i+1},t, \tv )
                     \end{pmatrix}
\]
which measures the distance of the invariant manifolds on the plane $\mathtt{u}, \mathtt{v}$. Since 
\[
\partial_{(\mathtt{u}, \mathtt{v})}  \mathcal{G}|_{(\mathtt{u}, \mathtt{v})=(0, 0)}=\begin{pmatrix}
1+\e \partial_{\tu} F^u_{\tu}|_{\mathtt{u}= 0} & 0\\
0 & 1+\e \partial_{\tu} F^s_{\tv}|_{ \mathtt{v}=0} 
\end{pmatrix}
\]
and $F^u_{\tu}$,$F^s_{\tv}$ are $C^1$ functions, taking $\e>0$ small enough we easily deduce that $\tu=\tv=0$ is a transverse zero of $\mathcal{G}$.

The same argument can be carried out to prove the transversality of the invariant manifolds along heteroclinic orbits between the periodic orbit $P_i$ and the torus $\T_{i,0}$.

\end{proof}

\section{The Lambda lemma
}\label{sec:lambdalemma}
We devote this section to prove Theorem \ref{thm:Lambda}. We prove both a Lambda lemma for maps and for flows. First, in Section \ref{sec:normalform}, we prove a normal form which straightens the stable and unstable foliation (within the invariant manifolds). Then, in Section \ref{sec:proofLambdaLemma} we  proof a  Lambda Lemma for maps. Finally, in Section \ref{sec:lambdalemmaflow} we deduce the Lambda Lemma for flows from the one proven in Section \ref{sec:proofLambdaLemma} for maps.

\subsection{Normal Form}\label{sec:normalform}

In this section we use the notations \eqref{notation:linfS} used in Section \ref{sec:manifold}.
Let us consider the following hypothesis.
\begin{itemize}
\item[$(\mathbf{H})_{\mathrm{NF}}$] Assume that $f_3\in C^2_{\Gamma}(\mathcal{M}_{\delta}; \T^d)$, $f_3(0, 0, \theta, 0)=0$ and for $(w, \nu)\in \mathcal{M}_{\delta}\times (0, \mu)$ (recall \eqref{def:mapF})
\begin{align*}
 \| \partial_{\theta} f_3(w, \nu) \|_{\mathcal{L}_{\Gamma}(\R^n)} &\le K_{\theta} (\| x \|_{\ell^{\infty}}+\| y \|_{\ell^{\infty}}+|r|_d),
\end{align*}
where $\delta$ is introduced in Theorem \ref{thm:C2case} and $K_{\theta}$ is the constant introduced in \eqref{condition:lambdastrong}.
\end{itemize}
We consider the map $F_{\nu}$ in \eqref{def:F} and we neglect the dependence on $\nu$ if it makes no confusion. We define
\[
\omega_0:=\omega(0, 0, 0).
\]
\begin{theorem}\label{thm:normalform}
Let $F$ be a map of the form \eqref{def:F} that satisfies the assumptions of Theorem \ref{thm:C2case} and $(\mathbf{H})_{\mathrm{NF}}$. If $\delta$ and $\mu$ are small enough, then there exists a $C^2_{\Gamma}$ change of coordinates $\Phi$ which both maps $\mathcal{M}_{\delta}$ to $\mathcal{M}_{2\delta}$ and $\mathcal{M}_{j, \Gamma, \delta}$ to $\mathcal{M}_{j, \Gamma, 2\delta}$, which is $\mathcal{O}(\delta+L)$ close to the identity in the $C^1_{\Gamma}$ topology and such that:
\begin{equation}\label{def:form}
\tilde{F}:=\Phi^{-1} F \Phi=\big(\tilde{A}_-(\theta)\, x, \tilde{A}_+(\theta)\, y, \theta+{\omega}(x, y, r), \tilde{B}(\theta)\, r \big)+\tilde{f}(w)
\end{equation}
where $\omega$ is the function introduced in \eqref{def:F} and $\tilde{A}_{\pm}$, $\tilde{B}$, $\tilde{f}=(\tilde{f}_1, \tilde{f}_2, \tilde{f}_3, \tilde{f}_4)$ are $C^2_\Gamma$ and satisfy:
\begin{itemize}
\item[(i)] For $\theta\in \T^d$,
\[
\|\tilde{A}_{\pm}(\theta)-A_{\pm} (\theta)\|_{\mathcal{L}_{\Gamma}(\ell^{\infty})}=\mathcal{O}(\delta+L), \qquad  \|\tilde{B}(\theta)-B(\theta)\|_{\mathcal{L}_{\Gamma}(\R^d)}=\mathcal{O}(\delta+L).
\]
\item[(ii)] The restriction of $\Phi$ to $\T_0$ is the identity.
\item[(iii)] The stable and unstable invariant manifolds of $\T_0$ are respectively locally expressed by $\{(x, 0, \theta, 0)\}$ and $\{ (0, y, \theta, 0) \}$.
\item[(iv)] We have \begin{align*}
\tilde{f}_{1, 4}(0, y, \theta, 0)=0, \qquad \omega(0, y, 0)+\tilde{f}_{3}(0, y, \theta, 0)=\omega_0,\\
\tilde{f}_{2, 4}(x, 0, \theta, 0)=0, \qquad \omega(x, 0, 0)+\tilde{f}_{3}(x, 0, \theta, 0)=\omega_0.
\end{align*}
\item[(v)] The derivative of $\tilde{F}$ on $\T_0$ has the form
\[
D \tilde{F} (0, 0, \theta, 0)=\begin{pmatrix}
\tilde{A}_-(\theta) & 0 & 0 & 0\\
0 & \tilde{A}_+(\theta) & 0 & 0\\
0 & 0 & \id & P(\theta)\\
0 & 0 & 0 & \tilde{B}(\theta),
\end{pmatrix}
\]
for some function $P\in C^1_\Gamma(\T^d; \mathcal{L}_{\Gamma}(\R^d; \R^d))$.
\end{itemize}
\end{theorem}

\noindent\textbf{Proof of Theorem \ref{thm:normalform}}
We recall that, by Theorem \ref{thm:C2case}, we can express the invariant manifolds of $\T_0$, $W^s$ and $W^u$, locally as graphs of functions $\gamma^{s}=(\gamma_y^{s}, \gamma_r^{s})$, $\gamma^{u}=(\gamma_x^{u}, \gamma_r^{u})$ as
\[
(x, \gamma_y^s(x, \theta), \theta, \gamma_r^s(x, \theta)) \qquad (x, \theta)\in B_{\delta}(\ell^{\infty}) \times \T^d,
\]
\[
( \gamma_x^u(y, \theta), y, \theta, \gamma_r^u(y, \theta)) \qquad (y, \theta)\in B_{\delta}(\ell^{\infty}) \times \T^d
\]
for $\delta$ small enough. Recall the estimates in Theorem \ref{thm:C1case}
\begin{equation}\label{bound:C2norm}
\| \gamma^{s, u} \|_{C^1_{\Gamma}(B_{\delta}(\ell^{\infty}) \times \T^d;\ell^{\infty}\times \R^d)}\le \mathcal{O}(\delta+L).
\end{equation}
 The first step consists in straightening the invariant manifolds. We achieve this by performing two changes of variables. First we consider 
 \[
 \Phi_1(x, y, \theta, r)=\big(x+\gamma_x^u(y, \theta), y, \theta, r+\gamma_r^u(y, \theta) \big).
 \]
 It is easy to see that the inverse has the same form
 \[
  \Phi^{-1}_1(\tilde{x}, \tilde{y}, \tilde{\theta}, \tilde{r})=\big(\tilde{x}-\gamma_x^u(\tilde{y}, \tilde{\theta}), \tilde{y}, \tilde{\theta}, \tilde{r}-\gamma_r^u(\tilde{y}, \tilde{\theta}) \big).
 \]
 By \eqref{bound:C2norm}, the maps $\Phi_1^{\pm 1}$ are $\mathcal{O}(\delta+L)$-close to the identity in the $C^1_{\Gamma}$ topology. Moreover by the fact that $\gamma^{s, u}(0, \theta)=0$, these changes of coordinates are the identity when  restricted to $\T_0$.
Therefore $\Phi_1$ is well defined in a neighborhood of $\T_0$ and $\Phi_1^{-1} F \Phi_1$ has the form \eqref{def:F} with some function $\tilde{f}$ instead of $f$. 

We claim that, by the properties of $\gamma^{s, u}$ constructed in Theorem \ref{thm:C2case},  we have that $\tilde{f}$ satisfies the same assumptions of the function $f$ introduced in Section \ref{subsec:C2case} and $(\mathbf{H})_{\mathrm{NF}}$.

By direct computations we have that $\tilde{f}$ contains terms of the following form
\[
A_{-}(\theta)\gamma_x^u, \quad B(\theta)\gamma_r^u,  \quad \omega(x+\gamma_x^u, y, r+\gamma_r^u)-\omega(x, y, r)
\] and $f_j\,\circ\,\Phi_1^{-1}$ with $j=1, 2, 3, 4$. The fact that the derivatives of at most second order, and the Lipschitz constants of these derivatives, are bounded on the domains $\mathcal{M}_{\delta}$ and $\tB_{\delta}$ comes by composition. Hence we need to check two things: (i) the $C^1_{\Gamma}$-norm is of order $\mathcal{O}(\delta+L)$, (ii) the increment on the $\theta$-variables can be made arbitrarily small by considering smaller neighborhood of the torus $\T_0$ (namely we need estimates like \eqref{hyp:H3}). These properties hold for $\gamma^u$, then it is just a matter of applying the chain rule and the Faa De Bruno formula. We observe that assumption $(\mathbf{H})_{\mathrm{NF}}$ is needed when  we make derivatives in the angles of the term $f_3$.

It is easy to see that the stable invariant manifold reads in these variables (locally) as $\{x=0, r=0\}$. With abuse of notation, let $\gamma^s(x, \theta)$ be the parametrization of the stable manifold in these new variables. 
The second change of coordinates is 
 \[
 \Phi_2(x, y, \theta, r)=\big(x, y+\gamma_y^s(x, \theta), \theta, r+\gamma_r^s(x, \theta) \big)
 \]
 and its inverse is
 \[
 \Phi_2^{-1}(\tilde{x}, \tilde{y}, \tilde{\theta}, \tilde{r})=\big(\tilde{x}, \tilde{y}-\gamma_y^s(\tilde{x}, \tilde{\theta}), \tilde{\theta}, \tilde{r}-\gamma_r^s(\tilde{x}, \tilde{\theta}) \big).
 \] 
Reasoning as before we have that the unstable invariant manifold reads as $\{ y=0, r=0 \}$ and the conjugated map has the form \eqref{def:form}. We observe that $\Phi^{\pm 1}_2$ are the identity on the unstable manifold. We rename $\Phi_2^{-1} \Phi_1^{-1} F \Phi_1 \Phi_2$ by $F$.\\
  Now we find a change of variables such that $\tilde{f_3}\equiv 0$ on the stable manifold $W^s=\{ y=0, r=0 \}$. We look for a change of coordinates $\Phi_3$ whose inverse has the following form
\[
\Phi^{-1}_3(x, y, \theta, r)=(x, y, \theta+g(x, \theta), r)
\]
and conjugates $F$ restricted to $W^s$ to a map $\tilde{F}$ such that, restricted to $W^s$, in the $\theta$-component is just the rotation with frequency $\omega_0$. From the relation $\tilde{F}\circ \Phi^{-1}_{3_{|_{W^s}}}=\Phi_3^{-1}\circ F_{|_{W^s}}$ we find $g(x, \theta)$ by solving the fixed point equation
\begin{equation}\label{eq:fixedpoint}
\begin{aligned}
g(x, \theta)=\mathtt{T}(g(x, \theta)):=&\omega(x, 0, 0)-\omega_0+f_3(x, 0, \theta, 0)\\
&+g \big(A_-(\theta)x+f_1(x, 0, \theta, 0), \theta+\omega(x, 0, 0)+f_3(x, 0, \theta, 0)\big).
\end{aligned}
\end{equation}
We want a $C_{\Gamma}^2$-solution $g$ of \eqref{eq:fixedpoint} such that $\Phi_3^{-1}$ is invertible. We first prove this in the $C^1_{\Gamma}$ setting.
We introduce the space
\[
{\Xi}:=\left\{ g\in C_{\Gamma}^1( B_{\delta}(\ell^{\infty})\times \T^d; \R^d) : g(0, \theta)=0\,\,\mbox{and}\,\, \| \partial_x g \|_0, \| \partial_{\theta} g\|_1<\infty \right\}
\]
where
\[
\| \partial_x g \|_0:=\sup_{(x, \theta)\in B_{\delta}(\ell^{\infty})\times \T^d} \| \partial_x g(x, \theta) \|_{\mathcal{L}_{\Gamma}(\ell^{\infty}; \R^d)}, \qquad \| \partial_{\theta} g \|_1:=\sup_{\substack{(x, \theta)\in B_{\delta}(\ell^{\infty})\times \T^d,\\ x\neq 0}} \frac{\| \partial_{\theta} g(x, \theta) \|_{\mathcal{L}_{\Gamma}(\R^d)}}{\| x \|_{\ell^{\infty}}}.
\]
We equip this space with the norm
\[
\| g \|_{\Xi}:=\alpha_0 \| \partial_x g \|_0+\alpha_1 \| \partial_{\theta} g \|_1
\]
for $\alpha_0, \alpha_1$ to be opportunely chosen. Let $\kappa\in (0, 1)$ such that $\kappa>(\lambda^{-1}+L)(1+K_{\theta})^2$ (recall \eqref{assumption:constants}). We take $\delta$ and $\mu$ small enough and choose $\alpha_0, \alpha_1$ satisfying the condition 
\begin{equation}\label{cond:rapp}
\frac{2 K}{\kappa-(\lambda^{-1} + L)}<\frac{\alpha_0}{\alpha_1}<\left(\frac{1}{\delta} \right)\,\frac{\kappa-(\lambda^{-1}+L)(1+K_{\theta})}{2 K (\lambda^{-1}+L)}.
\end{equation}
 Now we prove that $\mathtt{T}$ maps $\Xi$ into itself. By the definition of $\omega_0$, Hypothesis $(\mathbf{H})_{\mathrm{NF}}$ and the fact that $f_1$ vanishes on the torus $\T_0$ implies that  $\mathtt{T} g(0, \theta)=0$. By $\mathbf{(H1)_{C^1}}$, $\mathbf{(H2)_{C^1}}$ and $\mathbf{(H3)_{C^1}}$ we have
\[
\| \partial_x \mathtt{T} g \|_0\le 2 K+\| \partial_x g\|_0 (\lambda^{-1}+L)+2 K \| \partial_{\theta} g \|_1 (\lambda^{-1}+L)\,\delta.
\]
Moreover, $(\mathbf{H})_{\mathrm{NF}}$ implies that
\[
\| \partial_{\theta} f_3(x, 0, \theta, 0) \|_{\mathcal{L}_{\Gamma}(\R^n)}\le \K_{\theta} \| x \|_{\ell^{\infty}}.
\]
By $\mathbf{(H2)_{C^1}}$ (in particular by \eqref{bound:dthetafk}) we have that
\[
\| \partial_{\theta} f_1(x, 0, \theta, 0) \|_{\mathcal{L}_{\Gamma}(\R^n; \ell^{\infty})}\le K \| x \|_{\ell^{\infty}}.
\]
By the explicit expression of $\partial_{\theta} \mathtt{T} g$ and the above inequalities
we deduce that 
\[
\| \partial_{\theta} T g \|_0\le \big[  K_{\theta}+2 K \| \partial_x g \|_0 +(\lambda^{-1} +L) (1+K_{\theta})\| \partial_{\theta} g \|_1  \big] \| x \|_{\ell^{\infty}}.
\] This proves that $\mathtt{T}\colon \Xi\to \Xi$. Now we see that by considering \eqref{cond:rapp} we have that $\mathtt{T}$ is a contraction. By the choice of $\kappa$ we have that
\begin{align*}
\| \mathtt{T} g_1-\mathtt{T} g_2 \|_{\Xi} &\le \Big( \alpha_0 (\lambda^{-1}+L)+2 \alpha_1 \K  \Big) \| \partial_x g_1-\partial_x g_2  \|_0\\
&+\Big( 2\,\alpha_0\, \delta\, \K (\lambda^{-1} + L) +\alpha_1 (1+K_{\theta}) (\lambda^{-1}+L) \Big)  \| \partial_{\theta} g_1-\partial_{\theta} g_2  \|_1\\
&\le \kappa \| g_1-g_2 \|_{\Xi}.
\end{align*}
{To prove that the fixed point is actually of class $C^2_{\Gamma}$ one can follow word by word the proof of Lemma $7.2$ in \cite{FontichM98}.}\\
To obtain that $\tilde{f}_3$ vanishes also on $W^u$ we repeat the same arguments as for $\Phi_3$. The desired change of coordinates has the form $\Phi_4(x, y, \theta, r)=(x, y, \theta+\tilde{g}(y, \theta), r)$ with $\tilde{g}(0, \theta)=0$. Note that $\Phi_4$ is the identity on $W^s$.\\
Thanks to the changes of coordinates $\Phi_1, \Phi_2, \Phi_3, \Phi_4$ we have that $\tilde{F}$ has the form \eqref{def:form} and the following derivatives vanish at $\T_0$
\[
\partial_{x} \tilde{F}_{j}, \quad j=2, 3, 4, \qquad \partial_{y} \tilde{F}_{j}, \quad j=1, 3, 4, \qquad \partial_{\theta} \tilde{F}_{j}, \quad j=1, 2, 4.
\]
Therefore we have obtained the following
\[
D \tilde{F}(0, 0, \theta, 0, \nu)=\begin{pmatrix}
\tilde{A}_-(\theta) & 0 & 0 & h_1(\theta)\\
0 & \tilde{A}_+(\theta) & 0 & h_2(\theta)\\
0 & 0 & \id & h_3(\theta)\\
0 & 0 & 0 & \tilde{B}(\theta)
\end{pmatrix}
\]
{for some {$C_{\Gamma}^1(\T^d)$} }
functions $h_j(\theta)$, $j=1, 2, 3$ and
 \begin{equation}\label{AtildeA}
\|\tilde{A}_{\pm}(\theta)-A_{\pm} (\theta)\|_{\mathcal{L}_{\Gamma}(\ell^{\infty})}=\mathcal{O}(\delta+L), \qquad  \|\tilde{B}(\theta)-B(\theta)\|_{\mathcal{L}_{\Gamma}(\R^d)}=\mathcal{O}(\delta+L).
\end{equation}
  
This comes from the assumptions $\mathbf{(H1)_{C^1}}$, $\mathbf{(H2)_{C^1}}$,  $\mathbf{(H3)_{C^1}}$, $(\mathbf{H})_{\mathrm{NF}}$ and the $C_{\Gamma}^1$-smallness of $\gamma^{s, u}$ \eqref{bound:C2norm}.

We want to prove that there exist $a_j(\theta)\in \mathcal{L}_{\Gamma}(\R^d; \ell^{\infty})$, $j=1, 2$ such that a change of coordinates of the form
\[
\Phi_5(x, y, \theta, r)=(x+a_1(\theta) r, y+a_2(\theta) r, \theta, r)
\]
is $C^2_{\Gamma}$ and it gives the differential of $G=\Phi_5^{-1} \tilde{F} \Phi_5$ on $\T_0$ in the desired form (see item (iv)). We observe that $\Phi_5$ is the identity on $\T_0$. Then $G(0, 0, \theta, 0)=\tilde{F}(0, 0, \theta, 0)$. We have to impose that
\[
D \Phi_5(\tilde{F}(0, 0, \theta, 0))\,DG (0, 0, \theta, 0)=D \tilde{F} (0, 0, \theta, 0)\,D \Phi_5(0, 0, \theta, 0).
\]
This is equivalent to solve the equations
\begin{align}
&\tilde{A}_-(\theta)\,a_1(\theta)+h_1(\theta)=a_1(\theta+\omega_0) \,\tilde{B}(\theta), \label{eq:1}\\ \label{eq:2}
&\tilde{A}_+(\theta)\,a_2(\theta)+h_2(\theta)=a_2(\theta+\omega_0) \,\tilde{B}(\theta).
\end{align}
We show how to find $a_2(\theta)$. 
We can invert $\tilde{A}_+(\theta)$ by using \eqref{AtildeA} and Lemma \ref{lem:neumann}. Then we write \eqref{eq:2} as a fixed point equation
\[
a_2(\theta)=\mathtt{B} (a_2(\theta)):=\tilde{A}^{-1}_+(\theta) \big( a_2(\theta+\omega_0) \,\tilde{B}(\theta)-h_2(\theta) \big).
\]
We look for a fixed point of $\mathtt{B}$ in 
\[
\Xi_{\rho}:=\{g\in  C^0(\T^d, \mathcal{L}_{\Gamma}(\R^d; \ell^{\infty})) , \| g\|_0\le \rho   \}
\]
 for some $\rho>0$, endowed with the norm 
\[
\| g \|_0:=\sup_{\theta\in\T^d} \| g (\theta)\|_{\mathcal{L}_{\Gamma}(\R^d; \ell^{\infty})}.
\]
First we prove that $\mathtt{B}\colon \Xi_{\rho}\to \Xi_{\rho}$ for 
\begin{equation}\label{ro}
\rho>\lambda^{-1} \| h_2 \|_0\,(1-\beta \lambda)^{-1}
\end{equation}
 (recall \eqref{condition:lambdastrong}). For $\delta$ and $\mu$ small enough,  applying \eqref{AtildeA} and Lemma \ref{lem:neumann} we have
\[
\| \mathtt{B} \,g \|_{0}\le (\lambda^{-1}+\mathcal{O}(\delta+L)) \Big( (\beta+\mathcal{O}(\delta+L)) \|g \|_{0}+\| h_2 \|_{0}\Big)\le  \rho.
\]
Proceeding analogously and applying \eqref{condition:lambdastrong} we see that $\mathtt{B}$ is a contraction
\[
\|  \mathtt{B} g_1-\mathtt{B} g_2  \|_{0}\le (\lambda^{-1} \beta+\mathcal{O}(\delta+L)) \| g_1-g_2 \|_{0}<\| g_1-g_2 \|_{0}.
\]
This provides the existence of a fixed point $a_2$ of $\mathtt{B}$ in $\Xi_{\rho}$.

\smallskip

 Now we prove that this fixed point is $C_{\Gamma}^1$. We look for $H$ such that $\partial_{\theta} [\mathtt{B}(a_2)]=H(a_2, \partial_{\theta} a_2)$. For that we differentiate formally $\mathtt{B}(g)$ and we substitute $\partial_{\theta} g$ with $\Psi$. We have
\[
H(g, \Psi):=\partial_{\theta} \tilde{A}^{-1}_+  \Big( \mathcal{T}_0\,g\,\tilde{B}-h_2 \Big)+\tilde{A}^{-1}_+ \Big(\mathcal{T}_0 \,\Psi\,\tilde{B} +\mathcal{T}_0\,g\,\partial_{\theta} \tilde{B}-\partial_{\theta} h_2 \Big)
\]
where $\mathcal{T}_0 f(\theta) :=f(\theta+\omega_0)$. We consider the ball
\[
D\Xi_{\kappa}=\{ \Psi\in C^0\big(\T^d, \mathcal{L}_{\Gamma}^2( \R^d; \ell^{\infty}) \big), \| \Psi \|_0\le \kappa   \}
\]
where $\kappa>0$ and the norm considered is
\[
\| \Psi\|_0:=\sup_{\theta\in \T^d} \| \Psi \|_{\mathcal{L}_{\Gamma}^2(\R^d; \ell^{\infty})}.
\]
\begin{lem}
Assume \eqref{condition:lambdastrong}. Take $\rho$ satisfying \eqref{ro} and $\kappa$ such that
\begin{equation}\label{cond:kappar}
\kappa>\frac{\rho\, \K\,(\beta+\lambda^{-1})}{1-\beta\lambda^{-1}}.
\end{equation}
Then if $g\in \Xi_{\rho}$ and $\Psi\in D\Xi_{\kappa}$ one has that $H(g, \Psi)\in D\Xi_{\kappa}$.
\end{lem}
\begin{proof}
By \eqref{def:KomegaKf} we have
\begin{align*}
\| H(g, \Psi) \|_{0} &\le \K ( \beta \| g \|_0+ \| h_2 \|_0)+\lambda^{-1} (\beta\, \| \Psi \|_0+\K  \| g \|_0+ \| h_2 \|_0)\\
&\le \rho \K ( \beta+ \lambda^{-1})+\kappa \beta \lambda^{-1}+(\delta+L) (\K+\lambda^{-1}).
\end{align*}
By considering $\delta$ and $\mu$ small enough the right hand side reduces to $\rho \K ( \beta+ \lambda^{-1})+\kappa \beta \lambda^{-1}$.
By \eqref{cond:kappar} we conclude that $\| H(g, \Psi) \|_{0}<\kappa$.
\end{proof}
\begin{lem}
If $\delta$ and $\mu$ are small enough, then $H(g, \cdot)\colon D \Xi_{\kappa}\to D \Xi_{\kappa}$ is a contraction uniformly in $g\in \Xi_{\rho}$.
\end{lem}
\begin{proof}
Since $H(g, \Psi_1)-H(g, \Psi_2)=\tilde{A}^{-1}_+(\theta) [\mathcal{T}_0 (\Psi_1-\Psi_2)] \tilde{B}$ we have
\[
\| H(g, \Psi_1)-H(g, \Psi_2) \|_0\le (\lambda^{-1} \beta+\mathcal{O}(\delta+L)) \| \Psi_1-\Psi_2\|_0.
\]
Then for $\delta$ and $\mu$ small enough, \eqref{assumption:constants} implies that $H(g, \cdot)$ is a contraction.
\end{proof}

It is immediate to see that the function $g\to H(g, \Psi)$ is continuous. Then we can apply the Fiber Contraction Theorem \ref{thm:fibrecontraction} and conclude that the fixed point is $C_{\Gamma}^1$. By using the assumptions of the $C_{\Gamma}^2$ case (see Section \ref{subsec:C2case}), we can prove in a similar way that the fixed point is $C_{\Gamma}^2$.

To obtain $a_1(\theta)$ in \eqref{eq:1}
one can reason in a similar way recalling the following fact: $\tilde{B}(\theta)$ is $O(\delta+L)$-close to $B(\theta)$ in $\mathcal{L}_{\Gamma}(\R^d)$-norm. Since $B(\theta)$ is invertible, $\tilde{B}(\theta)$ is invertible if $\delta$ and $\nu$ are taken small enough (see Lemma \ref{lem:neumann}).

\smallskip

It only remains to prove that all the changes of coordinates map $\mathcal{M}_{j,\Gamma, \delta}$ to $\mathcal{M}_{j, \Gamma,2\delta}$. This comes from the fact that all these transformations vanish at the torus $\{x=y=0, r=0 \}$ and Lemma \ref{lemma:SigmajGammazero}.

\subsection{A  Lambda Lemma for maps}\label{sec:proofLambdaLemma}

Let $\nu\in (0, \mu)$ for some $\mu>0$. Let us consider the complete metric space $\cM:=\ell^{\infty}\times \ell^{\infty}\times \T^d\times \R^d$ and the map
\[
F_{\nu}\colon \mathcal{M}_{\delta}:= B_{\delta} (\ell^{\infty})\times B_{\delta}(\ell^{\infty})\times \T^d\times B_{\delta}(\R^d)\subset \mathcal{M}\to \mathcal{M}
\]
given by
\begin{equation}\label{def:lambda}
\begin{aligned}
&F_{\nu}(w)=F_0(w)+f_{\nu}( w), \quad w:=(x, y, \theta, r),\\
 &F_0(w):=(A_-(\theta) x, A_+(\theta) y, \theta+\omega(x, y, r), B(\theta)\,r), \quad f_{\nu}( w):=\big(f_1(\nu; w), f_2(\nu; w), f_3(\nu; w), f_4(\nu; w) \big)\\
&f_1(\nu; \cdot), f_2(\nu; \cdot)\colon \mathcal{M}_{\delta}\to \ell^{\infty}, \quad f_3(\nu; \cdot)\colon\mathcal{M}_{\delta}\to \T^d, \quad f_4(\nu; \cdot) \colon \mathcal{M}_{\delta}\to \R^d.
\end{aligned}
\end{equation}
where 
\[
A_{\pm}(\theta)\in \mathcal{L}_{\Gamma}(\ell^{\infty}), \quad B(\theta)\in\mathcal{L}_{\Gamma}(\R^d).
\]
As for the normal form procedure performed in Section \ref{sec:normalform}, we assume that this map satisfies Hypotheses  $\mathbf{(H0)_{C^2}}$--$\mathbf{(H5)_{C^2}}$. 
We need an extra hypothesis: that the dynamics on the torus is a non-resonant rigid rotation. That is, we assume that the vector 
$\omega_0:=\omega(0, 0, 0)$ satisfies
\begin{equation}\label{def:nonresonantmaps}
\omega_0k+m\neq 0\qquad \text{for any}\qquad k\in\mathbb{Z}^d\setminus\{0\}, m\in\mathbb{Z}.
\end{equation}

Recall that we ultimately we want to prove a Lambda lemma for (formal) Hamiltonian vector fields. That is, for vector fields with formal first integrals. For this reason, we need a Lambda lemma which applies to maps $F_\nu$ with formal first integrals in the sense of Definition \ref{def:firstintegralmaps}.
Associated to the formal first integrals, we also introduce a related notion of transversality.

\begin{definition}\label{def:transversalitydegmaps}
Consider a map $F\colon \mathcal{M}_{\delta}\subset \mathcal{M}\to \mathcal{M}$ which has a formal first integral $G$ in the sense of Definition \ref{def:firstintegralmaps}.
 Fix $p\in\cM_\de$  and  consider two Banach submanifolds $\cN_1$ and $\cN_2$ of $\cM$ such that $p\in\cN_1,\cN_2$ and are invariant under $F$. Then, we say that $\cN_1,\cN_2$ intersect transversally at $p$ if 
\begin{enumerate}
\item They satisfy
\[
\dim\left( T_p\cN_1\cap T_p\cN_2\right) =1.
\]
\item The map
\[
 T: T_p\cN_1\times T_p\cN_2 \to T_p\cM,\qquad T(v_1,v_2)=v_1+v_2
\]
is a linear continuous map whose image is equal to $\Ker dG(p)$.
\end{enumerate}
\end{definition}

Let us give some explanation of Item 1 in this definition. The case of having formal first integrals will be applied ultimately to time $T$ maps of a Hamiltonian flows. Then, one has to take into account that the intersection of invariant manifolds contain the full trajectories and therefore it must have at least dimension 1 (see Definition \ref{def:transversalityflows}). Item 2 just means that the transversality has to be considered restricted to the leave of the foliation defined by $G$ (see Remark \ref{rmk:Frobenius}).

%


Next theorem provides a Lambda lemma for the invariant manifolds of invariant tori in various settings, both in $\cM_\de$ and $\cM_{j,\Gamma,\delta}$ and both for maps with and whithout formal first integrals.

\begin{theorem}\label{thm:LambdaLemmaMaps}
Consider a map $F_\nu$ of the form \eqref{def:lambda} and assume  $\mathbf{(H0)_{C^2}}$--$\mathbf{(H5)_{C^2}}$  and that the frequency $\omega_0:=\omega(0, 0, 0)$ is non-resonant as in \eqref{def:nonresonantmaps}. Then, the invariant torus $\T_0:=\{ x=y=0, r=0 \}$ possesses $C^2_{\Gamma}$ invariant manifolds 
$W^{s, u}\subset\mathcal{M}$ that satisfy the following.

Consider a $C^1_\Gamma$ submanifold $\Gamma\subset \mathcal{M}$ which intersects transversally the stable manifold $W^s$ at $q_0$ in the sense of Definition \ref{def:transversality}. Then
\begin{itemize}
 \item[$(i)$] The iterates of $\Gamma$ satisfy
 \[
W^u\subset \overline{\bigcup_{n\geq 0} F^n(\Gamma)}.
\]
where the closure is taken with respect to the  $\mathcal{M}$-topology.
\item[$(ii)$] Moreover there exists a submanifold $D$ of $\Gamma$ diffeomorphic to an open set of $\ell^{\infty}$ such that if $D_n$ is the connected component of $F^n(D)\cap \mathcal{M}_{\delta}$ which contains $F^n(q_0)$ then, for any $\e>0$, there exists $n_0$ such that $D_n$ is $\e$-close, in the $C^1_{\Gamma}(\mathcal{M}_{\delta})$ topology, to a subset of $W^u$ if $n>n_0$.
\end{itemize}
Moreover,
\begin{itemize}
\item[$(iii)$] Assumes that the map $F_\nu$ has a first integral $G$ in the sense of Definition \ref{def:firstintegralmaps}.
Assume furthermore that  $\Gamma$ intersects transversally $W^s$ in the sense of Definition \ref{def:transversalitydegmaps}. 
Then, the statements $(i)$ and $(ii)$ are also  satisfied.
\item[$(iv)$] If one considers as phase space $\mathcal{M}_{j,\Gamma}$ and a $C^1_\Gamma$ submanifold $\Gamma\subset\mathcal{M}_{j,\Gamma}$ the statements $(i)$, $(ii)$ and $(iii)$ are also  true with respect to the closure in the  $\mathcal{M}_{j,\Gamma}$-topology and the convergence in $C^1_\Gamma( \mathcal{M}_{j,\Gamma, \delta})$.
\end{itemize}
\end{theorem}
%
%

The assumptions in Item $(iii)$ are satisfied naturally if the map is the flow at time $T$ of a (formal) Hamiltonian vector field. 

We devote the rest of this section to prove this theorem. 
Recall that we consider constants $\lambda$, $\beta$, $K_{\theta}$ satisfying $\mathbf{(H0)_{C^2}}$. 
By applying the change of coordinates of Theorem \ref{thm:normalform} we can assume that $F_{\nu}$ satisfies the assumptions of Theorem \ref{thm:C2case} and
\begin{itemize}
\item[$\mathbf{(H1)_{\Lambda}}$] The frequency $\omega_0:=\omega(0, 0, 0)$ is non-resonant in the sense of \eqref{def:nonresonantmaps}
\item[$\mathbf{(H2)_{\Lambda}}$] For $j=1, 4$
\[
f_j(0, y, \theta, 0)=(0, 0, 0, 0), \qquad \omega(0, y, 0)+f_{3}(0, y, \theta, 0)=\omega_0.
\]
\item[$\mathbf{(H3)_{\Lambda}}$] For $j=2, 4$
\[
f_j(x, 0, \theta, 0)=(0, 0, 0, 0), \qquad \omega(x, 0, 0)+f_{3}(x, 0, \theta, 0)=\omega_0.
\]
\item[$\mathbf{(H4)_{\Lambda}}$] The derivative of $F_{\nu}$ on $\mathbb{T}_0:=\{ x=y=0, r=0 \}$ has the form
\[
\begin{pmatrix}
\tilde{A}_-(\theta) & 0 & 0 & 0\\
0 & \tilde{A}_+(\theta) & 0 & 0\\
0 & 0 & \id & P(\theta)\\
0 & 0 & 0 & \tilde{B}(\theta)
\end{pmatrix}
\]
where
\[
\|\tilde{A}_{\pm}(\theta)-A_{\pm} (\theta)\|_{\mathcal{L}_{\Gamma}(\ell^{\infty})}=\mathcal{O}(\delta+L), \qquad  \|\tilde{B}(\theta)-B(\theta)\|_{\mathcal{L}_{\Gamma}(\R^d)}=\mathcal{O}(\delta+L).
\]
and some function $P\in C^1_\Gamma(\T^d;\mathcal{L}_\Gamma (\R^d))$.
\end{itemize}
Now the local invariant manifolds of the torus $\T_0$ are
\[
W^s=\{ y=0, r=0 \}, \qquad   W^u=\{ x=0, r=0  \}.
\]
By hypotheses $\mathbf{(H2)_{\Lambda}}$-$\mathbf{(H3)_{\Lambda}}$ these invariant manifolds are $F_\nu$-invariant. Note also that $f_3$ restricted to $W^s\cup W^u$ vanishes. Then $F_{\nu}$ restricted to $\T_0$ is just the rotation of frequency $\omega_0$.\\
%

\begin{proposition}\label{thm:lambdalemma}
Consider a map $F_\nu$ of the form \eqref{def:form} and a $C_\Gamma^1$ submanifold $\Gamma\subset \cM_\de$. Assume that $\mathbf{(H1)_{\Lambda}}$--$\mathbf{(H4)_{\Lambda}}$ hold and that there exists  $q_0\in \Gamma\cap W^s$ where $\Gamma$ and $W^s$ intersect transversally in the sense of Definition \eqref{def:transversality}. 

Then, there exists $\zeta_0>0$ and  $Q\colon B_{\zeta_0}(\ell^{\infty})\to \mathcal{M}$, $Q(z)=(x(z), z, \theta(z), r(z))$, $Q(0)=q_0$ 
such that for all $p_0\in W^u$ and $\e>0$ there exists $j\in\mathbb{N}$ such that
\[
\{ F^j(Q(z)) : z\in B_{\zeta_0}(\ell^{\infty}) \} \cap \{ p\in \mathcal{M} : d(p, p_0)<\e \}\neq \emptyset.
\]
Moreover  there exists $j_0$ such that $D [F^j(Q(z))](\ell^{\infty})$ is $\e$-close to a subspace of $T W^u$ for $j>j_0$.

Finally,  assume that the map $F_\nu$ has a first integral in the sense of Definition \ref{def:firstintegralmaps} and that $\Gamma$ and $W^u$ intersect transversally at $q_0$ in the sense of Definition \ref{def:transversalitydegmaps}. 
Then, the same statements hold.
\end{proposition}

Statements $(i)$, $(ii)$ and $(iii)$ of Theorem \ref{thm:LambdaLemmaMaps} are a direct consequence of this proposition. We devote the rest of this section to prove Proposition \ref{thm:lambdalemma}.

First we need the following lemma which provides the function $Q_0(z)$. In Lemma \ref{lem:disc} we prove the existence of $Q_0$ assuming transversality in the sense of \eqref{def:transversality}. Then, in Lemma \ref{lem:disc:FI} we   construct it for maps with (formal) first integrals.

\begin{lem}\label{lem:disc}
Fix $q_0\in \Gamma\cap W^s$ and assume that $\Gamma$ and $W^s$ intersect transversally in the sense of Definition \ref{def:transversality}.
Then there exists $\sigma>0$ and a $C^1$ map $Q_0\colon B_{\sigma}(\ell^{\infty})\to \Gamma$ such that $Q_0(0)=q_0$, the image of $Q_0(z)$ can be written as
\[
\{ (x_0(z), z, \theta_0(z), r_0(z)) : z\in B_{\sigma}(\ell^{\infty}) \}.
\]
\end{lem}

\begin{proof}
Consider the components of $q_0=(x_0,y_0,\theta_0,r_0)$ and define the map  
\[
\pi_{y,r}:\Gamma\to \ell^\infty\times\R^d,\qquad  \pi_{y,r}(x,y,\theta,r)=(y,r).
\]
Since $\Gamma$ and $W^s$ intersect transversally  in the sense of Definition \ref{def:transversality} and $W^s=\{ y=0, r=0 \}$, the Implicit Function Theorem implies that $\pi_{y,r}$ is a local diffeomorphism in a small neighborhood $(y_0,r_0)$. Then, one has a local parameterization of $\Gamma$ as 
\[
 q(y,r)=(x(y,r), y, \theta (y,r), r),
\]
which satisfies  $q(y_0,r_0)=q_0$. Then, one can define the function $Q$ in the statement of the lemma as 
\[
 q(y)=(x(y,r_0), y, \theta (y,r_0), r_0).
\]
\end{proof}

Now we state an analogus lemma with deals with maps with formal first integrals.

\begin{lem}\label{lem:disc:FI}
Fix $q_0\in \Gamma\cap W^s$ and assume that $\Gamma$ and $W^s$ intersect transversally in the sense of Definition \ref{def:transversalitydegmaps} (with respect to a first integral $G$ in the sense of Definition \ref{def:firstintegralmaps}) at a point $q_0$.

Then there exists $\sigma>0$ and a $C^1$ map $Q\colon B_{\sigma}(\ell^{\infty})\to \Gamma$ such that $Q(0)=q_0$, the image of $Q(z)$ can be written as
\[
\{ (x_0(z), z, \theta_0(z), r_0(z)) : z\in B_{\sigma}(\ell^{\infty}) \}.
\]
\end{lem}
To prove this lemma we rely on the following lemma which analyzes the first integral $G$.
\begin{lem}\label{lem:goodfirstintegral}
Consider a map $F_\nu$ of the form \eqref{def:form} which satisfies $\mathbf{(H1)_{\Lambda}}$--$\mathbf{(H4)_{\Lambda}}$ and assume that it posesses a first integral $G$ in the sense of Definition \ref{def:firstintegralmaps}. Assume furthermore that the dynamics on its invariant torus $\T_0:=\{ x=y=0, r=0 \}$ is a 
a non-resonant rotation of frequency $\omega_0:=\omega(0, 0, 0)$ (with respect to the definition \eqref{def:nonresonantmaps}). 

Then, $F_\nu$ has  a (possibly different) first integral $\widetilde G$ satisfying $dG=d\widetilde G$ which is bounded (and constant) on the torus. Moreover,  $\widetilde G$ is also bounded on the the invariant manifolds of the torus and takes the same value as in the torus.
\end{lem}
\begin{proof}
Imposing the condition \eqref{cond:FI:maps} on the torus $\{x=y=r=0\}$ and relying on the particular form of $DF_\nu$ given by Hypothesis $\mathbf{(H4)_{\Lambda}}$ one obtains that 
\[
 \pa_\theta G(0,0,\theta+\omega,0)=\pa_\theta G(0,0,\theta,0)\qquad \text{ for all}\quad\theta\in\mathbb{T}_0.
\]
Since we are assuming that $\pa_\theta G$ is well defined and continuous, the minimality of the dynamics in the torus implies that  $\pa_\theta G(0,0,\theta,0)=0$.

Take any $\theta^* \in\mathbb{T}_0$ and define 
\[
 \widetilde G(x,y,\theta,r)=G(x,y,\theta,r)-G(0,0,\theta^*,0).
\]
Then, by the definition of  first integral (see Definition \ref{def:firstintegralmaps}), one can conclude that  $\widetilde G(0,0,\theta,0)=0$ for all $\theta\in\mathbb{T}_0$. 

Proceeding analgously and using that the dynamics on the stable manifold is an exponential contraction (analogously in the unstable manifold for the inverse map), one can also prove that 
\[
\begin{split}
 \pa_x G(x,0,\theta,0)&=0 \qquad \text{ for all }\quad \theta\in\mathbb{T}_0\quad \text{ and }\quad\|x\|_{\ell^\infty}\leq \delta\\
 \pa_y G(0,y,\theta,0)&=0 \qquad \text{ for all }\quad \theta\in\mathbb{T}_0\quad \text{ and }\quad\|y\|_{\ell^\infty}\leq \delta.
\end{split}
 \]
In conclusion,
\[
 \widetilde G|_{W^s}=\widetilde G|_{W^u}=0.
\]
\end{proof}

We use this lemma to prove Lemma \ref{lem:disc:FI}.
\begin{proof}[Proof of Lemma \ref{lem:disc:FI}]
From the proof of Lemma \ref{lem:goodfirstintegral}, one can conclude that for points in the torus $\theta\in \mathbb{T}_0$, one has that 
\[
 dG(0,0,\theta,0)=(0,0,0,\pa_rG(0,0,\theta,0)).
\]
We assume that 
\[
 \pa_rG(0,0,\theta,0)\neq 0
\]
(note that if this is true for one $\theta\in \mathbb{T}_0$, it is also true for any $\theta$). This implies that  $\Ker dG(p)$ defines a Banach subspace of codimension 1. The case $\pa_rG(0,0,\theta,0)=0$ can be handled as in Lemma \ref{lem:disc}.

We assume without loss of generality that $\pa_{r_1}G(0,0,\theta,0)\neq 0$. Note,  that since $dG$ is continuous by assumption (see Definition \ref{def:firstintegralmaps}),  taking $\de>0$ small enough, 
\[
 \pa_{r_1}G(p)\neq 0\qquad \text{for all}\qquad p\in\mathcal{M}_\delta.
\]
Therefore, if we denote $\tilde r=(r_2,\ldots, r_d)$, one has that for $p\in\mathcal{M}_\delta$,
\[
 \Pi:\Ker dG(p)\to \ell^{\infty}\times\ell^{\infty}\times\mathbb{R}^d\times\mathbb{R}^{d-1},\qquad \Pi (x,y,\theta,r)=(x,y,\theta,\tilde r)
\]
is an isomorphism. (Note that we are abusing notation and denoting by $\theta$ points in the tangent space of the torus).

Moreover, note that by Definition \ref{def:transversalitydegmaps}, the transverse intersection between $W^u$ and $\Gamma$ implies that its intersection has dimension one. Let us denote by $\mathcal{N}_{q_0}$ its tangent space at the point $q_0$. Then, by Remark \ref{rmk:complement}, we can find a complement to $\mathcal{N}_{q_0}$, that is a subspace $\widetilde{\Gamma}_{q_0}$ such that
\[
 T_{q_0}\Gamma=\mathcal{N}_{q_0}\oplus\widetilde{\Gamma}_{q_0}
\]
which implies that 
\[
T_{q_0}W^s\oplus \widetilde{\Gamma}_{q_0}=\Ker dG(q_0).
\]
Thus, since $W^s=\{y=0,r=0\}$, one has that \[
\pi_{y,\tilde r}:\widetilde\Gamma_{q_0}\to \ell^\infty\times\R^{d-1},\qquad  \pi_{y,\tilde r}(x,y,\theta,\tilde r)=(y,\tilde r).
\]
is an isomorphism. Then, one can proceed as in the proof of Lemma \ref{lem:disc}.

\end{proof}

\noindent{\textbf{Proof of Proposition \ref{thm:lambdalemma}}}. 
We can assume that $p_0$ and $q_0$ belong to $\mathcal{M}_{\delta}$, otherwise we can consider iterations of these points through $F^{-1}$ and $F$ respectively. In the normal form coordinates 
\[
p_0=(0, y_0, \varphi_0, 0), \qquad q_0=(x_0, 0, \theta_0, 0)
\]
with $\| x_0\|_{\ell^{\infty}}$, $\| y_0 \|_{\ell^{\infty}} \le \delta$ . First we study the form of the differential of the map $F$ on $\mathcal{M}_{\delta}$. We have
\begin{equation}\label{def:matrix2}
DF_{|_{\mathcal{M}_{\delta}}}=\begin{pmatrix}
\tilde{A}_-+a_{11} & a_{12} & a_{13} & a_{14}\\
a_{21} & \tilde{A}_++a_{22} & a_{23} & a_{24}\\
a_{31} & a_{32} & \id+a_{33} & a_{34}\\
a_{41} & a_{42} & a_{43} & \tilde{B}+a_{44}
\end{pmatrix}.
\end{equation}
Since $F\in C^2_{\Gamma}$ then the coefficients $a_{ij}$ are $C_{\Gamma}^1$ functions (in the corresponding spaces) and by $\mathbf{(H4)_{\Lambda}}$ we have 
\begin{equation}\label{aij}
a_{i j}(0, 0, \theta, 0)=0 \qquad \mbox{for}\,\,\,(i, j)\neq (3, 4).
\end{equation}
Let $\lambda_0$ be such that (recall Hypothesis $\mathbf{(H0)_{C^2}}$),
\begin{equation}\label{bound:betalambda1}
1\le \beta<\lambda_0<\lambda.
\end{equation}
Given $\eta_0>0$, let $\kappa>0$ be such that
\begin{align}
&\beta+3\kappa<\lambda_0<\lambda-\kappa(1+\eta_0),   \label{bound:acaso1}\\
&\lambda_0<\lambda-\frac{7\kappa}{4},\\   \label{bound:acaso3}
&0<\frac{\kappa+(\beta+\kappa)\eta_0}{\lambda_0}<\eta_0,\\ \label{bound:acaso4}
& \lambda^{-1}+2 \kappa<1.
\end{align}
By $\mathbf{(H2)_{\Lambda}}$ we have  that for $j=1, 3, 4$,
\[
a_{j 2}(0, y, \theta, 0)=0.
\]
Since $a_{j2}$ are $C^1_{\Gamma}$ functions there exists a constant $K_a>0$ such that
\begin{equation}\label{bound:aj2}
\| a_{j 2}(w) \|_{\mathcal{L}_{\Gamma}}\le K_a (\| x \|_{\ell^{\infty}}+|r|_d) \qquad \forall w\in \mathcal{M}_{\delta}, \qquad j=1, 3, 4.
\end{equation}
In the above inequality and from now on, when giving statements on the norms of the $a_{i j}$'s we abuse the notation and we do not specify the domain when writing $\mathcal{L}_{\Gamma}$.
By assumption, recall \eqref{def:KomegaKf}, we have $\| a_{3 4} \|_{\mathcal{L}_{\Gamma}}\le 2 \K$. We can assume that
\begin{equation}\label{bound:shrink}
2 \K<1, \qquad K_a\le 1.
\end{equation}
Indeed we could consider the scaled variables $(\alpha x, \alpha y, \theta, \alpha r)$ and work in the neighborhood $\mathcal{M}_{\delta/\alpha}$. Then $\|  a_{j 2} \|_{\mathcal{L}_{\Gamma}}\le \alpha K_a (\| x \|_{\ell^{\infty}}+|r|_d) $, $\| a_{34} \|_{\mathcal{L}_{\Gamma}} \le 2 \K \alpha $ and taking $\alpha$ small enough we can obtain \eqref{bound:shrink}.

By $\mathbf{(H4)_{\Lambda}}$ and \eqref{aij} we can consider $\delta$ small enough such that 
\begin{equation}\label{bound:kappa}
\sup_{w\in \mathcal{M}_{\delta}} \| a_{ij}(w) \|_{\mathcal{L}_{\Gamma}}\le \kappa \qquad \mbox{for}\,\,\,(i, j)\neq (3, 4).
\end{equation}
We introduce $\tau$ such that
\begin{equation}\label{bound:tau}
\tau>1, \quad \tau>\frac{2}{\lambda_0-(\beta+3\kappa)}>0, \quad \tau>\frac{4 \K}{\lambda-\lambda_0-\frac{7\kappa}{4}}, \quad \tau>\frac{1}{\delta}.
\end{equation}
Since we can consider $p_0$ arbitrarily close to $\T_0$ we take $p_0$ such that
\[
\| y_0 \|_{\ell^{\infty}}<\frac{1}{\tau}.
\]
Now we study the restriction of the differential of $F$ at $W^s\cap \mathcal{M}_{\delta}$. By $ \mathbf{(H3)_{\Lambda}}$ we have
\begin{equation}\label{def:matrix}
DF_{|_{W^s\cap \mathcal{M}_{\delta}}}=\begin{pmatrix}
\tilde{A}_-+a_{11} & a_{12} & a_{13} & a_{14}\\
0 & \tilde{A}_++a_{22} & 0 & a_{24}\\
0 & a_{32} & \id & a_{34}\\
0& a_{42} & 0 & \tilde{B}+a_{44}
\end{pmatrix}.
\end{equation}

Now we consider the graph defined by $Q_0(z)$ found either in Lemma \ref{lem:disc} or Lemma \ref{lem:goodfirstintegral} and we iterate it by $F$. We prove that after some iterations its tangent space at $z=0$ is close enough to the space spanned by $\partial/\partial y$. Then we continue iterating the graph until it is close enough to $W^u$. We recall that
\[
Q_0(0)=q_0, \quad \pa_z Q_0(z)=(\pa_z x_0(z), \id, \pa_z \theta_0(z),\pa_z  r_0(z))
\]
and we choose $\eta_0$ such that $\| \pa_z r_0(0)\|_{\mathcal{L}_{\Gamma}(\ell^{\infty}; \R^d)}\le \eta_0$.
We define the following sequence
\[
\tilde{Q}_{k+1}(z)=(\tilde{x}_{k+1}(z), \tilde{y}_{k+1}(z), \tilde{\theta}_{k+1}(z), \tilde{r}_{k+1}(z) ):=F ( Q_k(z)).
\]
We shall prove that $\pa_z \tilde{y}_{k+1}(0)$ is invertible at each step. Then we would be able to define
\[
Q_{k+1}(z)=(x_{k+1}(z), y_{k+1}(z), \theta_{k+1}(z), r_{k+1}(z)):=\tilde{Q}_{k+1} ( \mathcal{R}_{k+1}(z))
\]
where
\[
\mathcal{R}_{k+1}:=\big(\partial_{z} \tilde{y}_{k+1}(0)\big)^{-1}\in\mathcal{L}_{\Gamma}(\ell^{\infty}).
\]
The reparametrization $\mathcal{R}_{k+1}$ of the iterations of $q_0(z)$ implies that {$\partial_{z} y_{k+1}(0)=\id$}.

Since $Q_k(0)=\tilde{Q}_k(0)=F^k(q_0)=(x_k(0), 0, \theta_k(0), 0)\in W^s$ tends to the torus each $Q_k(z)$ is defined for small $z$. Let us define
\[
A^{(k)}_{\pm}:=\tilde{A}_{\pm}(\theta_k), \quad B^{(k)}:=\tilde{B}(\theta_k), \quad a_{ij}^{(k)}=a_{ij}(q_k(0)), \quad \kappa_k:=\max \left\{ \| a_{i j}^{(k)} \|_{\mathcal{L}_{\Gamma}}, \,\,j\neq 4 \right\}.
\]
By \eqref{bound:kappa} we have that $\kappa_k\le \kappa$. By the Hypothesis $\mathbf{(H4)_{\Lambda}}$ on the differential of $F_{\nu}$ on $\T_0$, the fact that $Q_k(0)$ tends to the torus and the fact that $F_{\nu}$ is $C^2_{\Gamma}$ we have that $\kappa_k\to 0$.
By \eqref{def:matrix} we have that

\begin{equation*}
\pa_z\tilde{Q}_{k+1}(0)=DF(Q_k(0)) [\pa_zQ_k(0)]=\begin{pmatrix}
(A^{(k)}_-+a_{11}^{(k)})\,\pa_zx_k(0)+a_{12}^{(k)}+a_{13}^{(k)} \pa_z\theta_k(0)+a_{1 4}^{(k)} \,\pa_zr_k(0)\\[2mm]
A^{(k)}_+ + a_{22}^{(k)}+a_{2 4}^{(k)}\,\pa_zr_k(0)\\[2mm]
a_{3 2}^{(k)}+\pa_z\theta_k(0)+a_{3 4}^{(k)} \,\pa_zr_k(0)\\[2mm]
a_{4 2}^{(k)}+(B^{(k)}+a_{44}^{(k)}) \,\pa_zr_k(0)
\end{pmatrix}.
\end{equation*}
This allows to define the recurrences
\begin{align*}
&\pa_z x_{k+1}(0)=\Big((A^{(k)}_-+a_{11}^{(k)})\,\pa_zx_k(0)+a_{12}^{(k)}+a_{13}^{(k)}\pa_z \theta_k(0)+a_{1 4}^{(k)} \,\pa_zr_k(0)\Big) \Big( A^{(k)}_++a_{22}^{(k)}+a_{2 4}^{(k)}\,\pa_zr_k(0) \Big)^{-1},\\
&\pa_zy_{k+1}(0)=\id,\\
&\pa_z\theta_{k+1}(0)= \Big( a_{3 2}^{(k)}+\pa_z\theta_k(0)+a_{3 4}^{(k)} \,\pa_zr_k(0) \Big)  \Big( A^{(k)}_++a_{22}^{(k)}+a_{2 4}^{(k)}\,\pa_zr_k(0) \Big)^{-1},\\
&\pa_zr_{k+1}(0)=\Big( a_{4 2}^{(k)}+(B^{(k)}+a_{44}^{(k)}) \,\pa_zr_k(0)\Big)  \Big( A^{(k)}_++a_{22}^{(k)}+a_{2 4}^{(k)}\,\pa_zr_k(0) \Big)^{-1}.
\end{align*}

We introduce
\[
m_k:=\max\left\{ \| \pa_zx_k(0)\|_{\mathcal{L}_{\Gamma}(\ell^{\infty})}, \| \pa_z\theta_k(0)\|_{\mathcal{L}_{\Gamma}(\ell^{\infty}; \R^d)}\right \}, \quad \eta_k:=\|\pa_z r_k(0) \|_{\mathcal{L}_{\Gamma}(\ell^{\infty}; \R^d)}.
\]
We claim that
\begin{align}
&m_{k+1}\le \frac{m_k+\kappa_k+\eta_k}{\lambda_0} \label{rec1}\\ \label{rec2}
&\eta_{k+1}\le \frac{\kappa_k+(\beta+\kappa)\eta_k}{\lambda_0}\\ \label{rec3}
&\eta_{k+1}\le \eta_0.
\end{align}
We prove these bounds by induction. Hence suppose that those bounds hold for $k$ and let us prove it for $k+1$.
First we see that 
\[
\pa_z\tilde{y}_{k+1}(0)=A^{(k)}_++a_{22}^{(k)}+a_{2 4}^{(k)}\,\pa_zr_k(0)
\]
is invertible. Using the inductive hypothesis \eqref{rec3} we have that $\| a_{22}^{(k)}+a_{2 4}^{(k)}\,\pa_zr_k(0) \|_{\mathcal{L}_{\Gamma}(\ell^{\infty})}\le \kappa_k+\kappa \eta_k\le \kappa(1+\eta_0)$.  By $\mathbf{(H1)_{C^1}}$ we have $\| \big(A_+^{(k)} \big)^{-1}\|_{\mathcal{L}_{\Gamma}(\ell^{\infty})}<\lambda^{-1}$.  Then
\[
\| \big(A_+^{(k)} \big)^{-1}\|_{\mathcal{L}_{\Gamma}(\ell^{\infty})} \,\,\| a_{22}^{(k)}+a_{2 4}^{(k)}\,\pa_z r_k(0) \|_{\mathcal{L}_{\Gamma}(\ell^{\infty})}<\frac{\kappa(1+\eta_0)}{\lambda}\stackrel{\eqref{bound:acaso1}, \eqref{bound:betalambda1}}{<} 1.
\]
This proves that $\pa_z\tilde{y}_{k+1}(0)$ is invertible by Neumann series (see Lemma \ref{lem:neumann}) and
\begin{equation}\label{bound:inv}
\| (\pa_z\tilde{y}_{k+1}(0))^{-1}\|_{\mathcal{L}_{\Gamma}(\ell^{\infty})}\le \frac{\| \big(A_+^{(k)}\big)^{-1}\|_{\mathcal{L}_{\Gamma}(\ell^{\infty})}}{1-\| \big(A_+^{(k)}\big)^{-1}\|_{\mathcal{L}_{\Gamma}(\ell^{\infty})} \| a^{(k)}_{22}+a_{ 2 4}^{(k)} \,\pa_zr_k(0) \|_{\mathcal{L}_{\Gamma}(\ell^{\infty})}}\le \lambda_0^{-1}.
\end{equation}
By \eqref{bound:inv} and \eqref{bound:acaso4} we get the estimate for $\| \pa_zx_{k+1}(0)\|_{\mathcal{L}_{\Gamma}(\ell^{\infty})}$.

By \eqref{bound:shrink}, the definition of $\kappa_k$ and $\eta_k$ and the bound \eqref{bound:inv} we obtain the estimate for $\|\pa_z\theta_{k+1}(0)\|_{\mathcal{L}_{\Gamma}(\ell^{\infty}; \R^d)}$.
By \eqref{bound:kappa} we get the bound for \eqref{rec2}. Since $\kappa_k\le \kappa$ for all $k$, we obtain \eqref{rec3} by \eqref{bound:acaso3}.
Now we need the following lemma.

\begin{lem}\label{lem:analisi}
Let $(a_k)_k, (b_k)_k, (\xi_k)_k$ be sequences of positive real numbers such that $a_k\to 0$ and $b_k\le b<1$ and $\xi_{k+1}\le a_k+b_k\,\xi_k$, $k\geq 1$. Then $\xi_k\to 0$.
\end{lem}

We consider $b_k\rightsquigarrow \frac{\beta+\kappa}{\lambda_0}$ and $a_k\rightsquigarrow \frac{\kappa_k}{\lambda_0}$. By \eqref{bound:acaso3} we have $b_k\equiv b< 1$ and $a_k\to 0$ because $\kappa_k\to 0$. Then by Lemma \ref{lem:analisi} the sequence $\eta_k\to0$.

We consider $b_k\rightsquigarrow \lambda_0^{-1}$ and $a_k\rightsquigarrow \frac{\kappa_k+ \eta_k}{\lambda_0}$. By \eqref{bound:acaso3} we have $b_k\equiv b< 1$ and $a_k\to 0$. Then by Lemma \ref{lem:analisi} the sequence $m_k\to0$.

Therefore, there exists $k_0$ such that if $k\geq k_0$, then $m_k, \eta_k<\varepsilon/4$  and $\| x_k(0)\|_{\ell^{\infty}}<\varepsilon/(2\tau)$.
We consider the graph $Q_{k_0}(z)$ and we reparametrize it in such a way that $y_{k_0} (z)=\id$ . This is
possible, locally at $z = 0$, because $\pa_z y_{k_0} (0) = \id$. This does not change the bounds $m_{k_0}$ and $\eta_{k_0}$.
We denote again by $Q_0(z)$ the resulting graph $Q_{k_0} (z)$.
Now we perform the second step where we iterate this new $Q_0(z), z \in B_{\zeta_0}$ , with $\zeta_0>0$ chosen below.
We consider again a sequence of reparametrized iterations
\[
\tilde{Q}_{k+1}(z)=(\tilde{x}_{k+1}(z), \tilde{y}_{k+1}(z), \tilde{\theta}_{k+1}(z), \tilde{r}_{k+1}(z) ):=F ( Q_k(z)), 
\]
\[
Q_{k+1}(z)=(x_{k+1}(z), z, \theta_{k+1}(z), r_{k+1}(z)):=\tilde{Q}_{k+1} ( \mathcal{Q}_{k+1}(z))
\]
where
\begin{equation}\label{calQ}
\mathcal{Q}_{k+1}:=\tilde{y}_{k+1}^{-1}.
\end{equation}
We note that this time the reparametrization $\mathcal{R}_{k+1}$ is not linear. The sequence is defined for
\[
\| z \|_{\ell^{\infty}}<\zeta_{k+1}:=\min\{  \lambda_0 \zeta_k, 1/\tau \}.
\]
Let
\[
m_k:=\sup_{z\in B_{\zeta_k}(\ell^{\infty})}\{ \| \pa_zx_k(z) \|_{\mathcal{L}_{\Gamma}(\ell^{\infty})},  \| \pa_z\theta_k(z) \|_{\mathcal{L}_{\Gamma}(\ell^{\infty}; \R^d)}\} \quad \eta_k:= \sup_{z\in B_{\zeta_k}(\ell^{\infty})} \{ \| \pa_zr_k(z) \|_{\mathcal{L}_{\Gamma}(\ell^{\infty}; \R^d)}     \}.
\]
We consider the neighborhood of $W^u$ defined by
\[
V_{\varepsilon/\tau}=\{ (x, y, \theta, r)\in \mathcal{M}_{\delta} : \| x \|_{\ell^{\infty}}+|r|_d<\varepsilon/\tau \}.
\]
Since $z\mapsto Q_0(z)$ is $C^1_\Gamma$ we can take $\zeta_0$ such that $m_0<\varepsilon/2$, $\eta_0<(\alpha \e) /2$, with $\alpha>0$ to be determined, and $Q_0(z)\in V_{\varepsilon/\tau}$ for all $z\in B_{\zeta_0}(\ell^{\infty})$. We check inductively that
\begin{itemize}
\item $Q_{k+1}$ is well defined.
\item $Q_{k+1}(z)\in V_{\varepsilon/\tau}$ for all $z\in B_{\zeta_{k+1}}(\ell^{\infty})$.
\item $m_{k+1}<\varepsilon/2$, $\eta_{k+1}<(\alpha \e) /2$.
\end{itemize}
We need to prove that $\tilde{y}_{k+1}$ is invertible, $B_{\lambda_0 \zeta_k}(\ell^{\infty})\subset \tilde{y}_{k+1}(B_{\zeta_k}(\ell^{\infty}))$ and
\begin{equation*}
\| \pa_z (\tilde{y}^{-1}_{k+1}) (z) \|_{\ell^{\infty}}<\lambda_0^{-1} \qquad z\in B_{\lambda_0 \zeta_k}(\ell^{\infty}).
\end{equation*}
We write
\[
\tilde{y}_{k+1}(z)=h_1(z)+h_2(z), \quad h_1(z)=\tilde{A}_+(\theta_k(z))\,z, \quad h_2(z)=f_2(Q_k(z)).
\]
Note that $\tilde{A}_+\in C^1_{\Gamma}$ and, by Hypothesis $\mathbf{(H1)_{C^1}}$, the operator $\tilde{A}_+(\theta_k(0))$ is a linear invertible operator. Then,  using  that $\theta_k$ is $C^1_\Gamma$ and that $m_k\leq\eps/2$, it is easy to see that $h_1$ is invertible.
%
Moreover, there exists $K>0$ such that 
\[
\sup_{z\in B_{\zeta_k}(\ell^{\infty})} \| h_1 (z) \|_{\ell^{\infty}} \geq \zeta_k \left(\lambda- \K \frac{\e}{2} \right).
\]
Then for $\varepsilon$ small enough $B_{\lambda \zeta_k}(\ell^{\infty})\subset h_1(B_{\zeta_k}(\ell^{\infty}))$. Using the above bound we obtain
\[
\lip(h_1^{-1})>\frac{1}{\lambda-\K\,\frac{\e}{2}\, \zeta_k}.
\]
Regarding $h_2$, by \eqref{bound:kappa} (recall that $Q_k(z)\in V_{\varepsilon/\tau}\subset \mathcal{M}_{\delta}$ by inductive hypothesis) and $m_k<\varepsilon/2$ we have that
\[
\lip(h_2)\le \left\lVert\pa_z h_2(z) \right\rVert_{\mathcal{L}(\ell^{\infty})} \le \kappa \left(1+\frac{3 \varepsilon}{2} \right).
\]
{By Theorem $1.5$ and Proposition $1.7$ in \cite{Hirsch1969StableMF} we have that $\tilde{y}_{k+1}=h_1+h_2$ is invertible and its inverse $\mathcal{Q}_{k+1}$ is defined on $B_{\zeta^*_{k+1}}(\ell^{\infty})$ with
\[
\zeta^*_{k+1}:=\zeta_k \left(\lambda-\frac{\e \K}{2\tau} \right)\left( 1-\frac{\kappa(1+3\e/2)}{\lambda-\frac{\e \K}{2\tau} } \right).
\]}
By \eqref{bound:tau} we have $\zeta^*_{k+1}\geq \lambda_0 \zeta_k$ and
\[
\lip\big( (\tilde{y}_{k+1})^{-1}\big)\le \frac{1}{\lambda-\frac{\varepsilon \K}{2\tau}-\kappa (1+\frac{3\e}{2}) }\le \lambda_0^{-1}.
\]
Let us denote by 
\[
\big(\mathtt{x}_k(z), \mathtt{y}_k(z), \vartheta_k(z), \mathtt{r}_k(z)\big)^T:=\pa_zQ_k(\mathcal{Q}_{k+1}(z)) [\pa_z\mathcal{Q}_{k+1}(z)].
\]
Then
\[
\pa_zQ_{k+1}(z)=DF \big(Q_k (\mathcal{Q}_{k+1}(z)) \big) \begin{pmatrix}
\mathtt{x}_k(z)\\ \mathtt{y}_k(z)\\ \vartheta_k(z)\\ \mathtt{r}_k(z)
\end{pmatrix}.
\]
We proved that $\|\mathcal{Q}_{k+1}(z)\|_{\ell^{\infty}}\le \zeta_k$ for all $\| z \|_{\ell^{\infty}}\le \zeta_k$. Then $Q_k (\mathcal{Q}_{k+1}(z))\in V_{\e/\tau}$.
Recall \eqref{calQ}, we have
\[
\sup_{z\in B_{\zeta_k}(\ell^{\infty})} \max\{  \| \mathtt{x}_k(z) \|_{\ell^{\infty}}, |\vartheta_k(z)|_d  \}\le \frac{m_k}{\lambda_0}, \qquad  \sup_{z\in B_{\zeta_k}(\ell^{\infty})} | \mathtt{r}_k(z) |_{d}\le \frac{\eta_k}{\lambda_0}, \qquad \sup_{z\in B_{\zeta_k}(\ell^{\infty})}   \| \mathtt{y}_k(z) \|_{\ell^{\infty}}\le \frac{1}{\lambda_0}.
\]
 By using \eqref{def:matrix2}, \eqref{bound:aj2} , \eqref{bound:kappa} we have
\[
\begin{split}
\| \pa_zx_{k+1}(z) \|_{\mathcal{L}_\Gamma(\ell^{\infty})}&\le \frac{(\lambda^{-1}+2\kappa) m_k}{\lambda_0}+\frac{\varepsilon}{\tau \lambda_0}+\frac{ \kappa \eta_k}{\lambda_0},\\
\| \pa_z\theta_{k+1}(z) \|_{\mathcal{L}_\Gamma(\ell^{\infty}; \R^d)}&\le\frac{\e}{\tau\lambda_0}+\frac{(1+2\kappa)m_k}{\lambda_0}+\frac{\eta_k}{\lambda_0},\\
\|\pa_z r_{k+1}(z) \|_{\mathcal{L}_\Gamma(\ell^{\infty};\R^d)}&\le \frac{2\kappa m_k}{\lambda_0}+\frac{(\beta+\kappa) \eta_k}{\lambda_0}+\frac{\varepsilon}{\tau \lambda_0}.
\end{split}
\]
Therefore we have
\[
\begin{split}
m_{k+1}&\le \max\left\{     \frac{(\lambda^{-1}+2\kappa) m_k}{\lambda_0}+\frac{\varepsilon}{\tau \lambda_0}+\frac{ \kappa \eta_k}{\lambda_0},     \frac{\e}{\tau\lambda_0}+\frac{(1+2\kappa)m_k}{\lambda_0}+\frac{\eta_k}{\lambda_0}   \right\},\\
\eta_{k+1}&\le \frac{2\kappa m_k}{\lambda_0}+\frac{(\beta+\kappa) \eta_k}{\lambda_0}+\frac{\varepsilon}{\tau \lambda_0}.
\end{split}
\]
If we choose $\alpha$ such that
\begin{equation}\label{cond:alpha}
\begin{split}
  \frac{(\lambda^{-1}+2\kappa) }{2\lambda_0}+\frac{1}{\tau \lambda_0}+\frac{ \kappa \alpha}{2\lambda_0}< \frac{1}{2},\\
   \frac{1}{\tau\lambda_0}+\frac{(1+2\kappa)}{2\lambda_0}+\frac{\alpha}{2\lambda_0}< \frac{1}{2},\\
   \frac{2\kappa}{2\lambda_0}+\frac{(\beta+\kappa) \alpha}{2\lambda_0}+\frac{1}{2\tau \lambda_0}<\frac{\alpha}{2}
\end{split}
\end{equation}
and we consider \eqref{bound:acaso1}-\eqref{bound:acaso4} then $m_{k+1}<\varepsilon/2$ and $\eta_{k+1}<(\alpha \e)/2$.
\begin{remark}
The conditions \eqref{bound:acaso1}-\eqref{bound:acaso4} and \eqref{cond:alpha} are satisfied by taking $\tau$ large enough, $\kappa$ small enough and $\alpha$ appropriately.
\end{remark}
We are left to check that $\{Q_{k+1}(z), z\in B_{\zeta_{k+1}}\}$ is contained in $V_{\e/\tau}$. We have
\[
\| x_{k+1}(z)-x_{k+1}(0) \|\le \sup_{z\in B_{\zeta_{k+1}}} \| \pa_zx(z) \| \,\zeta_{k+1}\le m_{k+1}\,\zeta_{k+1}\le \frac{\e}{2}\,\zeta_{k+1}.
\]
Since, by assumption $\zeta_{k+1}\le \tau^{-1}$ and therefore
\[
\| x_{k+1}(0)\|<\e/(2\tau), 
\]
and thus $\| x_{k+1}(z) \|_{\ell^{\infty}}\le \e/\tau$. We reason in the same way for $r_{k+1}$ (it is in fact easier because $r_{k+1}(0)=0$).

Since $\lambda_0^k \zeta_k$ increases geometrically we can assume that there exists $k_1$ such that $\zeta_k=\tau^{-1}$ for all $k\geq k_1$.

We recall that $p_0=(0, y_0, \varphi_0, 0)$. As $q_k(0)$ tends to the torus and the frequency on the torus is irrational there exists $k_2\geq k_1$ such that $| \theta_{k_1}(0)-\varphi_0 |_d<\e/2$. Then taking $z=y_0$
\[
d( Q_{k_1}(y_0), p_0 )\le \max\left\{ \frac{\e}{\tau}, | \theta_{k_1}(y_0)-\theta_{k_1}(0)|_d+| \theta_{k_1}(0)-\varphi_0 |_d  \right\}
\]
\[
\le \max\left\{ \frac{\e}{\tau}, \frac{\e}{2 \tau}+\frac{\e}{2} \right\}<\e.
\]
{The statement about the $C^1_\Gamma$ closeness follows because $m_k<\e/2$.}

\medskip

Fixed $j$, all the statements hold true of one replaces $\mathcal{M}_{\delta}$ with $\mathcal{M}_{j, \Gamma, \delta}$ and $\ell^{\infty}$ with $\Sigma_{j, \Gamma}$.

\subsection{A Lambda Lemma for flows}\label{sec:lambdalemmaflow}

Theorem \ref{thm:LambdaLemmaMaps} provides a Lambda Lemma for quasiperiodic tori of infinite dimensional maps with decay properties. We use this theorem to deduce an analogous statement for quasiperiodic tori of infinite dimensional vector fields. This result will be applicable to the vector field defined by the Hamiltonians of the form \eqref{def:Hamiltonian} and will imply Theorem \ref{thm:Lambda}.

To this end, we consider a $C^1_\Gamma$ vector field $\XX_\nu$ defined in $\mathcal{M}_\de$ for $\nu\in (0,\mu)$ for some $\mu>0$. We assume that it is of the form  
\begin{equation}\label{def:vectorfield:lambda}
\XX_\nu=\XX_0+\cF_\nu\quad \text{with}\quad   \XX_0(w)=\begin{pmatrix} \cA_-(\theta) x\\\cA_+(\theta) y\\ \omega(x,y,r)\\ \cB(\theta)r\end{pmatrix}
\quad \text{and}\quad
\cF_{\nu}( w)=\begin{pmatrix}\cF_1(\nu; w)\\ \cF_2(\nu; w)\\ \cF_3(\nu; w)\\ \cF_4(\nu; w) \end{pmatrix}
\end{equation}
where
\[
\cA_{\pm}(\theta)\in \mathcal{L}_{\Gamma}(\ell^{\infty}_{S^c}), \quad \cB(\theta)\in\mathcal{L}_{\Gamma}(\R_S^d)
\]
and $\cF_{1}, \cF_2 (\nu; \cdot)\colon \mathcal{M}_{\delta}\to \ell^{\infty}_{S^c}$, $\cF_3(\nu; \cdot)\colon\mathcal{M}_{\delta}\to \T_S^d$, $\cF_4(\nu; \cdot) \colon \mathcal{M}_{\delta}\to \R_S^d$.

We assume that $\cX_\nu$ satisfies Hypotheses $\mathbf{(H1)_{f}}$--$\mathbf{(H5)_{f}}$. In particular $\T_0:=\{ x=0, y=0, r=0 \}$
is an invariant torus for the flow associated to the vector field map $F_{\nu}$ for all $\nu\in(0,\mu)$ and its dynamics is given by 
\[
 \dot \theta=\tilde\omega_0=\tilde\omega(0,0,0).
\]
By Theorem \ref{thm:invmanflows}, Hypotheses  $\mathbf{(H1)_{f}}$--$\mathbf{(H5)_{f}}$ also imply that the torus $\T_0$ has $C^2_{\Gamma}$ invariant manifolds.

We assume moreover that the dynamics in the torus is non-resonant. That is
\begin{equation}\label{def:nonresonantflows}
 \tilde\omega_0\cdot k\neq 0\qquad \text{ for all }k\in\mathbb{Z}^d\setminus\{0\}
\end{equation}
(compare with the non-resonance condition for maps in \eqref{def:nonresonantmaps}). Under this conditions we can prove a Lambda lemma for the flow associated to $\cX_\nu$.

\begin{theorem}\label{thm:LambdaFlows}
 Consider a $\C^1_\Gamma$ vector field $\cX_\nu$ of the form \eqref{def:vectorfield:lambda} and assume that it satisfies Hypotheses $\mathbf{(H1)_{f}}$--$\mathbf{(H5)_{f}}$ and that it has a formal first integral $G$ in the sense of Definition \ref{def:firstintegralflows}. Denote by $\Phi^t$ its associated flow. 
Assume also that the frequency $\tilde\omega_0:=\tilde\omega(0, 0, 0)$ is non-resonant in the sense of \eqref{def:nonresonantflows}. 

Then, the invariant torus $\T_0:=\{ x=y=0, r=0 \}$ possesses $C^2_{\Gamma}$ invariant manifolds 
$W^{s, u}\subset\mathcal{M}$ that satisfy the following.

Consider a $C^1_\Gamma$ submanifold $\Gamma\subset \mathcal{M}$ which intersects transversally the stable manifold $W^s$ at $q_0$ in the sense of Definition \ref{def:transversalityflows} (with respect to the formal first integral $G$). Then
\begin{itemize}
 \item The iterates of $\Gamma$ satisfy
 \[
W^u\subset \overline{\bigcup_{t\geq 0} \Phi^t(\Gamma)}.
\]
where the closure is taken with respect to the  $\mathcal{M}$-topology.
\item Moreover there exists a submanifold $D$ of $\Gamma$ diffeomorphic to an open set of $\ell^{\infty}$ such that if $D_t$ is the connected component of $\Phi^t(D)\cap \mathcal{M}_{\delta}$ which contains $\Phi^t(q_0)$ then, for any $\e>0$, there exists $t_0$ such that $D_t$ is $\e$-close, in the $C^1_{\Gamma}(\mathcal{M}_{\delta})$ topology, to a subset of $W^u$ if $t>t_0$.
\end{itemize}
Moreover, if one considers as phase space $\mathcal{M}_{j,\Gamma}$ and a $C^1_\Gamma$ submanifold $\Gamma\subset\mathcal{M}_{j,\Gamma}$ the same statements are true with respect to the closure in the  $\mathcal{M}_{j,\Gamma}$-topology and the convergence in $C^1_\Gamma( \mathcal{M}_{j,\Gamma, \delta})$.
\end{theorem}

\begin{proof}[Proof of Theorem \ref{thm:Lambda}]
We deduce  Theorem \ref{thm:Lambda} from the Lambda Lemma for maps stated in Theorem \ref{thm:LambdaLemmaMaps}.

Consider the map $\Phi^T$ for some $T>0$ to be chosen later. As explained in Section \ref{sec:invmanflows}, this map satisfies the Hypotheses $\mathbf{(H1)_{C^2}}$--$\mathbf{(H5)_{C^2}}$. Moreover, the formal integral $G$ of the vector field $\cX_\nu$ is also a formal integral of the map $\Phi^T$ in the sense of Definition \ref{def:firstintegralmaps}. Then, to apply Theorem \ref{thm:LambdaLemmaMaps} it only remains to check  that the dynamics of $\Phi^T$ restricted to the torus is non-resonant in the sense of \eqref{def:nonresonantmaps}.

Indeed, the frequency at the torus for the map $\Phi^T$ is 
\[
 \omega_0=T\tilde\omega_0
\]
where $\tilde\omega_0$ satisfies \eqref{def:nonresonantflows}. It is well known that for  any non-resonant vector $\tilde\omega_0$ in the sense of \eqref{def:nonresonantflows} one can choose $T>0$ so that $(\tilde\omega_0, T^{-1})$ satisfies also \eqref{def:nonresonantflows} as a $(d+1)$-dimensional vector. Equivalently, for this choice of $T$, $ \omega_0=T\tilde\omega_0$ satisfies \eqref{def:nonresonantmaps}. Then, one can apply Theorem \ref{thm:LambdaLemmaMaps} which implies
\[
W^u\subset \overline{\bigcup_{n\geq 0} \Phi^{nT}(\Gamma)}\subset \overline{\bigcup_{t\geq 0} \Phi^t(\Gamma)}.
\]
Finally Theorem \ref{thm:LambdaLemmaMaps} and the $C^2_\Gamma$ regularity of  $\Phi^{T}$ implies the $C_\Gamma^1$ convergence.
\end{proof}

 \appendix

\section{Proof of Lemma \ref{lem:ball2}}
\label{sec:appendix}

The following lemmata can be proved with straightforward computations applying the assumptions of Theorem \ref{thm:C2case} and Lemma \ref{lemh}.

\begin{lem}\label{lem:A1}
Recall the definition \ref{def:Sigma2cMv}.
For $\gamma\in \Xi^2_{c, M, \mathtt{v}}$, $\Upsilon\in D \Xi^2_{\mathtt{v}, \widetilde{\mathtt{v}}}$ and $(x, \theta)\in B_{\delta}(\ell^{\infty})\times \T^d$ we have for $j=0, 1$ 
\begin{align}
\lip_x (\Upsilon_j\circ h) &\le \wkappa_j \lambda^{-1}+2 K \wrho_j\, (1+c)+\cO(L),   \label{lipxPsi}\\   \label{lipthetaPsi}
\lip_{\theta} (\Upsilon_j\circ h) &\le \wrho_j (1+K_{\theta})+\cO(\delta),\\
\lip_x (\Upsilon_2\circ h) &\le \wkappa_2 \lambda^{-1}+\cO(\delta),\\
\lip_{\theta} (\Upsilon_2\circ h)(x)  &\le \Big(\widehat{M} \lambda^{-1} (1+K_{\theta})+\wkappa_2 K(2+c)
+\mathcal{O}(L) \Big) \| x\|_{\ell^{\infty}}.
\end{align}
\end{lem}

\begin{lem}\label{lem:A2}
For $\delta$ and $\mu$ small enough the functions $h_1$, $h_2$ in \eqref{h} satisfy
\begin{multicols}{2}
\noindent
\begin{align*}
 \lip_x \partial_x^2 h_1 &\le \tC+L \wkappa_0,\\
 \lip_x \partial_x^2 h_2 &\le \tC+2 K \wkappa_0,\\
 \lip_x \partial_{\theta x}^2 h_1 &\le \mathtt{C}+\K (1+c) \,\wrho_1,\\
 \lip_x \partial_{\theta x}^2 h_2 &\le  \mathtt{C}+K\wkappa_1,\\
 \lip_x \partial_{\theta}^2 h_1 &\le  \mathtt{C}+L \wkappa_2,\\
 \lip_x \partial_{\theta}^2 h_2 &\le \mathtt{C}+\mathtt{C} \wkappa_2,
\end{align*}
\columnbreak
\noindent\begin{align*}
 \lip_{\theta} \partial_x^2 h_1 & \le \tC+L\,\wrho_0 \\
 \lip_{\theta} \partial_x^2 h_2 &\le \mathtt{C} +2 K  \,\wrho_0,\\
 \lip_{\theta} \partial_{\theta x}^2 h_1 &\le \mathtt{C}+L \wrho_1\\
 \lip_{\theta} \partial_{\theta x}^2 h_2 &\le  \mathtt{C}+2 K \wrho_1.\\
 \lip_{\theta} \partial_{\theta}^2 h_1 &\le 2 K\, \lip_{\theta} \Upsilon_2,\\
 \lip_{\theta} \partial_{\theta}^2 h_2 &\le 2 K\, \lip_{\theta} \Upsilon_2.
\end{align*}
\end{multicols}
\end{lem}

\begin{lem}\label{lem:A3}
We have, for all $z=(x, \theta)\in \dom$,
\begin{align*}
&\lip_x (\partial_x\gamma)(h(z)), \lip_x (\partial_\theta\gamma)(h(z)), \lip_{\theta} (\partial_x\gamma)(h(z))\le \tC,\\
&\lip_{\theta} (\partial_\theta\gamma)(h(z))\le \tC \| x \|_{\ell^{\infty}}.
\end{align*}
\end{lem}

We use these three lemmas to prove Lemma \ref{lem:ball2}.
\begin{proof}[Proof of Lemma \ref{lem:ball2}]
The proof is divided in several steps. We use repeatedly, without mentioning, Lemmas \ref{lem:A1}, \ref{lem:A2} and \ref{lem:A3}.

(i) We prove that $\lip_x \cH_0\le \wkappa_0$.
We have by Lemma \ref{rmk:hderivatives}, \eqref{lipxPsi} taking $\delta$ and $\mu$ small enough,
\begin{align*}
\,\lip_x \cH_0\le & \beta\lambda^{-3} \wkappa_0+\mathtt{C} (1+\wkappa_1+\wkappa_2+\wrho_0+\wrho_1)
\end{align*}
for some constant $\mathtt{C}>0$. By \eqref{condition:lambdastrong} we have $\beta \lambda^{-3}<1$ and by condition \eqref{cond:kzero} we can conclude. 

(ii) We prove $\lip_x \cH_1\le \wkappa_1$. By \eqref{lipxPsi} and Lemma \ref{rmk:hderivatives}, for all $z\in\dom$,
\[
\lip_x (\Upsilon_1 (h(z))) \|\partial_x h_1(z)\|_{\mathcal{L}_{\Gamma}(\ell^{\infty})}  \| \partial_{\theta} h_2(z)\|_{\mathcal{L}_{\Gamma}(\R^d)}\le \lambda^{-2} (1+K_{\theta}) \wkappa_1+\tC \wrho_1+\cO(\delta).
\]
Therefore by taking $\delta$ and $\mu$ small enough,
\begin{align*}
\lip_x \cH_1\le \lambda^{-2} (1+K_{\theta}) \wkappa_1+\tC (1+\wkappa_2+ \wrho_1)
\end{align*}
for some constant $\mathtt{C}>0$.
By \eqref{condition:lambdastrong} we have $\beta\lambda^{-2} (1+K_{\theta})<1$, then by \eqref{cond:kUno} we get the claim .

(iii) We prove $\lip_x \cH_2\le \wkappa_2$. Since for all $z\in \dom$
\[
\lip_x (\Upsilon_2(h(z)))\,\|\partial_{\theta} h_2(z)\|_{\mathcal{L}_{\Gamma}(\R^d)}^2\le \wkappa_2 \lambda^{-1} (1+K_{\theta})^2+\cO(\delta),
\]
then by taking $\delta$ and $\mu$ small enough
\[
\lip_x \cH_2\le \tC+ \wkappa_2 \beta \lambda^{-1} (1+K_{\theta})^2
\]
for some constant $\mathtt{C}>0$.
By \eqref{condition:lambdastrong} we have $\beta\lambda^{-1} (1+K_{\theta})^2<1$, then by \eqref{cond:kUno} we get the claim .
 
(iv) We prove that $\lip_{\theta} \cH_0\le \wrho_0$. By \eqref{lipthetaPsi}, for all $z\in \dom$,
\[
\lip_{\theta} \Upsilon_0(h(z))\,\|\partial_x h_1(z)\|_{\mathcal{L}_{\Gamma}(\ell^{\infty})}^2\le \wrho_0 (1+K_{\theta}) \lambda^{-2} +\mathcal{O}(\delta).
\]
Then by taking $\delta$ and $\nu$ small enough
\[
\lip_{\theta} \cH_0\le \mathtt{C}(1+\wrho_1+\wkappa_2)+\wrho_0 \beta (1+K_{\theta}) \lambda^{-2}
\]
for some constant $\mathtt{C}>0$.
By \eqref{condition:lambdastrong} we have $\beta\lambda^{-2} (1+K_{\theta})<1$, then by \eqref{cond:rozero} we get the claim .\\

(v) We prove that $\lip_{\theta} \cH_1\le \wrho_1$. By \eqref{lipthetaPsi}, for all $z\in \dom$,
\[
\lip_{\theta} \Upsilon_1(h(z))\,\|\partial_x h_1(z)\|_{\mathcal{L}_{\Gamma}(\ell^{\infty})}\,\|\partial_{\theta} h_2(z)\|_{\mathcal{L}_{\Gamma}(\R^d)}\le \wrho_1 (1+K_{\theta})^2 \lambda^{-1} +\mathcal{O}(\delta).
\]
Then by taking $\delta$ and $\nu$ small enough
\[
\lip_{\theta} \cH_1\le \mathtt{C}+\wrho_1 \beta (1+K_{\theta})^2 \lambda^{-1}
\]
for some constants $\mathtt{C}>0$.
By \eqref{condition:lambdastrong} we have $\beta\lambda^{-1} (1+K_{\theta})^2<1$, then by \eqref{cond:rozero} we get the claim.

\smallskip

(vi) We are left to prove that $\| \cH_2(x, \theta)-\cH_2(x, \theta') \|_{\mathcal{L}^2_{\theta \theta}}\le \widehat{M} \| x \|_{\ell^{\infty}}\,|\theta-\theta'|_d$. We observe that, by $\mathbf{(H2)_{C^1}}$, Lemma \ref{rmk:hderivatives} and the estimates \eqref{bound:lipthetah1}, \eqref{bound:lipthetah2} we have, for all $z\in\dom$,
{
\[
\lip_{\theta} \Big( \Upsilon_2(h(z)) \Big) \,\|\partial_{\theta} h_2(z)\|_{\mathcal{L}_{\Gamma}(\R^d)}^2\le \Big( (1+K_{\theta})^3 \lambda^{-1} \widehat{M}+(1+K_{\theta})^2 K (2+c) \wkappa_2+\cO(L) \Big) \delta.
\]
}
Then $\| \cH_2(x, \theta)-\cH_2(x, \theta') \|_{\mathcal{L}^2_{\theta\theta}}\le  \Big(\tC(1+\wrho_1+\wkappa_2)+\beta \lambda^{-1} (1+K_{\theta})^3\, \widehat{M} \Big)  \| x \|_{\ell^{\infty}}\,|\theta-\theta'|_d$. We conclude by using \eqref{condition:lambdastrong}.
\end{proof}

\bibliography{biblio2}

\def\cprime{$'$} \def\cprime{$'$}
\begin{thebibliography}{10}

\bibitem{Arnold64}
V.I. Arnold.
\newblock Instability of dynamical systems with several degrees of freedom.
\newblock {\em Sov. Math. Doklady}, 5:581--585, 1964.

\bibitem{Bambusi93}
D.~Bambusi and A.~Giorgilli.
\newblock Exponential stability of states close to resonance in
  infinite-dimensional {H}amiltonian systems.
\newblock {\em J. Statist. Phys.}, 71(3-4):569--606, 1993.

\bibitem{Bambusi06}
D.~Bambusi and A.~Ponno.
\newblock On metastability in {FPU}.
\newblock {\em Comm. Math. Phys.}, 264(2):539--561, 2006.

\bibitem{Berenguel}
R.~Berenguel.
\newblock {\em The Parametrisation Method for Invariant Manifolds of Tori in
  Skew-Product Lattices and An Entire Transcendental Family with a Persistent
  {S}iegel Disk}.
\newblock University of Barcelona, 2015.
\newblock PhD Thesis supervised by Fontich, E.

\bibitem{Berenguel19}
R.~Berenguel and E.~Fontich.
\newblock Invariant objects on lattice systems with decaying interactions.
\newblock In {\em Extended abstracts {S}pring 2018---singularly perturbed
  systems, multiscale phenomena and hysteresis: theory and applications},
  volume~11 of {\em Trends Math. Res. Perspect. CRM Barc.}, pages 137--143.
  Birkh\"{a}user/Springer, Cham, [2019] \copyright 2019.

\bibitem{Bernard08}
P.~Bernard.
\newblock The dynamics of pseudographs in convex {H}amiltonian systems.
\newblock {\em J. Amer. Math. Soc.}, 21(3):615--669, 2008.

\bibitem{Bernard16}
P.~Bernard, V.~Kaloshin, and K.~Zhang.
\newblock Arnold diffusion in arbitrary degrees of freedom and normally
  hyperbolic invariant cylinders.
\newblock {\em Acta Math.}, 217(1):1--79, 2016.

\bibitem{BertiBB03}
M.~Berti, L.~Biasco, and P.~Bolle.
\newblock Drift in phase space: a new variational mechanism with optimal
  diffusion time.
\newblock {\em J. Math. Pures Appl. (9)}, 82(6):613--664, 2003.

\bibitem{BeBo}
M.~Berti and P.~Bolle.
\newblock Fast arnold diffusion in systems with three time scales.
\newblock {\em Discrete and Continuous Dynamical Systems}, 8(3):795--811, 2002.

\bibitem{BeBo02}
M.~Berti and P.~Bolle.
\newblock A functional analysis approach to {Arnold} diffusion.
\newblock {\em Annales de l'I.H.P. Analyse non lin\'eaire}, 19(4):395--450,
  2002.

\bibitem{Bessi96}
U.~Bessi.
\newblock An approach to {A}rnol\cprime d's diffusion through the calculus of
  variations.
\newblock {\em Nonlinear Anal.}, 26(6):1115--1135, 1996.

\bibitem{Blazevski}
D.~Blazevski and R.~de~la Llave.
\newblock Localized stable manifolds for whiskered tori in coupled map lattices
  with decaying interaction.
\newblock {\em Ann. Henri Poincar\'{e}}, 15(1):29--60, 2014.

\bibitem{Bourgain00b}
J.~Bourgain.
\newblock Problems in {H}amiltonian {PDE}'s.
\newblock {\em Geom. Funct. Anal.}, Special Volume, Part I:32--56, 2000.
\newblock GAFA 2000 (Tel Aviv, 1999).

\bibitem{ChenLlave}
Q.~Chen and R.~de~la Llave.
\newblock Analytic genericity of diffusing orbits in a priori unstable
  {H}amiltonian systems.
\newblock {\em Nonlinearity}, 35(4):1986--2019, 2022.

\bibitem{ChengY04}
C.Q. Cheng and J.~Yan.
\newblock Existence of diffusion orbits in a priori unstable {H}amiltonian
  systems.
\newblock {\em J. Differential Geom.}, 67(3):457--517, 2004.

\bibitem{Cherchia95}
L.~Chierchia and P.~Perfetti.
\newblock Second order {H}amiltonian equations on {${\bf T}^\infty$} and
  almost-periodic solutions.
\newblock {\em J. Differential Equations}, 116(1):172--201, 1995.

\bibitem{CKSTT}
J.~Colliander, M.~Keel, G.~Staffilani, H.~Takaoka, and T.~Tao.
\newblock Transfer of energy to high frequencies in the cubic defocusing
  nonlinear {S}chr{\"o}dinger equation.
\newblock {\em Invent. Math.}, 181(1):39--113, 2010.

\bibitem{Cresson97}
J.~Cresson.
\newblock A {$\lambda$}-lemma for partially hyperbolic tori and the obstruction
  property.
\newblock {\em Lett. Math. Phys.}, 42(4):363--377, 1997.

\bibitem{Roeck15}
W.~De~Roeck and F.~Huveneers.
\newblock Asymptotic localization of energy in nondisordered oscillator chains.
\newblock {\em Comm. Pure Appl. Math.}, 68(9):1532--1568, 2015.

\bibitem{DelshamsLS08}
A.~Delshams, R.~de~la Llave, and T.~M. Seara.
\newblock Geometric properties of the scattering map of a normally hyperbolic
  invariant manifold.
\newblock {\em Adv. Math.}, 217(3):1096--1153, 2008.

\bibitem{DelshamsLS16}
A.~Delshams, R.~de~la Llave, and T.~M. Seara.
\newblock Instability of high dimensional {H}amiltonian systems: multiple
  resonances do not impede diffusion.
\newblock {\em Adv. Math.}, 294:689--755, 2016.

\bibitem{DelshamsLS06a}
A.~Delshams, R.~de~la Llave, and T.M. Seara.
\newblock A geometric mechanism for diffusion in hamiltonian systems overcoming
  the large gap problem: heuristics and rigorous verification on a model.
\newblock {\em Mem. Amer. Math. Soc.}, 2006.

\bibitem{DelhamsGKP10}
A.~Delshams, P.~Guti\'{e}rrez, O.~Koltsova, and J.~R. Pacha.
\newblock Transverse intersections between invariant manifolds of doubly
  hyperbolic invariant tori, via the {P}oincar\'{e}-{M}el'nikov method.
\newblock {\em Regul. Chaotic Dyn.}, 15(2-3):222--236, 2010.

\bibitem{FPU}
E.~Fermi, J.~Pasta, and S.~Ulam.
\newblock Studies of nonlinear problems. {I}. {L}os {A}lamos {S}cientific
  {L}aboratory, la-1940.
\newblock 1955.

\bibitem{FontMartin1}
E.~Fontich, R.~{de la Llave}, and P.~Mart\'in.
\newblock Dynamical systems on lattices with decaying interaction i: A
  functional analysis framework.
\newblock {\em Journal of Differential Equations}, 250(6):2838 -- 2886, 2011.

\bibitem{FontichLM11}
E.~Fontich, R.~de~la Llave, and P.~Mart\'{\i}n.
\newblock Dynamical systems on lattices with decaying interaction {II}:
  hyperbolic sets and their invariant manifolds.
\newblock {\em J. Differential Equations}, 250(6):2887--2926, 2011.

\bibitem{FontichLS15}
E.~Fontich, R.~de~la Llave, and Y.~Sire.
\newblock Construction of invariant whiskered tori by a parameterization
  method. {P}art {II}: {Q}uasi-periodic and almost periodic breathers in
  coupled map lattices.
\newblock {\em J. Differential Equations}, 259(6):2180--2279, 2015.

\bibitem{FontichM98}
E.~Fontich and P.~Mart{{\'\i}}n.
\newblock Differentiable invariant manifolds for partially hyperbolic tori and
  a lambda lemma.
\newblock {\em Nonlinearity}, 13(5):1561--1593, 2000.

\bibitem{Friesecke99}
G.~Friesecke and R.~L. Pego.
\newblock Solitary waves on {FPU} lattices. {I}. {Q}ualitative properties,
  renormalization and continuum limit.
\newblock {\em Nonlinearity}, 12(6):1601--1627, 1999.

\bibitem{FrohlichSW86}
J.~Fr\"{o}hlich, T.~Spencer, and C.~E. Wayne.
\newblock Localization in disordered, nonlinear dynamical systems.
\newblock {\em J. Statist. Phys.}, 42(3-4):247--274, 1986.

\bibitem{GaPa}
M.~Gallone and S.~Pasquali.
\newblock Metastability phenomena in two-dimensional rectangular lattices with
  nearest-neighbour interaction.
\newblock {\em Nonlinearity}, 34(7):4983--5044, 2021.

\bibitem{Geng07}
J.~Geng and Y.~Yi.
\newblock A {KAM} theorem for {H}amiltonian networks with long ranged
  couplings.
\newblock {\em Nonlinearity}, 20(6):1313--1342, 2007.

\bibitem{Geng14}
J.~Geng, J.~You, and Z.~Zhao.
\newblock Localization in one-dimensional quasi-periodic nonlinear systems.
\newblock {\em Geom. Funct. Anal.}, 24(1):116--158, 2014.

\bibitem{GideaM17}
M.~Gidea and J.~P. Marco.
\newblock Diffusion along chains of normally hyperbolic cylinders.
\newblock 2017.
\newblock Preprint available at \url{https://arxiv.org/abs/1708.08314}.

\bibitem{GiulianiGMS21}
F.~Giuliani, M.~Guardia, P.~Martin, and S.~Pasquali.
\newblock Chaotic-like transfers of energy in {H}amiltonian {PDE}s.
\newblock {\em Comm. Math. Phys.}, 384(2):1227--1290, 2021.

\bibitem{GuardiaHHMP19}
M.~Guardia, E.~Haus, Z.~Hani, A~Maspero, and M.~Procesi.
\newblock Strong nonlinear instability and growth of {S}obolev norms near
  quasiperiodic finite-gap tori for the 2{D} cubic {N}{L}{S} equation.
\newblock To appear on J. Eur. Math. Soc. (JEMS), 2020.

\bibitem{GuardiaHP16}
M.~Guardia, E.~Haus, and M.~Procesi.
\newblock Growth of {S}obolev norms for the analytic {NLS} on {$\Bbb{T}^2$}.
\newblock {\em Adv. Math.}, 301:615--692, 2016.

\bibitem{GuardiaK12}
M.~Guardia and V.~Kaloshin.
\newblock Growth of {S}obolev norms in the cubic defocusing nonlinear
  {S}chr\"{o}dinger equation.
\newblock {\em J. Eur. Math. Soc. (JEMS)}, 17(1):71--149, 2015.

\bibitem{Hani12}
Z.~Hani.
\newblock Long-time instability and unbounded {S}obolev orbits for some
  periodic nonlinear {S}chr\"odinger equations.
\newblock {\em Arch. Ration. Mech. Anal.}, 211(3):929--964, 2014.

\bibitem{HaniPTV15}
Z.~Hani, B.~Pausader, N.~Tzvetkov, and N.~Visciglia.
\newblock Modified scattering for the cubic {S}chr\"odinger equation on product
  spaces and applications.
\newblock {\em Forum Math. Pi}, 3:e4, 63, 2015.

\bibitem{HPquintic}
E.~Haus and M.~Procesi.
\newblock Growth of {S}obolev norms for the quintic {NLS} on {$T^2$}.
\newblock {\em Anal. PDE}, 8(4):883--922, 2015.

\bibitem{KappelerH08}
A.~Henrici and T.~Kappeler.
\newblock Results on normal forms for {FPU} chains.
\newblock {\em Comm. Math. Phys.}, 278(1):145--177, 2008.

\bibitem{KappelerH09}
A.~Henrici and T.~Kappeler.
\newblock Resonant normal form for even periodic {FPU} chains.
\newblock {\em J. Eur. Math. Soc. (JEMS)}, 11(5):1025--1056, 2009.

\bibitem{Hirsch1969StableMF}
M.~Hirsch and C.~Pugh.
\newblock Stable manifolds for hyperbolic sets.
\newblock {\em Bulletin of the American Mathematical Society}, 75:149--152,
  1969.

\bibitem{Huang17}
G.~Huang.
\newblock On energy transferring in a periodic pendulum lattice with analytic
  weak couplings.
\newblock {\em Ann. Henri Poincar\'{e}}, 18(6):2087--2121, 2017.

\bibitem{JiangL00}
M.~Jiang and R.~de~la Llave.
\newblock Smooth dependence of thermodynamic limits of {SRB}-measures.
\newblock {\em Comm. Math. Phys.}, 211(2):303--333, 2000.

\bibitem{KaloshinL08a}
V.~Kaloshin and M.~Levi.
\newblock An example of {A}rnold diffusion for near-integrable {H}amiltonians.
\newblock {\em Bull. Amer. Math. Soc. (N.S.)}, 45(3):409--427, 2008.

\bibitem{KaloshinL08b}
V.~Kaloshin and M.~Levi.
\newblock Geometry of {A}rnold diffusion.
\newblock {\em SIAM Rev.}, 50(4):702--720, 2008.

\bibitem{KaloshinLS14}
V.~Kaloshin, M.~Levi, and M.~Saprykina.
\newblock Arnol'd diffusion in a pendulum lattice.
\newblock {\em Comm. Pure Appl. Math.}, 67(5):748--775, 2014.

\bibitem{KaloshinZ20}
V.~Kaloshin and K.~Zhang.
\newblock {\em Arnold diffusion for smooth systems of two and a half degrees of
  freedom}, volume 208 of {\em Annals of Mathematics Studies}.
\newblock Princeton University Press, Princeton, NJ, 2020.

\bibitem{Lang95}
S.~Lang.
\newblock {\em Differential and Riemannian Manifolds}.
\newblock Springer-Verlag, Berlin, 1995.

\bibitem{Marco16}
J.~P. Marco.
\newblock Arnold diffusion for cusp-generic nearly integrable convex systems on
  $\mathbb{A}^3$.
\newblock 2016.
\newblock Preprint available at \url{https://arxiv.org/abs/1602.02403}.

\bibitem{Moser1956}
J.~Moser.
\newblock The analytic invariants of an area-preserving mapping near a
  hyperbolic fixed point.
\newblock {\em Communications on Pure and Applied Mathematics}, 9(4):673--692,
  1956.

\bibitem{Treschev04}
D.~Treschev.
\newblock Evolution of slow variables in a priori unstable hamiltonian systems.
\newblock {\em Nonlinearity}, 17(5):1803--1841, 2004.

\bibitem{Treschev12}
D.~Treschev.
\newblock Arnold diffusion far from strong resonances in multidimensional {\it
  a priori} unstable {H}amiltonian systems.
\newblock {\em Nonlinearity}, 25(9):2717--2757, 2012.

\bibitem{Wayne86b}
C.~E. Wayne.
\newblock Bounds on the trajectories of a system of weakly coupled rotators.
\newblock {\em Comm. Math. Phys.}, 104(1):21--36, 1986.

\bibitem{Wayne86a}
C.~E. Wayne.
\newblock On the elimination of nonresonance harmonics.
\newblock {\em Comm. Math. Phys.}, 103(3):351--386, 1986.

\bibitem{Wu21}
Y.~Wu and X.~Yuan.
\newblock On the {K}olmogorov theorem for some infinite-dimensional
  {H}amiltonian systems of short range.
\newblock {\em Nonlinear Anal.}, 202:Paper No. 112120, 34, 2021.

\bibitem{Yuan02}
X.~Yuan.
\newblock Construction of quasi-periodic breathers via {KAM} technique.
\newblock {\em Comm. Math. Phys.}, 226(1):61--100, 2002.

\end{thebibliography}

\bibliographystyle{plain}

\end{document}